\documentclass{amsart}

\newif\ifmpcps
\mpcpsfalse

\pdfoutput=1

\usepackage{amssymb,amsmath,stmaryrd,mathrsfs}

\def\definetac{\newif\iftac}    
\ifx\tactrue\undefined
  \definetac
  \ifx\state\undefined\tacfalse\else\tactrue\fi
\fi

\def\definebeamer{\newif\ifbeamer}
\ifx\beamertrue\undefined
  \definebeamer
  \ifx\uncover\undefined\beamerfalse\else\beamertrue\fi
\fi

\ifmpcps\else\iftac\else\usepackage{amsthm}\fi\fi
\usepackage[all,2cell]{xy}
\usepackage{tikz-cd}
\ifbeamer\else
  \ifmpcps\else\usepackage{enumitem}\fi
  \usepackage{xcolor}
  \usepackage[hyphens]{url} 
  \definecolor{darkgreen}{rgb}{0,0.45,0} 
  \definecolor{darkred}{rgb}{0.75,0,0}
  \definecolor{darkblue}{rgb}{0,0,0.6} 
  \usepackage[pagebackref,colorlinks,citecolor=darkgreen,linkcolor=darkgreen,urlcolor=darkblue]{hyperref}
\fi
\renewcommand*{\backref}[1]{}
\renewcommand*{\backrefalt}[4]{({%
    \ifcase #1 Not cited.%
          \or On p.~#2%
          \else On pp.~#2%
    \fi%
    })}
\usepackage{mathtools}          
\usepackage{graphics}           
\usepackage{ifmtarg}            
\usepackage{microtype}
\usepackage{braket}             
\let\setof\Set
\usepackage{url}                
\usepackage{xspace}             
\usepackage{aliascnt,cleveref}

\makeatletter
\let\ea\expandafter

\def\mdef#1#2{\ea\ea\ea\gdef\ea\ea\noexpand#1\ea{\ea\ensuremath\ea{#2}\xspace}}
\def\alwaysmath#1{\ea\ea\ea\global\ea\ea\ea\let\ea\ea\csname your@#1\endcsname\csname #1\endcsname
  \ea\def\csname #1\endcsname{\ensuremath{\csname your@#1\endcsname}\xspace}}

\DeclareRobustCommand\widecheck[1]{{\mathpalette\@widecheck{#1}}}
\def\@widecheck#1#2{%
    \setbox\z@\hbox{\m@th$#1#2$}%
    \setbox\tw@\hbox{\m@th$#1%
       \widehat{%
          \vrule\@width\z@\@height\ht\z@
          \vrule\@height\z@\@width\wd\z@}$}%
    \dp\tw@-\ht\z@
    \@tempdima\ht\z@ \advance\@tempdima2\ht\tw@ \divide\@tempdima\thr@@
    \setbox\tw@\hbox{%
       \raise\@tempdima\hbox{\scalebox{1}[-1]{\lower\@tempdima\box
\tw@}}}%
    {\ooalign{\box\tw@ \cr \box\z@}}}

\newcount\foreachcount

\def\foreachletter#1#2#3{\foreachcount=#1
  \ea\loop\ea\ea\ea#3\@alph\foreachcount
  \advance\foreachcount by 1
  \ifnum\foreachcount<#2\repeat}

\def\foreachLetter#1#2#3{\foreachcount=#1
  \ea\loop\ea\ea\ea#3\@Alph\foreachcount
  \advance\foreachcount by 1
  \ifnum\foreachcount<#2\repeat}

\def\definescr#1{\ea\gdef\csname s#1\endcsname{\ensuremath{\mathscr{#1}}\xspace}}
\foreachLetter{1}{27}{\definescr}
\def\definecal#1{\ea\gdef\csname c#1\endcsname{\ensuremath{\mathcal{#1}}\xspace}}
\foreachLetter{1}{27}{\definecal}
\def\definebold#1{\ea\gdef\csname b#1\endcsname{\ensuremath{\mathbf{#1}}\xspace}}
\foreachLetter{1}{27}{\definebold}
\def\definebb#1{\ea\gdef\csname d#1\endcsname{\ensuremath{\mathbb{#1}}\xspace}}
\foreachLetter{1}{27}{\definebb}
\def\definefrak#1{\ea\gdef\csname k#1\endcsname{\ensuremath{\mathfrak{#1}}\xspace}}
\foreachletter{1}{27}{\definefrak}
\foreachLetter{1}{27}{\definefrak}
\def\definesf#1{\ea\gdef\csname i#1\endcsname{\ensuremath{\mathsf{#1}}\xspace}}
\foreachletter{1}{6}{\definesf}
\foreachletter{7}{14}{\definesf}
\foreachletter{15}{27}{\definesf}
\foreachLetter{1}{27}{\definesf}
\def\definebar#1{\ea\gdef\csname #1bar\endcsname{\ensuremath{\overline{#1}}\xspace}}
\foreachLetter{1}{27}{\definebar}
\foreachletter{1}{8}{\definebar} 
\foreachletter{9}{15}{\definebar} 
\foreachletter{16}{27}{\definebar}
\def\definetil#1{\ea\gdef\csname #1til\endcsname{\ensuremath{\widetilde{#1}}\xspace}}
\foreachLetter{1}{27}{\definetil}
\foreachletter{1}{27}{\definetil}
\def\definehat#1{\ea\gdef\csname #1hat\endcsname{\ensuremath{\widehat{#1}}\xspace}}
\foreachLetter{1}{27}{\definehat}
\foreachletter{1}{27}{\definehat}
\def\definechk#1{\ea\gdef\csname #1chk\endcsname{\ensuremath{\widecheck{#1}}\xspace}}
\foreachLetter{1}{27}{\definechk}
\foreachletter{1}{27}{\definechk}
\def\defineul#1{\ea\gdef\csname u#1\endcsname{\ensuremath{\underline{#1}}\xspace}}
\foreachLetter{1}{27}{\defineul}
\foreachletter{1}{27}{\defineul}

\def\autofmt@n#1\autofmt@end{\mathrm{#1}}
\def\autofmt@b#1\autofmt@end{\mathbf{#1}}
\def\autofmt@l#1#2\autofmt@end{\mathbb{#1}\mathsf{#2}}
\def\autofmt@c#1#2\autofmt@end{\mathcal{#1}\mathit{#2}}
\def\autofmt@s#1#2\autofmt@end{\mathscr{#1}\mathit{#2}}
\def\autofmt@f#1\autofmt@end{\mathsf{#1}}
\def\autofmt@k#1\autofmt@end{\mathfrak{#1}}
\def\autofmt@u#1\autofmt@end{\underline{\smash{\mathsf{#1}}}}
\def\autofmt@U#1\autofmt@end{\underline{\underline{\smash{\mathsf{#1}}}}}
\def\autofmt@h#1\autofmt@end{\widehat{#1}}
\def\autofmt@r#1\autofmt@end{\overline{#1}}
\def\autofmt@t#1\autofmt@end{\widetilde{#1}}
\def\autofmt@k#1\autofmt@end{\check{#1}}

\def\auto@drop#1{}
\def\autodef#1{\ea\ea\ea\@autodef\ea\ea\ea#1\ea\auto@drop\string#1\autodef@end}
\def\@autodef#1#2#3\autodef@end{%
  \ea\def\ea#1\ea{\ea\ensuremath\ea{\csname autofmt@#2\endcsname#3\autofmt@end}\xspace}}
\def\autodefs@end{blarg!}
\def\autodefs#1{\@autodefs#1\autodefs@end}
\def\@autodefs#1{\ifx#1\autodefs@end%
  \def\autodefs@next{}%
  \else%
  \def\autodefs@next{\autodef#1\@autodefs}%
  \fi\autodefs@next}

\DeclareSymbolFont{bbold}{U}{bbold}{m}{n}
\DeclareSymbolFontAlphabet{\mathbbb}{bbold}

\mdef\delbar{\overline{\partial}}
\let\sm\wedge

\mdef\hf{\textstyle\frac12 }
\mdef\thrd{\textstyle\frac13 }
\mdef\qtr{\textstyle\frac14 }

\newcommand{\op}{^{\mathrm{op}}}

\let\adj\dashv
\SelectTips{cm}{}
\newdir{ >}{{}*!/-10pt/\dir{>}}    

\newcommand{\pullback}[1][dr]{\save*!/#1-1.2pc/#1:(-1,1)@^{|-}\restore}

\alwaysmath{ell}
\alwaysmath{infty}

\alwaysmath{odot}
\def\frc#1/#2.{\frac{#1}{#2}}   
\mdef\ten{\mathrel{\otimes}}

\mdef\sqten{\mathrel{\boxtimes}}

\ifmpcps\else\fi

\let\toot\rightleftarrows

\let\toto\rightrightarrows

\mdef\we{\overset{\sim}{\longrightarrow}}
\mdef\leftwe{\overset{\sim}{\longleftarrow}}

\let\fib\twoheadrightarrow

\let\xto\xrightarrow

\def\rightarrowtailfill@{\arrowfill@{\Yright\joinrel\relbar}\relbar\rightarrow}
\newcommand\xrightarrowtail[2][]{\ext@arrow 0055{\rightarrowtailfill@}{#1}{#2}}

\def\twoheadrightarrowfill@{\arrowfill@{\relbar\joinrel\relbar}\relbar\twoheadrightarrow}
\newcommand\xtwoheadrightarrow[2][]{\ext@arrow 0055{\twoheadrightarrowfill@}{#1}{#2}}

\let\xfib\xtwoheadrightarrow

\def\slashedarrowfill@#1#2#3#4#5{%
  $\m@th\thickmuskip0mu\medmuskip\thickmuskip\thinmuskip\thickmuskip
   \relax#5#1\mkern-7mu%
   \cleaders\hbox{$#5\mkern-2mu#2\mkern-2mu$}\hfill
   \mathclap{#3}\mathclap{#2}%
   \cleaders\hbox{$#5\mkern-2mu#2\mkern-2mu$}\hfill
   \mkern-7mu#4$%
}
\def\rightslashedarrowfill@{%
  \slashedarrowfill@\relbar\relbar\mapstochar\rightarrow}
\newcommand\xslashedrightarrow[2][]{%
  \ext@arrow 0055{\rightslashedarrowfill@}{#1}{#2}}
\mdef\hto{\xslashedrightarrow{}}
\mdef\htoo{\xslashedrightarrow{\quad}}

\def\toiso{\xto{\smash{\raisebox{-.5mm}{$\scriptstyle\sim$}}}}

\long\def\my@drawfill#1#2;{%
\@skipfalse
\fill[#1,draw=none] #2;
\@skiptrue
\draw[#1,fill=none] #2;
}
\newif\if@skip
\newcommand{\skipit}[1]{\if@skip\else#1\fi}
\newcommand{\drawfill}[1][]{\my@drawfill{#1}}

\def\thmqedhere{\expandafter\csname\csname @currenvir\endcsname @qed\endcsname}

\ifbeamer\else
  \renewcommand{\theenumi}{(\roman{enumi})}
  
\fi

\ifmpcps\else
\ifbeamer\else		
  \setitemize[1]{leftmargin=2em}		
  \setenumerate[1]{leftmargin=*}		
\fi		
\fi

\alwaysmath{alpha}
\alwaysmath{beta}
\alwaysmath{gamma}
\alwaysmath{Gamma}
\alwaysmath{delta}
\alwaysmath{Delta}
\alwaysmath{epsilon}
\mdef\ep{\varepsilon}
\alwaysmath{zeta}
\alwaysmath{eta}
\alwaysmath{theta}
\alwaysmath{Theta}
\alwaysmath{iota}
\alwaysmath{kappa}
\alwaysmath{lambda}
\alwaysmath{Lambda}
\alwaysmath{mu}
\alwaysmath{nu}
\alwaysmath{xi}
\alwaysmath{pi}
\alwaysmath{rho}
\alwaysmath{sigma}
\alwaysmath{Sigma}
\alwaysmath{tau}
\alwaysmath{upsilon}
\alwaysmath{Upsilon}
\alwaysmath{phi}
\alwaysmath{Pi}
\alwaysmath{Phi}
\mdef\ph{\varphi}
\alwaysmath{chi}
\alwaysmath{psi}
\alwaysmath{Psi}
\alwaysmath{omega}
\alwaysmath{Omega}
\let\al\alpha
\let\be\beta

\let\Gm\Gamma

\let\De\Delta
\let\si\sigma

\let\Th\Theta

\ifmpcps
  \def\qed{\unskip\hskip 1pc\nolinebreak\proofbox}
  \def\qedhere{\unskip\hskip 1pc\nolinebreak\proofbox}
\fi

\makeatother

\usepackage[mode=multiuser,layout=margin,status=final]{fixme}
\FXRegisterAuthor{ms}{ams}{MS}  
\FXRegisterAuthor{pl}{apl}{PLL} 
\fxusetheme{color}

\ifmpcps\else
\theoremstyle{plain}		
		
\theoremstyle{definition}		
		
\fi

\autodefs{\fap\fcomp\fzero\fsucc\fpair\ftt\fbase\floop\fone\fseg\fnorth\fsouth\fmerid\fsurf\fext\fdata\fwhere\fId\fEq\fproj\fsquash\frefl\fType\fcxt\fpretype\ftype\fVec\fnil\fcons\fdim\fapeq\fcoeq\fapp\fsplit\ftt\fExt\fpush\fExtn\ffst\fsnd\finl\finr}

\let\C\cC
\let\D\cD
\let\P\cP
\let\T\cT
\let\refl\frefl

\def\J{\ensuremath{\mathsf{J}}}

\usepackage{xifthen}
\makeatletter
\def\indeff#1#2#3{
  \begin{quote}
    \noindent \fdata ${#1} : {#2}$ \fwhere
    \@indef #3 \OR\OR
  \end{quote}
}

\def\@indef#1\OR{\ifthenelse{\isempty{#1}}{}{\\\hspace*{3mm} $#1$ \@indef}}

\usepackage{mathpartir}

\newcommand{\types}{\vdash}

\renewcommand{\id}[3][]{\fId_{#1}(#2,#3)}
\newcommand{\idover}[4][]{\fId_{#1}(#2,#3)_{#4}}

\let\ap\fap

\newcommand{\ivl}{{\Delta^1}}
\newcommand{\trunc}[2]{\mathopen{}\left\Vert #2\right\Vert_{#1}\mathclose{}}

\newcommand{\brck}[1]{\trunc{}{#1}}

\newcommand{\jdeq}{\equiv}

\newcommand{\blank}{\mathord{\hspace{1pt}\text{--}\hspace{1pt}}}

\newcommand{\opp}[1]{\mathord{{#1}^{-1}}}

\newcommand{\transf}[1]{\ensuremath{{#1}_{*}}\xspace} 

\newcommand{\ct}{%
  \mathchoice{\mathbin{\raisebox{0.5ex}{$\displaystyle\centerdot$}}}%
             {\mathbin{\raisebox{0.5ex}{$\centerdot$}}}%
             {\mathbin{\raisebox{0.25ex}{$\scriptstyle\,\centerdot\,$}}}%
             {\mathbin{\raisebox{0.1ex}{$\scriptscriptstyle\,\centerdot\,$}}}
}

\makeatletter
\def\prd#1{\@ifnextchar\bgroup{\prd@parens{#1}}{%
    \@ifnextchar\sm{\prd@parens{#1}\@eatsm}{%
    \@ifnextchar\prd{\prd@parens{#1}\@eatprd}{%
    \@ifnextchar\;{\prd@parens{#1}\@eatsemicolonspace}{%
    \@ifnextchar\\{\prd@parens{#1}\@eatlinebreak}{%
    \@ifnextchar\narrowbreak{\prd@parens{#1}\@eatnarrowbreak}{%
      \prd@noparens{#1}}}}}}}}
\def\prd@parens#1{\@ifnextchar\bgroup%
  {\mathchoice{\@dprd{#1}}{\@tprd{#1}}{\@tprd{#1}}{\@tprd{#1}}\prd@parens}%
  {\@ifnextchar\sm%
    {\mathchoice{\@dprd{#1}}{\@tprd{#1}}{\@tprd{#1}}{\@tprd{#1}}\@eatsm}%
    {\mathchoice{\@dprd{#1}}{\@tprd{#1}}{\@tprd{#1}}{\@tprd{#1}}}}}
\def\@eatsm\sm{\sm@parens}
\def\prd@noparens#1{\mathchoice{\@dprd@noparens{#1}}{\@tprd{#1}}{\@tprd{#1}}{\@tprd{#1}}}
\def\lprd#1{\@ifnextchar\bgroup{\@lprd{#1}\lprd}{\@@lprd{#1}}}
\def\@lprd#1{\mathchoice{{\textstyle\prod}}{\prod}{\prod}{\prod}({\textstyle #1})\;}
\def\@@lprd#1{\mathchoice{{\textstyle\prod}}{\prod}{\prod}{\prod}({\textstyle #1}),\ }
\def\tprd#1{\@tprd{#1}\@ifnextchar\bgroup{\tprd}{}}
\def\@tprd#1{\mathchoice{{\textstyle\prod_{(#1)}}}{\prod_{(#1)}}{\prod_{(#1)}}{\prod_{(#1)}}}
\def\dprd#1{\@dprd{#1}\@ifnextchar\bgroup{\dprd}{}}
\def\@dprd#1{\prod_{(#1)}\,}
\def\@dprd@noparens#1{\prod_{#1}\,}

\def\@eatnarrowbreak\narrowbreak{%
  \@ifnextchar\prd{\narrowbreak\@eatprd}{%
    \@ifnextchar\sm{\narrowbreak\@eatsm}{%
      \narrowbreak}}}
\def\@eatlinebreak\\{%
  \@ifnextchar\prd{\\\@eatprd}{%
    \@ifnextchar\sm{\\\@eatsm}{%
      \\}}}
\def\@eatsemicolonspace\;{%
  \@ifnextchar\prd{\;\@eatprd}{%
    \@ifnextchar\sm{\;\@eatsm}{%
      \;}}}

\def\lam#1{{\lambda}\@lamarg#1:\@endlamarg\@ifnextchar\bgroup{.\,\lam}{.\,}}
\def\@lamarg#1:#2\@endlamarg{\if\relax\detokenize{#2}\relax #1\else\@lamvar{\@lameatcolon#2},#1\@endlamvar\fi}
\def\@lamvar#1,#2\@endlamvar{(#2\,{:}\,#1)}
\def\@lameatcolon#1:{#1}

\def\lamu#1{{\lambda}\@lamuarg#1:\@endlamuarg\@ifnextchar\bgroup{.\,\lamu}{.\,}}
\def\@lamuarg#1:#2\@endlamuarg{#1}

\def\sm#1{\@ifnextchar\bgroup{\sm@parens{#1}}{%
    \@ifnextchar\prd{\sm@parens{#1}\@eatprd}{%
    \@ifnextchar\sm{\sm@parens{#1}\@eatsm}{%
    \@ifnextchar\;{\sm@parens{#1}\@eatsemicolonspace}{%
    \@ifnextchar\\{\sm@parens{#1}\@eatlinebreak}{%
    \@ifnextchar\narrowbreak{\sm@parens{#1}\@eatnarrowbreak}{%
        \sm@noparens{#1}}}}}}}}
\def\sm@parens#1{\@ifnextchar\bgroup%
  {\mathchoice{\@dsm{#1}}{\@tsm{#1}}{\@tsm{#1}}{\@tsm{#1}}\sm@parens}%
  {\@ifnextchar\prd%
    {\mathchoice{\@dsm{#1}}{\@tsm{#1}}{\@tsm{#1}}{\@tsm{#1}}\@eatprd}%
    {\mathchoice{\@dsm{#1}}{\@tsm{#1}}{\@tsm{#1}}{\@tsm{#1}}}}}
\def\@eatprd\prd{\prd@parens}
\def\sm@noparens#1{\mathchoice{\@dsm@noparens{#1}}{\@tsm{#1}}{\@tsm{#1}}{\@tsm{#1}}}
\def\lsm#1{\@ifnextchar\bgroup{\@lsm{#1}\lsm}{\@@lsm{#1}}}
\def\@lsm#1{\mathchoice{{\textstyle\sum}}{\sum}{\sum}{\sum}({\textstyle #1})\;}
\def\@@lsm#1{\mathchoice{{\textstyle\sum}}{\sum}{\sum}{\sum}({\textstyle #1}),\ }
\def\tsm#1{\@tsm{#1}\@ifnextchar\bgroup{\tsm}{}}
\def\@tsm#1{\mathchoice{{\textstyle\sum_{(#1)}}}{\sum_{(#1)}}{\sum_{(#1)}}{\sum_{(#1)}}}
\def\dsm#1{\@dsm{#1}\@ifnextchar\bgroup{\dsm}{}}
\def\@dsm#1{\sum_{(#1)}\,}
\def\@dsm@noparens#1{\sum_{#1}\,}

\makeatletter
\def\defthm#1#2#3{%
  \newaliascnt{#1}{thm}
  \newtheorem{#1}[#1]{#2}
  \aliascntresetthe{#1}
  \crefname{#1}{#2}{#3}
}
\ifmpcps
\def\defrmk#1#2#3{%
  \newaliascnt{#1}{thm}
  \newnumbered{#1}[#1]{#2}
  \aliascntresetthe{#1}
  \crefname{#1}{#2}{#3}
}\fi

\newtheorem{thm}{Theorem}[section]
\crefname{thm}{Theorem}{Theorems}
\defthm{cor}{Corollary}{Corollaries}
\defthm{lem}{Lemma}{Lemmas}
\defthm{axiom}{Axiom}{Axioms}
\defthm{sch}{Scholium}{Scholia}
\defthm{prob}{Problem}{Problems}
\ifmpcps
  \newenvironment{constr}{\begin{proof}[Construction.]}{\end{proof}}
  \defrmk{defn}{Definition}{Definitions}
  \defrmk{rmk}{Remark}{Remarks}
  \defrmk{eg}{Example}{Examples}
  \defrmk{noneg}{Non-example}{Non-examples}
  \defrmk{egs}{Examples}{Examples}
  \defrmk{notes}{Notes}{Notes}
\else
  \newenvironment{constr}{\begin{proof}[Construction]}{\end{proof}}
  \theoremstyle{definition}
  \defthm{defn}{Definition}{Definitions}
  \theoremstyle{remark}
  \defthm{rmk}{Remark}{Remarks}
  \defthm{eg}{Example}{Examples}
  \defthm{noneg}{Non-example}{Non-examples}
  \defthm{egs}{Examples}{Examples}
  \defthm{notes}{Notes}{Notes}
\fi
\let\c@equation\c@thm
\numberwithin{equation}{section}

\autodefs{\fPush\sFib\cCat\cAlg\cFam\cInst\cMnd}
\def\name#1{\ulcorner #1\urcorner}
\let\J\iJ
\let\W\iW
\let\F\cF
\let\N\iN
\let\sec\S
\let\S\cS

\let\P\cP
\def\emptyps{\diamond}
\def\typeps#1#2{\langle #1,#2\rangle}
\def\termps#1#2#3{\llbracket #1,#2,#3\rrbracket}

\def\alg{\text{-}\cAlg}
\def\algf{\alg_{\mathbf{f}}}
\mdef\dId{\mathbb{I}\mathsf{d}}
\def\dtprop{\dT_{\mathsf{Prop}}}
\mdef\Mf{\sM_{\mathbf{f}}}
\mdef\fibm{(\sM,\cF)}
\mdef\fibmf{(\Mf,\cF_{\mathbf{f}})}
\mdef\fibmbang{(\sM,\cF_!)}
\mdef\fibmfbang{(\Mf,\cF_{\mathbf{f},!})}
\let\C\cC
\let\T\cT
\let\r\ir
\def\inl{\mathsf{inl}}
\def\inr{\mathsf{inr}}
\def\zero{\mathsf{zero}}
\def\succ{\mathsf{succ}}
\def\nrec{\mathsf{nrec}}

\def\sup{\mathsf{sup}}
\def\fold{\mathsf{fold}}
\def\wrec{\mathsf{wrec}}
\def\tr{\mathsf{tr}}
\def\treq{\mathsf{treq}}
\def\trrec{\mathsf{trrec}}
\autodefs{\fId}
\let\Id\fId
\let\G\cG
\def\dG{\widehat{\G}}
\let\type\fibtype
\def\drefl{\refl'}
\usepackage{scalefnt}

\renewcommand{\idover}[4][]{\fId_{#1}^{#4}(#2,#3)}
\newcommand{\Idtwo}[2]{\widetilde{\Id^{#1}_{#2}}}
\newcommand{\rtwo}[2]{\widetilde{\r^{#1}_{#2}}}
\newcommand{\idovertwo}[4][]{\widetilde{\fId^{#4}_{#1}}(#2,#3)}
\newcommand{\aptwo}[1]{\widetilde{\ap_{#1}}}
\usepackage{mathtools} 
\usepackage[utf8]{inputenc}
\tikzset{idmap/.style={double equal sign distance,-}}
\usetikzlibrary{decorations.pathmorphing}
\def\pullback#1{\ar[#1,phantom,near start,"\lrcorner"]}
\tikzset{commutative diagrams/cof/.style={tail}}
\tikzset{commutative diagrams/fib/.style={two heads}}
\tikzset{commutative diagrams/acyc/.style={"\sim" {sloped,above}}}
\tikzset{commutative diagrams/acyc'/.style={"\sim" {sloped,below}}}

\crefformat{section}{\sec#2#1#3}
\Crefformat{section}{Section~#2#1#3}
\crefrangeformat{section}{\sec\sec#3#1#4--#5#2#6}
\Crefrangeformat{section}{Sections~#3#1#4--#5#2#6}
\crefmultiformat{section}{\sec\sec#2#1#3}{ and~#2#1#3}{, #2#1#3}{ and~#2#1#3}
\Crefmultiformat{section}{Sections~#2#1#3}{ and~#2#1#3}{, #2#1#3}{ and~#2#1#3}
\crefrangemultiformat{section}{\sec\sec#3#1#4--#5#2#6}{ and~#3#1#4--#5#2#6}{, #3#1#4--#5#2#6}{ and~#3#1#4--#5#2#6}
\Crefrangemultiformat{section}{Sections~#3#1#4--#5#2#6}{ and~#3#1#4--#5#2#6}{, #3#1#4--#5#2#6}{ and~#3#1#4--#5#2#6}

\title{Semantics of higher inductive types}

\newcommand{\thankstext}{This material is based on research sponsored by The United States Air Force Research Laboratory under agreement number FA9550-15-1-0053.  The U.S. Government is authorized to reproduce and distribute reprints for Governmental purposes notwithstanding any copyright notation thereon.  The views and conclusions contained herein are those of the author and should not be interpreted as necessarily representing the official policies or endorsements, either expressed or implied, of the United States Air Force Research Laboratory, the U.S. Government, or Carnegie Mellon University.
  The research presented here was also partly funded by the Swedish Research Council (VR) Grant 2015-03835 \emph{Constructive and category-theoretic foundations of mathematics}.}
\ifmpcps
  \author[Peter LeFanu Lumsdaine and Michael Shulman]{PETER LEFANU LUMSDAINE \and\ MICHAEL SHULMAN\thanks{\thankstext}}
\else
  \author{Peter LeFanu Lumsdaine}
  \author{Michael Shulman}
  \thanks{\thankstext}
\fi

\begin{document}

\maketitle

\ifmpcps
\else
 \vspace{-1.7\baselineskip} 
\fi

\begin{abstract}
  \emph{Higher inductive types} are a class of type-forming rules, introduced to provide basic (and not-so-basic) homotopy-theoretic constructions in a type-theoretic style.
  They have proven very fruitful for the ``synthetic'' development of homotopy theory within type theory, as well as in formalizing ordinary set-level mathematics in type theory.
  In this article, we construct models of a wide range of higher inductive types in a fairly wide range of settings.
  
  We introduce the notion of \emph{cell monad with parameters}: a semantically-defined scheme for specifying homotopically well-behaved notions of structure.
  We then show that any suitable model category has \emph{weakly stable typal initial algebras} for any cell monad with parameters.
  When combined with the local universes construction to obtain strict stability, this specializes to give models of specific higher inductive types, including spheres, the torus, pushout types, truncations, the James construction, and general localisations.

  Our results apply in any sufficiently nice Quillen model category, including any right proper, simplicially locally cartesian closed, simplicial Cisinski model category (such as simplicial sets) and any locally presentable locally cartesian closed category (such as sets) with its trivial model structure.
  In particular, any locally presentable locally cartesian closed $(\infty,1)$-category is presented by some model category to which our results apply.
\end{abstract}

\tableofcontents

\section{Introduction}
\label{sec:introduction}

Higher inductive types are a recent innovation in dependent type theory.
They were originally motivated by the homotopical interpretations of type theory~\cite{aw:htpy-idtype,klv:ssetmodel,hottbook}, as a way to construct types corresponding to homotopy-theoretic cell complexes such as spheres, tori, and so on.

For instance, a circle can be constructed as a topological cell complex with one point (0-cell) and one path (1-cell) with both endpoints glued to the point.
The homotopy-type-theoretic view of paths as corresponding to elements of identity types suggests considering a ``circle type'' $S^1$ that is ``inductively generated'' by a point $\mathsf{base}:S^1$ and an identity $\mathsf{loop}: \id[S^1]{\mathsf{base}}{\mathsf{base}}$.
Similarly, a torus can be constructed as a cell complex with one point, two paths, and one disc (2-cell) whose boundary traverses the paths in one order on one side and the other order on the other side; this suggests a ``torus type'' $T^2$ that is inductively generated by one point $\mathsf{base}:T^2$, two identities $\mathsf{left},\mathsf{right}:\id[T^2]{\mathsf{base}}{\mathsf{base}}$, and an ``iterated identity'' $\mathsf{sq} : \id[{\id[T^2]{\mathsf{base}}{\mathsf{base}}}]{\mathsf{left}\ct\mathsf{right}}{\mathsf{right}\ct\mathsf{left}}$ (where $\ct$ denotes the concatenation of identities).

Although it may seem odd to consider ``inductive constructors'' that take values not in an inductive type itself but in its identity types, it is not hard to write down rules for such types in the usual dependent-type-theory style of formation, introduction, elimination, and computation principles.
For instance, in the case of $S^1$, the terms $\mathsf{base}$ and $\mathsf{loop}$ are the introduction rules, while the elimination rule says that to define a section $x:S^1 \types f(x) : P(x)$ of a type family over $S^1$ it suffices to give a point $f(\mathsf{base}):P(\mathsf{base})$ and a lift of the loop $\ap_{f}(\mathsf{loop}) : \idover[P]{f(\mathsf{base})}{f(\mathsf{base})}{\mathsf{loop}}$ (see \cref{sec:coherence-pushouts} for the meaning of the notation), and the computation rules say that $f$ evaluates on $\mathsf{base}$ and $\mathsf{loop}$ to the given data.

The above examples are ``non-recursive'' higher inductive types, analogous to binary sum types and the empty type, considered as ordinary inductive types.
The ``most general'' non-recursive higher inductive type, in the sense that all others can be constructed from it, is a \emph{(homotopy) pushout}: given two functions $f_1:A\to B_1$ and $f_2:A\to B_2$, their pushout is a type $D$ inductively generated by maps $\nu_1:B_1\to D$ and $\nu_2:B_2\to D$ and a homotopy $\mu:\prod_{x:A} \id[D]{\nu_1(f_1(x))}{\nu_2(f_2(x))}$.
In particular, from pushouts we can construct all other finite homotopy colimits (and also some infinite ones, using the natural numbers type).

We can also consider higher inductive types that do have recursive constructors.
For instance, in homotopy type theory a \emph{proposition} is defined to be a type any two elements of which are equal, i.e.\ such that there is a term in $\prod_{x,y:A} \id[A]xy$.
The \emph{propositional truncation} of an arbitrary type $A$ is a proposition $\brck{A}$ with a map from $A$ that is ``universal''; it can be regarded as the higher inductive type generated by one ordinary constructor $\tr:A\to\brck{A}$ and one higher constructor $\treq : \prod_{x,y:\brck{A}}\id[\brck{A}]{x}{y}$.
(The latter is ``recursive'' because the inputs $x,y$ belong to the type $\brck{A}$ being constructed.)
Other recursive higher inductive types include the \emph{$n$-truncation} for any $n$ (a universal map into a homotopy $n$-type --- the propositional truncation is the case $n=-1$) and the \emph{localization} at a family of morphisms $\{f_i\}_{i\in I}$ (the universal map into a type that ``sees each $f_i$ as an isomorphism'').

In particular, the $0$-truncation of a pushout of $0$-types (types satisfying Uniqueness of Identity Proofs, a.k.a.\ ``sets'') is a pushout in the category of $0$-types.
If we also assume Voevodsky's univalence axiom (or just its restriction to propositions), we can show that the latter category is a pretopos; see~\cite[Chapter 10]{hottbook}.
In particular, one has exact set-quotients of equivalence relations, eliminating the need for the cumbersome method of ``setoids''.
Recursive higher inductive $0$-types (also known as ``quotient inductive types''~\cite{acdf:qiits}) have other applications, such as constructing free algebras for arbitrary algebraic theories (which is not possible in a general elementary topos \cite{blass:freealg}) and representing type theory inside itself~\cite{ak:tt-qit}.
Thus, higher inductive types add power even for the representation of ordinary set-based mathematics inside type theory.

A general discussion of the syntax of higher inductive types (though not going so far as to propose a general definition) can be found in~\cite[Chapter 6]{hottbook}.
They have succeeded admirably at their original goal of providing types in homotopy type theory that behave like cell complexes in classical homotopy theory (see~\cite[Chapter 8]{hottbook} and also more recent work).
Interestingly, although the original motivation came from inspecting small concrete cell complexes such as circles and tori, the more complicated recursive higher inductive types (such as $n$-truncation and localization) also correspond to constructions in classical homotopy theory (such as Postnikov towers and Bousfield localization) that are usually performed using cell complexes, albeit transfinite and often unwieldy ones.
This yields a strong intuition that higher inductive types are somehow a \emph{type-theoretic replacement for cell complexes}.
However, there is some distance from this intuition to a formal proof that higher inductive types actually \emph{exist} in homotopical models of type theory.

The goal of the present paper is to fill this gap.
The basic ideas were known to the authors several years ago, and circulated informally in the community, but the results have not been written out precisely until now.

Roughly speaking, there are the following issues to overcome in constructing semantics for higher inductive types.
\begin{enumerate}
\item When interpreting type theory in a model category, types are represented by \emph{fibrant} objects.
  But cell complexes are built out of colimits, which do not preserve fibrancy; thus we need to ``fibrantly replace'' them.\label{item:p1}
\item The usual constructions of homotopy theory, including cell complexes and fibrant replacement, produce the right answer only up to homotopy equivalence.
  But the induction principle of a higher inductive type is somewhat stricter than this (though it falls short of determining it up to strict isomorphism).\label{item:p2}
\item Everything in type theory is stable under substitution, so in order to be type-theoretically meaningful a category-theoretic construction must be stable under pullback (at least ``weakly stable'', as in~\cite{lw:localuniv}).
  But fibrant replacement is not in general stable under pullback, even weakly.\label{item:p3}
\item The constructors of an ordinary inductive type are unrelated to each other.
  But in a higher inductive type each constructor must be able to refer to the previous ones; specifically, the source and target of a path-constructor generally involve the previous point-constructors.
  No syntactic scheme has yet been proposed for this that covers all cases of interest while remaining meaningful and consistent, so the ``target'' of an interpretation theorem is not even precisely defined.\label{item:p4}
\end{enumerate}
As we will see, problem~\ref{item:p2} only arises in the recursive case; but on the other hand, it is not specific to \emph{higher} inductive types, being a problem even for interpreting ordinary recursive inductive types such as natural numbers and \W-types.

We present (in \cref{sec:cell-monads-param}) a general context that solves problems~\ref{item:p1}--\ref{item:p3}, and partially solves~\ref{item:p4} in the following sense.
We describe a \emph{semantic} construction (a ``cell monad'') that enables subsequent ``constructors'' to refer to previous ones, and determines a precise \emph{semantic} meaning for what the source and target of a path-constructor can be.
However, although in particular cases it is easy to see that the \emph{syntactic} source and targets in the specification of a higher inductive type give rise to the desired semantic structure, we do not give a general syntactic characterization of allowable ``source and target terms''.\footnote{In~\cite{acdf:qiits} a similar semantic characterization is internalized literally, using universes, under the assumption of UIP; but with present technology this does not generalize to type theory without UIP, since $(\infty,1)$-categories are not known to be definable internally.}
Moreover, there are some kinds of source and target terms that seem natural syntactically but do not fit into our semantic framework.

For expositional clarity, instead of jumping right to the general context of cell monads, we begin (after preliminaries in \cref{sec:model-categ-type} and a warm-up in \cref{sec:coproducts}) in \crefrange{sec:pushouts}{sec:pushouts-type-theory} by considering in detail the case of pushouts.
These are subject only to problems~\ref{item:p1} and~\ref{item:p3}.
We solve~\ref{item:p1} by a simple ordinary fibrant replacement, and \ref{item:p3} by pulling back separately along fibrations (which are easy) and acyclic cofibrations (using a pushout-product argument).

In \cref{sec:natural-numbers} we consider the natural numbers type, which is of course simpler because it lacks any path-constructors, but is recursive and thus subject to problems~\ref{item:p1}--\ref{item:p3}.
We solve problem~\ref{item:p2} by using an \emph{algebraic fibrant replacement}, which retains better control over the point-set-level behavior than an ordinary fibrant replacement, in particular enabling us to prove the type-theoretic universal property.
However, the simple pushout-product argument for~\ref{item:p3} that applies in the non-recursive case is no longer applicable in the recursive case.
Thus, to prove weak stability for natural numbers, we alter the usual interpretation of type theory a little by requiring \emph{contexts} as well as types to be fibrant (this turns out to be true anyway \emph{a posteriori} in the usual interpretation, it's just not usually assumed explicitly at the outset).
This allows us to ignore the case of pullback along acyclic cofibrations.

In \cref{sec:w-types} we consider \W-types, which add only one wrinkle to natural numbers: since they have nontrivial parameters, we have to be able to pull back along acyclic cofibrations in addition to along fibrations.
But requiring contexts to be fibrant still solves the problem, with some extra work: it ensures that acyclic cofibrations have retractions, which together with an ``extension lemma'' enables us to reduce an elimination problem over the domain to one over the codomain.

These specific cases, while admittedly rather specific, are by themselves sufficient for a large fraction of current homotopy type theory and synthetic homotopy theory.
In particular, Kraus~\cite{kraus:nonrec-hit} and van Doorn~\cite{vandoorn:proptrunc-nonrec} have shown that from pushouts and the natural numbers one can construct the propositional truncation, while Rijke~\cite{rijke:join,rijke:omegaloc-talk} has shown that one can even construct all $n$-truncations along with some localizations, and some have even conjectured that all higher inductive types might be similarly constructible.
To motivate the need for further work, therefore, in \cref{sec:blass} we describe a specific higher inductive type that definitely \emph{cannot} be constructed in terms of pushouts and natural numbers.

As one final warm-up for the general case, in \cref{sec:prop-trunc} we discuss the propositional truncation.
This is special in many ways, but still exhibits the one remaining basic new idea in our general approach, namely \emph{algebraic pushouts of monads}.
A sequence of such pushouts along ``generating monad cells'' yields what we call a \emph{cell monad}, since it is a precise analogue of a cell complex, only built \emph{in the category of monads}.
Each generating monad cell corresponds to one constructor of a higher inductive type; thus we might say that higher inductive types replace the large unwieldy cell complexes of classical algebraic topology with small manageable ``cell complexes'' constructed at a higher level of abstraction.

In \cref{sec:cell-monads} we describe cell monads for higher inductive types without parameters, and in \cref{sec:cell-monads-param} we generalize to the parametrized case.
Most of the work in these sections is just in setting up a sufficiently general context in which to prove the theorems.
As mentioned above, we do not have a syntactic context of equal generality; but we explain heuristically, with many examples, how to perform a syntax-to-semantics translation in particular cases.
We also discuss some limitations of our construction; most importantly, the fact that source and target terms cannot contain ``fibrant structure'' such as eliminations from other inductive and higher inductive types.
This includes concatenation of paths, as used for instance in the standard presentaiton of the torus; however, for this and many other examples, we show how to work around the issue by using more general cofibrations.
In \cref{sec:summary} we summarize our results for quick reference.

\begin{rmk}
  Since type-theoretic laws hold up to strict equality, whereas categorical constructions are well-behaved only up to isomorphism (or homotopy), the standard way to interpret dependent type theory into categories involves an intermediate ``strictification'' step.
  For this we will use the \emph{local universes coherence method}~\cite{lw:localuniv}, which replaces any well-behaved comprehension category by a split one in such a way that any ``weakly stable'' structure possessed by the original becomes ``strictly stable'' structure in the splitting.
  In \cref{sec:coproducts,sec:w-types} we use results from~\cite{lw:localuniv} verbatim, while elsewhere we have to prove the analogous theorems ourselves; this is generally straightforward, although it is worth noting that the ``cofibrancy'' of a cell monad plays a role not only in the weak stability of its homotopy-initial algebras but also the ``representability'' of its structures, both of which are necessary for the local universes theorem.

  Once such a strict algebraic structure has been built from a desired categorical model, the interpretation is completed by constructing the \emph{initial} such algebraic structure out of the syntax of type theory, so that there is a unique structured map interpreting all the syntax into the desired model.
  Such an initiality theorem was suggested by~\cite{cartmell:gatcc}, and proven thoroughly by~\cite{streicher:semtt} for a specific non-toy type theory, the Calculus of Constructions.
  Such proofs are too long and bureaucratic for anyone to want to write or read in detail, but opinions differ as to whether they are straightforward enough that initiality for other type theories (such as those considered here) may be considered established.
  We take no position on this question; our only goal is to construct split comprehension categories with the appropriate strictly stable structure to model higher inductive types, and so we will simply refer to these as ``models of type theory''.
\end{rmk}

Our perspective on higher inductive types can be summarized as follows.
It is well-known that ordinary inductive types can be described as initial algebras for polynomial endofunctors (and this remains true in homotopy type theory~\cite{ags:it-hott}).
An algebra for an endofunctor $F$ (in the endofunctor sense, i.e.\ an object $X$ with a map $F X \to X$) is the same as an algebra (in the monad sense) for the \emph{algebraically-free monad} generated by $F$; thus ordinary inductive types are equivalently initial algebras for \emph{free monads on polynomial endofunctors}.
Since (unlike an endofunctor) a monad always has an initial algebra, the existence of inductive types is really about the existence of such free monads themselves.

Higher inductive types generalize from \emph{free} monads to \emph{presented} monads, i.e.\ monads constructed as an \emph{algebraic colimit} of free monads on polynomial endofunctors.
(A colimit of monads is ``algebraic'' if an algebra structure for the colimit monad is determined by compatible algebra structures for the monads in the input diagram.)
Intuitively, a ``cell monad'' is one with such a presentation that is ``homotopically meaningful'', i.e.\ a ``homotopy colimit of monads''.
This suggests a way to say what it means for a category (or $(\infty,1)$-category) to ``have higher inductive types'': that algebraic colimits of algebraically-free monads on polynomial endofunctors exist (with all these words interpreted 1-categorically or $(\infty,1)$-categorically as appropriate).
This is known to be true for locally presentable 1-categories, and ought to be true for locally presentable $(\infty,1)$-categories as well; but it is not immediately obvious how to make it precise in an ``elementary'' way.

While our general construction achieves a lot, there is also still a lot left to do.
Open problems include the following.
\begin{itemize}
\item Our type theory does not contain any universes.
  In some respects this is a feature rather than a bug, since it makes the semantics more general: our model categories include all locally presentable locally cartesian closed 1-categories and suffice to present all locally presentable locally cartesian closed $(\infty,1)$-categories.
  Of course, if a model category does have universes, such as simplicial sets~\cite{klv:ssetmodel} and others constructed from it~\cite{shulman:invdia,shulman:elreedy,cisinski:elegant,shulman:eiuniv}, then higher inductive types are more powerful due to ``large eliminations'', which are used frequently in synthetic homotopy theory.

  However, there is an additional desirable feature that our method does not produce even in models that have universes.
  One hopes to have universes that are \emph{closed under} parametrized higher inductive types, for instance that if $A,B_1,B_2$ all lie in some universe then so does some pushout of any $f_1:A\to B_1$ and $f_2:A\to B_2$.
  But our method does not give this, because it involves fibrant replacement, which does not preserve smallness of fibers.
  (The different ``cubical'' methods of~\cite{chm:cubical-hits} do appear to produce universes closed under higher inductive types.)
\item As mentioned above, although our method allows very general ``source and target terms'', it does not allow ``fibrant structure'' therein, such as path-concatenation and eliminations from inductive types.
  While we can work around this in many cases, it would be better to be able to drop the restriction.
\item Of course, we would like a general syntactic scheme for higher inductive types, and a general theorem interpreting such syntax using our general semantics, rather than having to treat every case manually.
\item Finally, there are various generalizations of higher inductive types that we have not addressed, such as indexed higher inductive families, higher inductive-inductive types (e.g.~\cite[Chapter 11]{hottbook} and~\cite{ak:tt-qit}), and higher inductive-recursive types (e.g.~\cite{shulman:hiru-tdd}).
  A starting point would be to unify our approach with that of~\cite{acdf:qiits}.
\end{itemize}

\section{Good model categories}
\label{sec:model-categ-type}

We will work throughout in the following contexts.

\begin{defn}
  A \textbf{good model category} is a model category \sM with the following additional properties.
  \begin{enumerate}
  \item \sM is simplicial.\label{item:m0}
  \item Every monomorphism in \sM is a cofibration.\label{item:m1}
  \item Cofibrations in \sM are stable under arbitrary limits.\label{item:m1a}
  \item \sM is right proper, i.e.\ weak equivalences are preserved by pullback along fibrations.\label{item:m2}
  \item \sM is locally cartesian closed, as a simplicial category.\label{item:m3}
  \end{enumerate}
  An \textbf{excellent model category\footnote{Note that this nonce definition is different from the similarly-named~\cite[Definition A.3.2.16]{lurie:higher-topoi}, though it shares some of the same conditions.}} is a good model category that is in addition combinatorial, i.e.\ it satisfies the following further properties.
  \begin{enumerate}[resume]\ifmpcps\setcounter{enumi}{5}\fi
  \item \sM is cofibrantly generated.
  \item As a category, \sM is locally presentable.
  \end{enumerate}
\end{defn}

\begin{eg}
A \emph{Cisinski model category}~\cite{cisinski:topos,cisinski:presheaves} is a cofibrantly generated model structure on a Grothendieck topos whose cofibrations are the monomorphisms.
Therefore, any right proper, simplicially locally cartesian closed, simplicial Cisinski model category is an excellent model category.
Cisinski~\cite{cisinski:lccc-rpcmc} and Gepner--Kock~\cite{gk:univlcc} have shown that any locally presentable locally cartesian closed $(\infty,1)$-category (and in particular any Grothendieck $(\infty,1)$-topos~\cite{lurie:higher-topoi}) can be presented by such a model category.
\end{eg}

\begin{eg}
Any complete and cocomplete locally cartesian closed category with its trivial model structure (the weak equivalences are the isomorphisms and all morphisms are cofibrations and fibrations) is also a good model category.
If it is locally presentable, then it is an excellent model category.
Thus, we also include the standard examples of semantics for extensional type theory.
\end{eg}

\begin{noneg}
The category of groupoids, which was the first homotopical model of type theory~\cite{hs:gpd-typethy}, satisfies all the axioms of a good model category except that it is not locally cartesian closed: only fibrations of groupoids are exponentiable.
It is possible that with some care our method could be generalized to such examples, at the expense of increased awkwardness (for instance, \cref{thm:free-fibred} would no longer be true exactly as stated).
\end{noneg}

\begin{rmk}
The first version of this paper omitted the condition that good model categories should be \emph{simplicially} locally cartesian closed, required for copowers to be stable under pullback; see \cite{MO:enriched-cccs} and followups.    
\end{rmk}

Assumption~\ref{item:m0} tells us that \sM is simplicially enriched, with powers and copowers\footnote{Also known as cotensors and tensors, respectively.} satisfying the pushout-product and pullback-corner axioms, and (by assumption~\ref{item:m3}) preserved by pullbacks.
This enables us to construct path objects (which model identity types) as simplicial powers.
Specifically, if $\ivl$ denotes the 1-simplex (the simplicial interval), then for any fibrant object $A$, the power $A^\ivl$ is a path-object for $A$.
Likewise, if $p:A\fib \Gamma$ is a fibration, then the ``fibred power'' $A^\ivl_\Gamma = A^\ivl \times_{\Gamma^\ivl} \Gamma$ is a path-object for $A$ regarded as a type over $\Gamma$; see \cref{thm:stable-id}.

This is important for us because it means that homotopies can be represented in adjoint form: a map $A\to B^\ivl$ is equivalent to a map $A\ten \ivl \to B$, where $A\ten \ivl$ is the simplicial copower.
Note that like other colimits, the simplicial copower of $p:A\fib \Gamma$ in the slice category over $\Gamma$ is just $A\ten\ivl$ with the projection $A\ten\ivl \to A \to \Gamma$.
Moreover, each slice category of \sM is a simplicial model category, and pullback preserves both simplicial powers and copowers.
Given $A\in\sM/\Gamma$, we write $A^K_\Gamma$ for the simplicial power in $\sM/\Gamma$.

Every good model category is a type-theoretic model category in the sense of~\cite{shulman:invdia}, and hence the subcategory of fibrant objects is a type-theoretic fibration category.
In particular, cofibrations are stable under pullback, so acyclic cofibrations are stable under pullback along fibrations, and hence dependent products of fibrations along fibrations are fibrations.
Taken together, this implies that the fibrations in a good model category $\sM$ form a \emph{comprehension category} $\fibm$ that has the categorical structure corresponding to type theory with dependent sums (including a unit type), dependent products, and identity types.

To actually construct such an interpretation requires a coherence theorem making all the structure strictly stable under pullback.
We will use the coherence method of~\cite{lw:localuniv}, which applies to many different kinds of structure.
The input to this method is a general comprehension category $(\C,\T)$.
One then defines the \textbf{left adjoint splitting} $(\C,\T_!)$ (or just $\C_!$ for short) in which a type $A\in \T_!(\Gamma)$ consists of an object $V_A\in\C$, a type $E_A\in\T(V_A)$, and a map $\name{A}:\Gamma\to V_A$.
We call $V_A$ the ``local universe'' and think of this as a representative of the pullback of $E_A$ along $\name{A}$.
Reindexing of such types is done by simple composition ($A[f]$ is $\name{A}\circ f$) which is strictly associative.
The local universes technique then shows that if $\C$ satisfies a technical condition (see~\eqref{eq:lf} below) and $\C$ has ``weakly stable'' structure of some sort, meaning that it exists in each fiber $\T(\Gamma)$ and the reindexing of \emph{a} structure is \emph{a} structure, then $\C_!$ admits \emph{strictly} stable structure (obtained by constructing weakly stable structure once in the ``universal case''), and thus provides semantics for the corresponding type-theoretic rules.
See~\cite{lw:localuniv} for details.

We summarize the above discussion as follows.

\begin{thm}\label{thm:gmc-tt}
  Any good model category $\sM$ admits a natural structure of a comprehension category $\fibm$ where $\cF(\Gamma)$ is the category of fibrations with codomain $\Gamma$.
  Moreover, $\fibm$ has weakly stable dependent sums, unit type, dependent products, and identity types, and therefore $\fibmbang$ has strictly stable dependent sums, unit type, dependent products, and identity types.

  Similarly, writing $\fibmf$ for the restriction of $\fibm$ to fibrant objects, $\fibmf$ carries the same weakly stable structure, and hence $\fibmfbang$ the same strictly stable structure.
\end{thm}




\section{Coproducts}
\label{sec:coproducts}

We warm up for pushouts by considering the case of coproducts.
To start with, we specialize the definitions from~\cite[\sec3.4.1]{lw:localuniv} to our context of a good model category.
We qualify these definitions with ``typal'' (the adjective of ``type'') to distinguish them from the ordinary categorical constructions in \sM.

\begin{defn}\label{defn:sum}
  A \textbf{typal coproduct} of fibrations $A_1\to \Gamma$ and $A_2\to\Gamma$ consists of a fibration $A_1\oplus A_2 \to\Gamma$ with maps $\nu_i: A_i \to A_1\oplus A_2$ over $\Gamma$ such that for any fibration $C\to A_1\oplus A_2$ with sections $t_i : A_i \to C$ over $\nu_i$, there exists a section $s:A_1\oplus A_2 \to C$ such that $s\circ \nu_i = t_i$.

  We say \sM has \textbf{weakly stable typal coproducts} if for any $A_1\to \Gamma$ and $A_2\to\Gamma$ there exists a typal coproduct $A_1\oplus A_2$ such that for any $\sigma:\Delta\to\Gamma$, the pullback $\sigma^*(A_1\oplus A_2)$ with injections $\sigma^*\nu_i$ is a typal coproduct of $\sigma^*A_1$ and $\sigma^* A_2$.
\end{defn}

\begin{thm}[{\cite[Lemma 3.4.1.4]{lw:localuniv}}]
  If \sM has weakly stable typal coproducts, then its left adjoint splitting models type theory with coproduct types.
\end{thm}

Now, how do we actually \emph{construct} weakly stable typal coproducts?
Of course, since \sM is a model category, it has ordinary categorical coproducts, and the coproduct $A_1 + A_2 \to \Gamma$ satisfies all parts of \cref{defn:sum} except that it may not be a fibration.
(Admittedly, for some particularly nice \sM, such as simplicial sets, it is always a fibration.
However, we treat the general case not only out of a desire for generality, but because in the case of pushouts the analogous argument will be necessary even when \sM is simplicial sets.)

The obvious solution is to fibrantly replace it.
Thus, let $A_1 + A_2 \to A_1\oplus A_2 \fib \Gamma$ be an (acyclic cofibration, fibration) factorization.

\begin{thm}\label{thm:coproduct}
  For any fibrations $A_1 \fib \Gamma$ and $A_2\fib\Gamma$, 
  the fibration $A_1\oplus A_2 \fib\Gamma$ is a weakly stable typal coproduct.
\end{thm}
\begin{proof}
  We define the injections by composition with those of $A_1+A_2$.
  To show that $A_1\oplus A_2$ is a typal coproduct, let $C\fib A_1\oplus A_2$ be a fibration with sections $t_i$ over $A_i$.
  Then the universal property of $A_1+A_2$ induces a section $t : A_1+A_2 \to C$, and the lifting property of the acyclic cofibration $A_1+A_2 \to A_1\oplus A_2$ against the fibration $C\fib A_1\oplus A_2$ allows us to extend this section to $A_1\oplus A_2$.

  To show that this typal coproduct is weakly stable, it suffices to show that for any $\sigma :\Delta\to\Gamma$, the pullback $\sigma^*(A_1+A_2) \to \sigma^*(A_1\oplus A_2)$ is again an acyclic cofibration.
  Since cofibrations are stable under pullback, this map is a cofibration; thus it remains to show it is a weak equivalence.
  Weak equivalences are not generally stable under pullback, but weak equivalences between fibrations are; so we would be done if $A_1+A_2 \to \Gamma$ were a fibration, but the whole point is that it may not be.

  Right properness of \sM ensures that weak equivalences are also stable under pullback along fibrations, so we would be done if $\sigma$ were a fibration.
  In general it may not be either, but we can factor it as an acyclic cofibration followed by a fibration.
  Thus, without loss of generality we may assume that $\sigma$ is an acyclic cofibration.

  In this case, right properness tells us that the induced maps $\sigma^*A_i \to A_i$ and $\sigma^*(A_1\oplus A_2) \to (A_1\oplus A_2)$ are weak equivalences, hence acyclic cofibrations, since they are pullbacks of $\sigma$ along fibrations.
  Moreover, acyclic cofibrations are closed under coproducts (being the left class of a weak factorization system), so the induced map $\sigma^*A_1 +\sigma^*A_2 \to A_1+A_2$ is an acyclic cofibration.
  Moreover, since $\sigma^*$ is a left adjoint, it preserves coproducts.
  Thus, in the following square
  \[ \begin{tikzcd}
    \sigma^*(A_1+A_2) \ar[r,acyc,cof] \ar[d] & A_1+A_2 \arrow[d,acyc,cof]
    \\ \sigma^*(A_1\oplus A_2) \arrow[r, acyc] & A_1\oplus A_2
  \end{tikzcd} \]
  all the marked maps are acyclic cofibrations, hence weak equivalences.
  Hence, by 2-out-of-3, so is the remaining map, which is what we wanted.
\end{proof}

\begin{cor}
  The left adjoint splitting of \sM models coproduct types.\qed
\end{cor}

\section{Pushouts in model categories}
\label{sec:pushouts}

Higher inductive pushouts are not a traditionally standard type constructor, so we need to begin with definitions.
First we consider a sort of weakly stable structure that only makes sense in a context where we have simplicial homotopies.
In \cref{sec:coherence-pushouts} we will rephrase this in type-theoretic language and prove the relevant local universes coherence theorem.

\begin{defn}\label{defn:ivl-pushout}
  Let $f_1:A\to B_1$ and $f_2:A\to B_2$ be morphisms between fibrations $A\fib\Gamma$ and $B_i\fib\Gamma$.
  A \textbf{$\ivl$-typal pushout} is a fibration $D\fib\Gamma$ with maps $\nu_i : B_i \to D$ over $\Gamma$ and a homotopy $\mu:A \ten \ivl \to D$ over $\Gamma$ between $\nu_1 \circ f_1$ and $\nu_2\circ f_2$, such that for any fibration $C\fib D$ equipped with sections $t_i : B_i \to C$ over $\nu_i$ and a homotopy $u:A\ten \ivl \to C$ over $\mu$, there exists a section $s:D\to C$ such that $s\circ \nu_i = t_i$ and $s\circ \mu = u$.

  We say \sM has \textbf{weakly stable $\ivl$-typal pushouts} if for any $f_1,f_2$ there exists a $\ivl$-typal pushout $D$ such that for any $\sigma:\Delta\to\Gamma$, the pullback $\sigma^* D$ with injections $\sigma^*\nu_i$ and homotopy $\sigma^*\mu$ is a $\ivl$-typal pushout of $\sigma^*f_1$ and $\sigma^* f_2$.
\end{defn}

\begin{thm}
  Any good model category has weakly stable $\ivl$-typal pushouts.
\end{thm}
\begin{proof}
  Give $f_1,f_2$, let $Q$ be their explicit homotopy pushout, meaning the pushout of the following diagram in \sM:
  \[ \begin{tikzcd}
    A+A \ar[r,cof,"\iota"] \ar[d,"f_1+f_2"'] & A\ten\ivl
    \\ B_1+B_2 
  \end{tikzcd} \]
  By construction, this satisfies all parts of \cref{defn:ivl-pushout} except that it may not be a fibration over $\Gamma$.
  (And in this case it really isn't, even in simplicial sets.)
  Let $D$ be its fibrant replacement, i.e.\ we have a factorization $Q\to D \fib \Gamma$ as an acyclic cofibration followed by a fibration.
  Then $D$ has injections and a homotopy obtained by composition from $Q$, and for any fibration $C\fib D$ as in \cref{defn:ivl-pushout} we can first define a section over $Q$ by its universal property and then extend to $D$ by lifting against the acyclic cofibration $Q\to D$.

  It remains to deal with weak stability.
  As in \cref{thm:coproduct}, for this it suffices to show that $\sigma^*Q\to\sigma^*D$ is an acyclic cofibration for any acyclic cofibration $\sigma:\Delta\to\Gamma$, and by 2-out-of-3 it suffices to show that $\sigma^*Q \to Q$ is an acyclic cofibration.
  Again, $\sigma^*$ preserves colimits and simplicial copowers, so $\sigma^*Q $ is the pushout of $\sigma^*B_1+\sigma^*B_2$ and $\sigma^*A \ten\ivl$ under $\sigma^*A+\sigma^*A$.
  Furthermore, again as in \cref{thm:coproduct}, $\sigma^*A \to A$ and $\sigma^*B_i \to B_i$ are acyclic cofibrations, hence so are $\sigma^*A+\sigma^*A \to A+A$ and ${\sigma^*B_1 +\sigma^*B_2} \to B_1+B_2$.

  Now consider the following commutative cube, in which the left-hand and right-hand faces are pushouts, and the objects $R$ and $S$ are also pushouts.
  \[\begin{tikzcd}
    & \sigma^* A+\sigma^*A \arrow[dl,cof] \arrow[rr,cof,acyc] \arrow[dd] \ar[dr,phantom,"R"{name=sr}]
    & & A+A \arrow[dl,cof] \arrow[dd] \ar[to=sr] \\
    \sigma^*A\ten\ivl\arrow[dd] \ar[to=sr, crossing over] & & A\ten\ivl \ar[from=sr] \\
    & \sigma^*B_1+\sigma^*B_2 \arrow[dl] \arrow[rr,cof,near end,"\sim"] \ar[dr,phantom,"S"{name=tr}]
    \ar[from=sr,to=tr,crossing over]
    \arrow[from=ul,to=ur, crossing over,cof,near start,"\sim"']
    & & B_1+B_2 \arrow[dl] \ar[to=tr] \\
    \sigma^*Q \arrow[rr] \ar[to=tr] & & Q \arrow[from=uu, crossing over]
    \ar[from=tr] \\
  \end{tikzcd}\]
  To show that $\sigma^*Q \to Q$ is an acyclic cofibration, it will suffice to show that both of its factors $\sigma^*Q\to S$ and $S\to Q$ are such.
  The former is easy, since it is a pushout of the acyclic cofibration ${\sigma^*B_1 +\sigma^*B_2} \to B_1+B_2$.
  For the latter, a standard argument shows that it is the pushout of the map $R\to A\ten\ivl$ in the upper square.
  However, this map is the pushout-product of the acyclic cofibration $\sigma^*A\to A$ and the cofibration of simplicial sets $\mathbf{2}\to\ivl$.
  Thus since the model structure is simplicial, it is an acyclic cofibration as well.
\end{proof}

\section{Pushouts in comprehension categories}
\label{sec:coherence-pushouts}

Inside of (ordinary, Martin-L\"of) type theory, of course, we do not have a ``$\ivl$'', so \cref{defn:ivl-pushout} does not correspond directly to anything type-theoretic the way \cref{defn:sum} does.
Instead we need a version of this definition that refers only to identity types, which categorically means path-objects.
This is an instance of ``one type constructor stacked on top of another'', like the case of $\mathsf{W}$-types considered in~\cite[\sec3.4.4]{lw:localuniv}; hence we need to start by defining good classes of identity types.

Let $(\C,\T)$ be a comprehension category.
As in~\cite{lw:localuniv}, if $A\in\T(\Gamma)$ we denote its comprehension by $\Gamma.A\to\Gamma$, and its reindexing along $\sigma:\Delta\to\Gamma$ by $A[\sigma]$.

In what follows, by a \emph{family} we mean an \emph{indexed family}.
That is, a ``family of elements of $S$'', for any set $S$, consists of a set $F$ and a function $F\to S$.
We often abuse notation by identifying elements of $F$ with their images in $S$, but it is important that such a family is not just a subset of $S$.
There is a category $\mathrm{Fam}$ whose objects are families $F\to S$ and whose morphisms are commutative squares.

\begin{defn}\label{defn:id}
  A \textbf{stable class of identity types} on \C consists of:
  \begin{itemize}
  \item For each $A\in \T(\Gamma)$, a non-empty family $\G_\Id(A)$ of elements of $\T(\Gamma.A.A)$, called ``good identity types'' $\Id_A$.
    These must be weakly stable under reindexing, in that for any $\sigma:\Delta\to\Gamma$, there is a morphism of families $\G_\Id(A) \to \G_\Id(A[\sigma])$ over the reindexing functor $\T(\Gamma.A.A) \to \T(\Delta.A[\sigma].A[\sigma])$.
    If $\Id_A\in\G_\Id(A)$ is a good identity type for $A$, we abuse notation by writing $\Id_A[\sigma] \in \G_\Id(A[\sigma])$ for its image.
  \item For each good identity type $\Id_A\in\G_\Id(A)$, a non-empty family $\G_\r(A,\Id_A)$ of ``good reflexivity terms'' that are sections of $\Id_A[\delta_A] \in \T(\Gamma.A)$.
    These must be weakly stable under reindexing, in that we have morphisms of families $\G_\r(A,\Id_A) \to \G_\r(A[\sigma],\Id_A[\sigma])$ over the reindexing morphisms of types and good identity types.
  \item For each good identity type $\Id_A$ and good reflexivity term $\r$, and each type $C\in\T(\Gamma.A.A.\Id_A)$ equipped with a section $c$ of $C[\delta_A,\r]\in\T(\Gamma.A)$, a non-empty family $\G_\J(\Id_A,\r,C,c)$ of ``good extensions'' of $c$ to a section of $C$ itself.
    Of course, every good extension must actually be an extension, i.e.\ its composite with $\delta_A \circ \r$ must be $c$.
    Good extensions must also be weakly stable under reindexing, in that we have morphisms $\G_\J(\Id_A,\r,C,c) \to \G_\J(\Id_A[\sigma],\r[\sigma],C[\sigma],c[\sigma])$ over the reindexings of everything else.
  \end{itemize}
\end{defn}

Note that the underlying data of a good identity type and a good reflexivity term consist equivalently of a factorization of $\Gamma.A \to \Gamma.A.A$ through the comprehension $\Gamma.A.A.\Id_A \to \Gamma.A.A$ of some type $\Id_A$.

\begin{thm}\label{thm:stable-id}
  If \sM is a good model category, then \fibm has a stable class of identity types, called the \textbf{canonical stable class of identity types}, in which:
  \begin{itemize}
  \item $\G_\Id(A)$ is the set of objects of $\sM/\Gamma$ equipped with data exhibiting the universal property of the simplicial power $(\Gamma.A)^\ivl_\Gamma$ therein.
  \item For each such object, the factorization $\Gamma.A \to (\Gamma.A)^\ivl_\Gamma \to \Gamma.A.A \cong \Gamma.A \times_\Gamma \Gamma.A$ is induced by powering with the maps $\mathbf{2} \to \ivl \to \mathbf{1}$ of simplicial sets.
    In particular, every good identity type has exactly one good reflexivity term.
  \item For any $C$ and $c$, every extension of $c$ to $\Id_A$ is good in a unique way (i.e.\ the family of good extensions is simply the set of all extensions).
  \end{itemize}
\end{thm}
\begin{proof}
  The projection $(\Gamma.A)^\ivl_\Gamma \to \Gamma.A\times_\Gamma \Gamma.A$ is the pullback corner map for the fibrant object $\Gamma.A\in\sM/\Gamma$ and the cofibration $\mathbf{2}\to \ivl$.
  Since $\sM/\Gamma$ is a simplicial model category, this map is a fibration, hence the comprehension of a type over $\Gamma.A.A$.

  The stability of simplicial powers under pullback gives the reindexing operations for good identity types and good reflexivity terms.
  Note that unlike for many other type constructors such as $\Sigma$- and $\Pi$-types, although these identity types are determined by a 1-categorical universal property, the type-theoretic data we consider (the reflexivity term) is not sufficient to describe this universal property.

  Similarly, the inclusion $\Gamma.A\to (\Gamma.A)^\ivl_\Gamma$ is the pullback corner map for $A$ and the projection $\ivl \to \mathbf{1}$.
  Since $\Gamma.A$ is fibrant in $\sM/\Gamma$ and $\ivl \to \mathbf{1}$ is a weak equivalence between cofibrant objects, this is a weak equivalence.
  Moreover, it is a split monomorphism, hence a cofibration, and thus an acyclic cofibration.
  It follows that given any $C$ and $c$ there exists such an extension, i.e.\ the family of good extensions defined in the theorem statement is non-empty.
  Of course the reindexing of any extension is again an extension, and pseudofunctoriality is automatic.
\end{proof}

By the adjointness between simplicial powers and copowers, a simplicial homotopy $A\ten\ivl \to B$ between $f,g:A\toto B$ is equivalently a lift of $(f,g):A\to B\times B$ to any canonical identity type $\Id_B = B^\ivl$, i.e.\ a term of $\Id_B[(f,g)]$ in context $A$.
But in order to rephrase \cref{defn:ivl-pushout} relative to a stable class of identity types, we also need to talk about ``homotopies over homotopies'', for which we need \emph{dependent identity types}.
Inside of type theory, the dependent identity type looks like this, for a dependent type $x:A \types B(x)\type$:
\[ a_1:A, a_2:A, e:\id[A]{a_1}{a_2}, b_1:A(a_1), b_2:A(a_2) \types \idover[x.B(x)]{b_1}{b_2}{e} \type \]
That is, it tells us how to identify two points in different fibers along a path in the base.
Inside type theory, there are many ways to define such a type:
\begin{enumerate}\label{idover}
\item If we first define the \emph{transport} operation $\transf{e}:B(x) \to B(y)$ (using the eliminator for identity types), then we can define $\idover[x.B(x)]{u}{v}{e}$ to be $\id[B(y)]{\transf{e}(u)}{v}$.
  This is the definition used by~\cite{hottbook} and~\cite{hottcoq,bglsss:hottcoq}.\label{item:idover1}
\item We could instead use $\id[B(x)]{u}{\transf{(\opp{e})}(v)}$, where $\opp{e}:\id[A]{y}{x}$ is the inverse path of $e$ (i.e.\ $\opp{(\blank)}$ witnesses the symmetry of equality).\label{item:idover2}
\item We could use the eliminator for identity types on $e$, with $\idover[x.B(x)]{u}{v}{\refl_x}$ defined to be $\id[B(x)]{u}{v}$.
  This is the definition used by~\cite{hottagda}.\label{item:idover3}
\item We could define $\idover[x.B(x)]{u}{v}{e}$ as an inductive family, with a single constructor giving for any $x:A$ and $u:B(x)$ an element $\refl_u : \idover[x.B(x)]{u}{u}{\refl_x}$.\label{item:idover4}
\end{enumerate}
Compared to the first three, option~\ref{item:idover4} has the disadvantage that $\idover[x.B(x)]{u}{v}{\refl_x}$ is not judgmentally equal to $\id[B(x)]{u}{v}$.
However, option~\ref{item:idover4} is also the one that corresponds most directly to the native path-object structure in a good model category, so it is the one we will adopt.
In \cref{sec:pushouts-type-theory} we will show that our construction also gives pushouts relative to choices~\ref{item:idover1}--\ref{item:idover3}, albeit in a slightly weaker sense.

\begin{defn}\label{defn:dep-id}
  Suppose \C has a stable class of identity types.
  A \textbf{stable class of dependent identity types} relative to this stable class consists of the following data, each of which must be weakly stable in a straightforward sense.
  \begin{itemize}
  \item For each $A\in\T(\Gamma)$ and $B\in\T(\Gamma.A)$, and each good identity type $\Id_A\in \G_\Id(A)$, a non-empty family $\dG_\Id(A,\Id_A,B)$ of ``good dependent identity types'' $\Id^A_B$ in $\T(\Gamma.A.A.\Id_A.B.B)$.
  \item For each good dependent identity type as above, and each good reflexivity term $\r_A$ for $\Id_A$, a non-empty family $\dG_\r(A,\Id_A,\r,B,\Id^A_B)$ of ``good dependent reflexivity terms'' that are sections of $\Id^A_B[\delta_A,\r_A,\delta_B] \in \T(\Gamma.A.B)$.
  \item For each good dependent identity type and good dependent reflexivity term as above, and each type $C\in \T(\Gamma.A.A.\Id_A.B.B.\Id^A_B)$ equipped with a section $c$ of $C[\delta_A,\r_A,\delta_B,\r^A_B]\in \T(\Gamma.A.B)$, a nonempty family $\dG_\J(A,\Id_A,\r_A,B,\Id^A_B,\r^A_B)$ of ``good extensions'' of $c$ to a section of $C$ itself (which are actually extensions thereof).
  \end{itemize}
\end{defn}

Similarly to the non-dependent case, the underlying data of a good dependent identity type and a good dependent reflexivity term consist of a factorization of the diagonal of $\Gamma.A.B$ that lies over a given factorization of the diagonal of $\Gamma.A$, such that the dashed pullback map shown below is a dependent projection:
\[
\begin{tikzcd}[row sep=huge,column sep=huge]
  \Gamma.A.B \ar[r] \ar[d,fib] & \Gamma.A.A.\Id_A.B.B.\Id^A_B \ar[rr] \ar[d] \ar[drr,phantom,near start,"\Gamma.A.A.\Id_A.B.B"{name=hi}] \ar[dashed,fib,to=hi] && \Gamma.A.A.B.B \ar[d] \ar[from=hi] \\
  \Gamma.A \ar[r] & \Gamma.A.A.\Id_A \ar[rr,fib] \ar[from=hi] && \Gamma.A.A
\end{tikzcd}
\]

\begin{thm}\label{thm:stable-dep-id}
  If \sM is a good model category, then \fibm has a \textbf{canonical stable class of dependent identity types} over the canonical stable class of identity types, in which:
  \begin{itemize}
  \item $\dG_\Id(A,\Id_A,B)$ is the set of objects of $\sM/\Gamma$ equipped with the universal property of a simplicial power $(\Gamma.A.B)^\ivl_\Gamma$.
  \item For each such object, the above factorization is given by the following diagram:
    \[
    \begin{tikzcd}[row sep=large]
      \Gamma.A.B \ar[r] \ar[d,fib] & (\Gamma.A.B)^\ivl_\Gamma \ar[rr] \ar[d] \ar[drr,phantom,near start,"\bullet"{name=hi}] \ar[dashed,to=hi,fib] && \Gamma.A.A.B.B \ar[d] \ar[from=hi] \\
      \Gamma.A \ar[r] & (\Gamma.A)^\ivl_\Gamma \ar[rr,fib] \ar[from=hi] && \Gamma.A.A
    \end{tikzcd}
    \]
    In particular, every good dependent identity type has a unique good dependent reflexivity term.
  \item for any $C$ and $c$, every extension of $c$ to $\Id^A_B$ is good.
  \end{itemize}
\end{thm}
\begin{proof}
  Similarly to the non-dependent case, the dotted projection above is the pullback corner map for the fibration $\Gamma.A.B\to\Gamma.A$ and the cofibration $\mathbf{2}\to \ivl$ in the simplicial model category $\sM/\Gamma$, hence a fibration.
  The inclusion $\Gamma.A.B \to (\Gamma.A.B)^\ivl_\Gamma$ is an acyclic cofibration for the same reasons as in \cref{thm:stable-id}, so the set of good extensions is non-empty.
\end{proof}

Finally, in \cref{defn:ivl-pushout} we also need to postcompose a homotopy with a dependent function.
That is, given $x:A \types f(x):B(x)$, we need the following term:
\[ a_1:A, a_2:A, e:\id[A]{a_1}{a_2} \types \ap_f(e) : \idover[x.B(x)]{f(a_1)}{f(a_2)}{e} \]
Inside type theory there is a standard way to define $\ap$, defined using the eliminator for the identity type.
However, once again, in a good model category there is a different way to define it using the functoriality of simplicial powers (see \cref{thm:stable-ap}).
Thus, we also introduce this abstractly.

\begin{defn}\label{defn:ap}
  Suppose \C has stable classes of identity types and dependent identity types.
  A \textbf{stable class of identity applications} consists of, for each types $A\in\T(\Gamma)$ and $B\in\T(\Gamma.A)$, each section $f$ of $B$, each good identity type and reflexivity term for $A$, and each corresponding dependent identity type and dependent reflexivity term for $B$, a non-empty family of ``good identity applications'', which are sections $\ap_f$ of the projection $\Gamma.A.A.\Id_A.B.B.\Id^A_B \to \Gamma.A.A.\Id_A$ such that the following squares commute:
  \[
  \begin{tikzcd}
    \Gamma.A.B \ar[r,"\r"] & \Gamma.A.A.\Id_A.B.B.\Id^A_B \ar[r] & \Gamma.A.A.B.B \\
    \Gamma.A \ar[r,"\r"] \ar[u,"f"] & \Gamma.A.A.\Id_A \ar[u,"{\ap_f}"] \ar[r] & \Gamma.A.A \ar[u,"{(f,f)}",swap]
  \end{tikzcd}
  \]
  Moreover, good identity applications must be weakly stable under reindexing in an evident sense.
\end{defn}

\begin{thm}\label{thm:stable-ap}
  If \sM is a good model category, then \fibm has a \textbf{canonical stable class of identity applications} over the canonical stable classes of identity types and dependent identity types, in which the good sections $\ap_f$ are the maps $f^\ivl_\Gamma$ induced by the functoriality of simplicial powers.
\end{thm}
\begin{proof}
  The requisite squares commute by the two-variable functoriality of simplicial powers.
\end{proof}

Note that in order for these maps to be well-defined, we must require an object of $\G_\Id(A)$ to be \emph{equipped with} the structure of a simplicial power, and hence $\G_\Id(A)$ cannot be merely a sub\emph{set} of $\T(\Gamma.A.A)$, since this structure is not determined even by the reflexivity term.

Finally, we can define a type-theoretic notion of pushout.

\begin{defn}\label{defn:pushout}
  Suppose $\C$ has stable classes of identity types, dependent identity types, and identity applications.
  A \textbf{stable class of typal pushouts} relative to these stable classes consists of the following data, all weakly stable under reindexing.
  \begin{itemize}
  \item For each pair $f_1:\Gamma.A\to \Gamma.B_1$ and $f_2:\Gamma.A\to \Gamma.B_2$ of morphisms over $\Gamma$, a non-empty family of ``good pushout types'' $D\in\T(\Gamma)$ equipped with ``good injections'' $\nu_i : \Gamma.B_i \to \Gamma.D$ over $\Gamma$.
  \item For any good pushout $D$ and good injections $\nu_i$, and any good identity type for $D$, a non-empty family of ``good glueing data'' maps $\mu : \Gamma.A \to \Gamma.D.D.\Id_D$ over $(\nu_1 f_1,\nu_2 f_2)$.
  \item For any good identity type, good reflexivity term, and good gluing data and good injections for a good pushout $D$, and any type $C\in\T(D)$ equipped with sections $t_i$ over $\nu_i$, a good dependent identity type $\Id^D_C$ with good reflexivity term, and a morphism $u:A\to \Gamma.D.D.\Id_D.C.C.\Id^D_C$ over $(\mu,t_1,t_2)$, we have a non-empty family of good sections $s$ of $C$ such that $s \circ \nu_i = t_i$, and for any good identity application $\ap_s$ we have $\ap_s \circ \mu = u$.
  \end{itemize}
\end{defn}

\begin{thm}\label{thm:model-pushout}
  In any good model category, the $\ivl$-typal pushouts (\cref{defn:ivl-pushout}) are a stable class of typal pushouts relative to the canonical stable classes of identity types, dependent identity types, and identity applications.
\end{thm}
\begin{proof}
  The canonical stable classes were constructed exactly so that the data of \cref{defn:pushout} reduces to that of \cref{defn:ivl-pushout}.
\end{proof}

\section{Pushouts in type theory}
\label{sec:pushouts-type-theory}

Finally, we move to strict stability and type theory.
Recall the basic condition on a comprehension category $(\C,\T)$ that makes local universe theorems work:
\begin{equation}
  \label{eq:lf}
  \tag{LF} \parbox{10cm}{$\C$ has finite products, and dependent exponentials of display maps and product projections along display maps.}
\end{equation}
(A \textbf{display map} is a finite composite of the dependent projections $\Gamma.A\to\Gamma$ associated to the comprehensions of types.)

Of course, since a good model category $\sM$ is locally cartesian closed, $\fibm$ always satisfies condition~\eqref{eq:lf}.

\begin{defn}
  A split comprehension category $\C$ has:
  \begin{itemize}
  \item \textbf{strictly stable identity types} if it has a stable class of identity types (\cref{defn:id}) for which each family of good structures is a singleton.
  \item \textbf{strictly stable dependent identity types} if in addition it has a stable class of dependent identity types for which each family of good structures is a singleton.
  \item \textbf{strictly stable identity applications} if in addition it has a stable class of identity applications for which each family of good structures is a singleton.
  \item \textbf{strictly stable typal pushouts} if in addition it has a stable class of typal pushouts for which each family of good structures is a singleton.
\end{itemize}
\end{defn}

Our strictly stable identity types are easily seen to be equivalent to~\cite[Definition 3.4.3.1]{lw:localuniv} with the ``Frobenius'' condition omitted (which we don't bother with since our intended models all have $\Pi$-types).
Syntactically, the above strictly stable structure corresponds to the type-theoretic rules shown in \cref{fig:id,fig:depid,fig:ap,fig:pushouts}.

\begin{rmk}\label{rmk:rule-style}
  Because we are not concerned with syntactic properties such as admissibility of substitution, we have no qualms about adding variables to the context of the conclusion.
  As a result, multiple rules often have the same premises; we combine these by listing multiple conclusions in the same rule.
  This has the additional advantage that each multiple-conclusion rule corresponds to one new local universe on the semantic side.
  (On the other hand, unlike in some syntactic presentations, we do include all type judgment premises in all rules, since these are necessary semantically.)
\end{rmk}

\begin{figure}
  \centering
  \begin{mathpar}
    \inferrule{\Gamma\types A\type} 
    {\Gamma,a_1:A, a_2:A\types \id[A]{a_1}{a_2}\type \\ 
      \Gamma,a:A \types \refl_a : \id[A]{a}{a}}\and
    \inferrule{\Gamma\types A\type \\ \Gamma,x:A,y:A,e:\id[A]{x}{y}\types C\type \\ \Gamma,x:A \types c:C[x/y,\refl_x/e]} 
    {\Gamma,a_1:A,a_2:A,p:\id[A]{a_1}{a_2}\types \J(xye.C,x.c,a_1,a_2,p):C[a_1/x,a_2/y,p/e] \\ 
    \Gamma,a:A\types \J(xye.C,x.c,a,a,\refl_a) \jdeq c[a/x]}
  \end{mathpar}
  \caption{Identity types}
  \label{fig:id}
\end{figure}

\begin{figure}
  \centering
  \begin{mathpar}
    \inferrule{\Gamma\types A\type \\ \Gamma,x:A \types B\type}
    {\Gamma, a_1:A, a_2:A, e:\id[A]{a_1}{a_2}, b_1:B[a_1/x], b_2:B[a_2/x] \types \idover[B]{b_1}{b_2}{e}\type \\ 
    \Gamma, a:A, b:B[a/x]\types \drefl_b : \idover[B]{b}{b}{\refl_a}}\and
    \inferrule{\Gamma\types A\type \\ \Gamma,x:A \types B\type \\ \Gamma,x:A,y:A,e:\id[A]{x}{y}, u:B, v:B[y/x], d:\idover[B]{u}{v}{e} \types C\type \\ \Gamma,x:A,u:B \types c:C[x/y,\refl_x/e,u/v,\drefl_u/d]}
    {\Gamma, a_1:A, a_2:A, p:\id[A]{a_1}{a_2}, b_1:B[a_1/x], b_2:B[a_2/x], q:\idover[B]{b_1}{b_2}{p} \hspace{3em} \\ 
      \hspace{1em} \types \J'(xyeuvd.C,xu.c,a_1,a_2,p,b_1,b_2,q):C[a_1/x,a_2/y,p/e,b_1/u,b_2/v,q/d] \\ 
    \Gamma, a:A, b:B[a/x]\types \J'(xyeuvd.C,xu.c,a,a,\refl_a,b,b,\drefl_b) \jdeq c[a/x,b/u] \hspace{1em}}
  \end{mathpar}
  \caption{Dependent identity types}
  \label{fig:depid}
\end{figure}

\begin{figure}
  \centering
  \begin{mathpar}
    \inferrule{\Gamma\types A\type \\ \Gamma,x:A\types B\type \\ \Gamma,x:A \types f:B} 
    {\Gamma,a_1:A, a_2:A, p:\id[A]{a_1}{a_2} \types \ap_{x.f}(a_1,a_2,p) : \idover[B]{f[a_1/x]}{f[a_2/x]}{p} \\ 
      \Gamma,a:A\types \ap_{x.f}(a,a,\refl_a)\jdeq \drefl_{f[a/x]}}\and
  \end{mathpar}
  \caption{Identity applications}
  \label{fig:ap}
\end{figure}

\begin{figure}
  \centering
  \begin{mathpar}
    \inferrule{\Gamma\types A\type \\ \Gamma\types B_1\type \\ \Gamma\types B_2\type \\\\
      \Gamma,x:A\types f_1:B_1 \\ \Gamma,x:A \types f_2:B_2}
    {\Gamma \types \fPush (x.f_1,x.f_2) \type \\\\
    \Gamma,y_1:B_1\types \nu_1(y_1):\fPush (x.f_1,x.f_2) \\
    \Gamma,y_2:B_2\types \nu_2(y_2):\fPush (x.f_1,x.f_2) \\
    \Gamma,a:A\types \mu(a):\id[\fPush(x.f_1,x.f_2)]{\nu_1(f_1[a/x])}{\nu_2(f_2[a/x])}}\and
    \inferrule{\Gamma\types A\type \\ \Gamma\types B_1\type \\ \Gamma\types B_2\type \\\\
      \Gamma,x:A\types f_1:B_1 \\ \Gamma,x:A \types f_2:B_2 \\\\ \Gamma, u:\fPush(x.f_1,x.f_2) \types C\type \\
      \Gamma, y_1:B_1\types t_1:C[\nu_1(y_1)/u] \\ \Gamma, y_2:B_2\types t_2:C[\nu_2(y_2)/u] \\
      \Gamma, x:A \types m:\idover[C]{t_1[f_1(a)/y_1]}{{t_2[f_2(a)/y_2]}}{\mu(x)}}
    {\Gamma,p:\fPush(x.f_1,x.f_2)\types \mathsf{pe}(u.C,y_1.t_1,y_2.t_2,a.m,p) : C[p/u] \\
    \Gamma,b_1:B_1 \types \mathsf{pe}(u.C,y_1.t_1,y_2.t_2,a.m,\nu_1(b_1)) \jdeq t_1[b_1/y_1] \\
    \Gamma,b_2:B_2 \types \mathsf{pe}(u.C,y_1.t_1,y_2.t_2,a.m,\nu_2(b_2)) \jdeq t_2[b_2/y_2] \\
    \Gamma,a:A \types \ap_{u.\mathsf{pe}(u.C,y_1.t_1,y_2.t_2,a.m,u)}(a) \jdeq m[a/x]}\and
  \end{mathpar}
  \caption{Pushouts}
  \label{fig:pushouts}
\end{figure}

We now briefly sketch the local universe coherence theorems for all of these structures; they are all straightforward applications of the techniques of~\cite{lw:localuniv}.

\begin{thm}
  If $\C$ has a stable class of identity types and satisfies~\eqref{eq:lf}, then $\C_!$ has strictly stable identity types.
\end{thm}
\begin{proof}
  Just like~\cite[Theorem 3.4.3.2]{lw:localuniv}, only simpler due to the absence of Frobenius.
  If $A\in\T_!(\Gamma)$ are represented by $E_A\in\T(V_A)$, then $\Id_A\in\T_!(\Gamma.A.A)$ is represented by a chosen good identity type $\Id_{E_A}\in \T(V_A.E_A.E_A)$, with reflexivity term also chosen in the same universal case.
  The local universe for the elimination and computation rules is just as in~\cite[Theorem 3.4.3.2]{lw:localuniv} without the types ``$B_i$'':
  \begin{align*}
    V = [ & a:V_A, \\
    & c:\prod x,x':E_A(a), y:\Id_{E_A}(a,x,x'). V_C,\\
    & d:\prod x:E_A(a). E_C(c(x,x,\r_{E_A}(a,x)))].
  \end{align*}
  (As in~\cite{lw:localuniv}, we express local universes as iterated $\Sigma$-types in the internal extensional type theory of $\C$.)
\end{proof}

\begin{thm}
  If $\C$ has a stable class of dependent identity types and satisfies~\eqref{eq:lf}, then $\C_!$ has strictly stable dependent identity types.
\end{thm}
\ifmpcps\begin{proof*}\else\begin{proof}\fi
  If $A\in\T_!(\Gamma)$ and $B\in\T_!(\Gamma.A)$ is represented by $E_A\in\T(V_A)$ and $E_B\in\T(V_B)$, then $\Id^A_B\in\T_!(\Gamma.A.A.\Id_A.B.B)$ is represented by a chosen good dependent identity type for $E_B$ over $E_A$, with the local universe
  \begin{align*}
    [ a:V_A, b: (V_B)^{E_A(a)} ]
  \end{align*}
  with dependent reflexivity term likewise chosen in this case.
  (Note that this local universe is the same one $V_A \triangleleft V_B$ used in~\cite{lw:localuniv} for $\Pi$-types and $\Sigma$-types.)
  The local universe for the elimination and computation rules is
  \begin{align*}
    [ & a:V_A, \\
    & b:\prod x:E_A(a). V_B,\\
    & c:\prod x,x':E_A(a), y:\Id_{E_A}(a,x,x'), u:E_B(b(x)), u':E_B(b(x')), z:\Id^{E_A}_{E_B}(a,x,x',y,b,u,u'). V_C,\\
    & d:\prod x:E_A(a), u:E_B(b(x)) . E_C(c(x,x,\r_{E_A}(a,x),u,u,\r^{E_A}_{E_B}(a,x,u)))].\qedhere
  \end{align*}
\ifmpcps\end{proof*}\else\end{proof}\fi

It is perhaps worth emphasizing how the above local universes are obtained essentially algorithmically from the rules in \cref{fig:depid}, once given that the latter are written in the style of \cref{rmk:rule-style}.
Namely, each premise of a rule corresponds to one term in the iterated $\Sigma$-type, which is a $\Pi$-type of the consequent of that premise over the additional variables (beyond the arbitrary context $\Gamma$) in the context of that premise.

\begin{thm}
  If $\C$ has a stable class of identity applications and satisfies~\eqref{eq:lf}, then $\C_!$ has strictly stable identity applications.
\end{thm}
\ifmpcps\begin{proof*}\else\begin{proof}\fi
  The local universe is
  \begin{align*}
    [ & a:V_A \\
    & b:\prod x:E_A(a). V_B,\\
    & f:\prod x:E_A(a). E_B(b(x)) ].\qedhere
  \end{align*}
\ifmpcps\end{proof*}\else\end{proof}\fi

\begin{thm}\label{thm:lu-pushout}
  If $\C$ has a stable class of typal pushouts and satisfies~\eqref{eq:lf}, then $\C_!$ has strictly stable typal pushouts.
\end{thm}
\begin{proof}
  The local universe for the formation and introduction rules is
  \begin{align*}
    [ & a:V_A, b_1:V_{B_1}, b_2:V_{B_2}, \\
    &f_1: \prod x:E_A(a). E_{B_1}(b_1),\\
    &f_2: \prod x:E_A(a). E_{B_2}(b_2) ].
  \end{align*}
  And the local universe for the elimination and computation rules is
  \begin{align*}
    [ & a:V_A, b_1:V_{B_1}, b_2:V_{B_2}, \\
    &f_1: \prod x:E_A(a). E_{B_1}(b_1),\\
    &f_2: \prod x:E_A(a). E_{B_2}(b_2),\\
    &c : \prod u:E_{\fPush(f_1,f_2)}(a,b_1,b_2,f_1,f_2). V_C, \\
    &t_1 : \prod y_1:E_{B_1}(b_1). E_C(c(\nu_1(y_1))), \\
    &t_2 : \prod y_2:E_{B_2}(b_2). E_C(c(\nu_2(y_2))), \\
    &m : \prod x:E_A(a). \Id^{\fPush(f_1,f_2)}_C(a,b_1,b_2,f_1,f_2,c,\nu_1(f_1(x)),\nu_2(f_2(x)),\mu(x),t_1(f_1(x)),t_2(f_2(x)))
    ].
  \end{align*}
\end{proof}

\begin{thm}
  For any good model category $\sM$, the split comprehension category $\fibmbang$ has strictly stable typal pushouts, relative to the strictly stable identity types, dependent identity types, and identity applications obtained by strictifying its canonical weakly stable ones.
\end{thm}
\begin{proof}
  Combine \cref{thm:model-pushout,thm:lu-pushout}.
\end{proof}

This is almost, but not quite, the result we want.
The pushout rules in \cref{fig:pushouts} differ from those commonly used (e.g.\ in~\cite{hottbook}) by the presence of the additional rules in \cref{fig:depid,fig:ap}.
The usual approach is instead to start with only identity types as in \cref{fig:id}, and then \emph{define} dependent identity types using one of the methods~\ref{item:idover1}--\ref{item:idover3} listed on page~\pageref{idover}, similarly define $\ap$ using the eliminator of identity types, and then state rules such as those in \ref{fig:pushouts} relative to these defined structures.

Such definitions certainly give \emph{some} classes of dependent identity types and identity applications, which are strictly stable if the identity types we started with are strictly stable.
However, if such constructions are performed in $\fibmbang$, the resulting strictly stable classes will not generally be the \emph{same} as those obtained from the canonical weakly stable ones by the local universes construction.
They will indeed be \emph{equivalent}, and we can transfer the structure of a pushout across such an equivalence, but at the cost of weakening the computation rule for the path-constructor from a judgmental equality to the existence of a path, as follows.

\begin{defn}\label{defn:wk-pushouts}
  A \textbf{stable class of weak typal pushouts} in a comprehension category $(\C,\T)$, relative to given stable classes of identity types, dependent identity types, and identity applications, consists of all the data and properties in \cref{defn:pushout} except that instead of $\ap_s\circ \mu = m$, we have for any good identity type $\Id_{\Id^D_C}\in\T(\Gamma.D.D.\Id_D.C.C.\Id^D_C.\Id^D_C)$ a non-empty family of ``good'' maps $v:A\to \Id_{\Id^D_C}$ over $(\ap_s\circ \mu, m)$.
  If $\C$ is split and each family of good structures is a singleton, we say that $\C$ has \textbf{strictly stable weak typal pushouts}.
\end{defn}

Of course, typal pushouts are also weak typal pushouts, since if $\ap_s\circ \mu = m$ we can compose this map with $\r$ to get $v$.
The corresponding modified elimination rule in syntactic type theory is shown in~\cref{fig:wk-pushout}.
In the case when the dependent identity types and identity applications are defined from the ordinary identity types as on page~\pageref{idover}, this yields exactly the notion of pushout type used in~\cite[\sec6.8]{hottbook}.
Some reasons for choosing a weak computation rule for the path-constructor are adumbrated in~\cite[\sec6.2 and Chapter 6 Notes]{hottbook}, one of which (non-strictness of the equivalence between left and right homotopies) is roughly the same reason such a rule appears here.

\begin{figure}
  \centering
  \begin{mathpar}
    \inferrule{\Gamma\types A\type \\ \Gamma\types B_1\type \\ \Gamma\types B_2\type \\\\
      \Gamma,x:A\types f_1:B_1 \\ \Gamma,x:A \types f_2:B_2 \\\\ \Gamma, u:\fPush(x.f_1,x.f_2) \types C\type \\
      \Gamma, y_1:B_1\types t_1:C[\nu_1(y_1)/u] \\ \Gamma, y_2:B_2\types t_2:C[\nu_2(y_2)/u] \\
      \Gamma, x:A \types m:\idover[C]{t_1[f_1(a)/y_1]}{{t_2[f_2(a)/y_2]}}{\mu(x)}}
    {\Gamma,a:A \types \mathsf{pec}(u.C,y_1.t_1,y_2.t_2,a.m,a) : \id{\ap_{u.\mathsf{pe}(u.C,y_1.t_1,y_2.t_2,a.m,u)}(a)}{m[a/x]}}\and
  \end{mathpar}
  \caption{Computation for weak pushouts}
  \label{fig:wk-pushout}
\end{figure}

The following theorem could be stated for weakly stable structures in addition to strict ones, but we only need to apply it after the local universes splitting has occurred, so we don't bother with that generality.

\begin{thm}\label{thm:transfer-depid}
  Let $\C$ be a split comprehension category equipped with strictly stable identity types.
  If $\C$ has strictly stable weak typal pushouts with respect to one choice of strictly stable dependent identity types and identity applications, then with respect to any other choice of strictly stable dependent identity types and identity applications it also has strictly stable weak typal pushouts.
\end{thm}
\begin{proof}
  We use the ordinary notation for the given identity types, dependent identity types, identity applications, and typal pushouts, and we write $\idovertwo[B]{x}{y}{e}$ and $\aptwo{f}$ and so on for some other choice of dependent identity types and identity applications.
  Using the eliminators for dependent identity types in both directions (or equivalently the uniqueness-up-to-homotopy of weak factorization systems), we obtain maps in both directions making all the triangles commute:
  \[
  \begin{tikzcd}
    & \Gamma.A.A.\Id_A.B.B.\Id^A_B \ar[dr,fib] \ar[dd,shift left,"h"]\\
    \Gamma.A.B \ar[ur,"{\r^A_B}"] \ar[dr,swap,"{\rtwo{A}{B}}"] && \Gamma.A.A.\Id_A.B.B\\
    & \Gamma.A.A.\Id_A.B.B.\Idtwo{A}{B} \ar[ur,fib] \ar[uu,shift left,"k"]
  \end{tikzcd}
  \]
  Similarly, we obtain homotopies over $\Gamma.A.A.\Id_A.B.B$ between the round-trip composites in both directions (i.e.\ maps $\Gamma.A.B \to \Gamma.A.A.\Id_A.B.B.\Id^A_B.\Id^A_B.\Id_{\Id^A_B}$ and similarly), giving a homotopy equivalence $\Id^A_B \simeq \Idtwo{A}{B}$.
  And if we have a section $f$ of $B\in\T(\Gamma.A)$, then the composite
  \[ \Gamma.A.A.\Id_A \xto{\ap_f} \Gamma.A.A.\Id_A.B.B.\Id^A_B \xto{h} \Gamma.A.A.\Id_A.B.B.\Idtwo{A}{B} \]
  is a lift of $(f,f)$, and has the property that when composed with $\r_A : \Gamma.A \to \Gamma.A.A.\Id_A$ it becomes $\rtwo{A}{B}$.
  Since $\aptwo{f}$ also has these properties, we can use identity elimination to produce a homotopy $h\circ \ap_f \sim \aptwo{f}$ over $(f,f)$, i.e.\ a lift of $(h\circ \ap_f,\aptwo{f})$ to $\Gamma.A.A.\Id_A.B.B.\Idtwo{A}{B}.\Idtwo{A}{B}.\Id_{\Idtwo{A}{B}}$.

  Now, note that only the eliminator and computation rules for pushouts refer to dependent identity types; the formation and introduction rules can be exactly as given.
  For the eliminator, all the structure on $C$ in the second case can be used as-given except for $u:A\to \Gamma.D.D.\Id_D.C.C.\Idtwo{D}{C}$, which we compose with $k$ to produce data for the given eliminator.
  This produces a section of the desired form, satisfying the computation rules $s\circ \nu_i = t_i$ and with a homotopy $\ap_s\circ \mu \sim k\circ m$.
  Postcomposing this homotopy with $h$ (using ordinary $\Id$-elimination) and concatenating with the above-constructed homotopies $h\circ \ap_s \sim \aptwo{s}$ and $h\circ k \sim 1$, we get
  \[ \aptwo{s} \circ \mu \sim h\circ \ap_s \circ \mu \sim h\circ k\circ m \sim m \]
  which is what we want.
  Since everything that went into the construction was assumed strictly stable, so is the result.
\end{proof}

\begin{cor}
  For any good model category $\sM$, the split comprehension category $\fibmbang$ has strictly stable weak typal pushouts, relative to the strictly stable identity types obtained by strictifying its canonical weakly stable ones, and the dependent identity types obtained from any of~\ref{item:idover1}--\ref{item:idover3} on page~\pageref{idover} and the identity applications obtained from $\Id$-elimination.\qed
\end{cor}

\begin{rmk}
  As noted in the introduction, from pushout types (and binary sums and the empty type) we can construct all finite (homotopy) colimits.
  In particular, the coequalizer of $f,g:A\to B$ is the pushout of $[f,g] : A+A \to B$ and the fold map $\nabla: A+A\to A$.
  Conversely, the pushout of $f_1:A\to B_1$ and $f_2:A\to B_2$ can be constructed as the coequalizer of $\inl\circ f_1, \inr\circ f_2 : A\to B_1+B_2$.
  Thus, to obtain all finite colimits we could equally well have chosen to treat coequalizers rather than pushouts; similar issues arise in the construction of both.
\end{rmk}

\section{Natural numbers}
\label{sec:natural-numbers}

As a warm-up for the recursive case, we consider the interpretation of the type of natural numbers.
Though a standard type constructor, this is not considered explicitly in~\cite{lw:localuniv}; thus we begin with its local universes theory.

\begin{defn}
  In a comprehension category $(\C,\T)$, a \textbf{natural numbers type} over $\Gamma\in\C$ is a type $\N\in\T(\Gamma)$ together with a section $\zero:\Gamma\to\Gamma.\N$ and a map $\succ:\Gamma.\N\to\Gamma.\N$ over $\Gamma$, plus an operation assigning to any type $C\in\T(\Gamma.\N)$ equipped with $z:\Gamma\to\Gamma.\N.C$ over $\zero$ and $s:\Gamma.\N.C \to \Gamma.\N.\C$ over $\succ$, a section $f$ of $C$ such that $f \circ \zero = z$ and $f\circ \succ = s \circ f$.

  We say $\C$ has \textbf{weakly stable natural numbers types} if for any $\Gamma$ there is a natural numbers type $\N\in\T(\Gamma)$ such that for any $\sigma:\Delta\to\Gamma$, $(\N[\sigma],\zero[\sigma],\succ[\sigma])$ is a natural numbers type over $\Delta$.
  Similarly, if $\C$ is split, it has \textbf{strictly stable natural numbers types} if it has an operation assigning to each $\Gamma$ a natural numbers type $\N\in\T(\Gamma)$ such that reindexing along any $\sigma:\Delta\to\Gamma$ preserves natural numbers types along with their data and specified sections.
\end{defn}

Of course, if $\C$ has a terminal object (as it usually does), then the quantification over $\Gamma$ in the latter definition follows from the special case $\Gamma=1$.
Strictly stable natural numbers types are the semantic version of the usual rules for a natural numbers type, shown in \cref{fig:nno}.

\begin{figure}
  \centering
  \begin{mathpar}
    \inferrule{ }{\Gamma\types\N\type}\and
    \inferrule{ }{\Gamma\types \zero:\N}\and
    \inferrule{ }{\Gamma,x:\N\types \succ(x):\N}\and
    \inferrule{\Gamma,x:\N\types C\type \\ \Gamma\types z:C[\zero/x] \\ \Gamma,x:\N,y:C \types s : C[\succ(x)/x]}{\Gamma,n:\N\types \nrec(x.C,z,xy.s,n) : C[n/x] \\ \Gamma\types \nrec(x.C,z,xy.s,\zero) \jdeq z\\ \Gamma,n:\N \types \nrec(x.C,z,xy.s,\succ(n)) \jdeq s[n/x,\nrec(x.C,z,xy.s,n)/y]}
  \end{mathpar}
  \caption{Natural numbers type}
  \label{fig:nno}
\end{figure}

\begin{thm}
  If $\C$ satisfies~\eqref{eq:lf} and has weakly stable natural numbers types, then $\C_!$ has strictly stable natural numbers types.
\end{thm}
\ifmpcps\begin{proof*}\else\begin{proof}\fi
  Condition~\eqref{eq:lf} provides a terminal object.
  We take this to be the local universe for the formation and introduction rules, choosing a weakly stable natural numbers type over it.
  For the elimination and computation rules, the local universe is
  \begin{align*}
    [&c: \prod x:E_{\N}. V_C,\\
    &z: E_C(c(\zero)),\\
    &s: \prod x:E_{\N}, y:E_C(c(x)) . E_C(c(\succ(x))) ].\qedhere
  \end{align*}
\ifmpcps\end{proof*}\else\end{proof}\fi

In a good model category, the obvious candidate for a natural numbers type is the countable coproduct $\coprod_{n:\dN} 1$ of copies of the terminal object, which is automatically a ``natural numbers object'' in the usual sense of category theory.
As usual, the problem is that in general this may not be fibrant (though it is in many examples, such as simplicial sets).
But unlike in the preceding sections, it does not suffice to simply fibrantly replace it.
If $\N$ is a fibrant replacement of $\coprod_{n:\dN} 1$, we can extend the zero and successor operations to it, and for any fibration $C\fib\N$ equipped with $z$ and $s$ we can find a section of $C$ over $\coprod_{n:\dN} 1$ using its universal property and then extend that section to $\N$ by lifting against the acyclic cofibration $\coprod_{n:\dN} 1 \to \N$.

However, such a lift will not generally satisfy the computation rules.
It might be possible to choose it cleverly to satisfy the computation rule for $\zero$, but for $\succ$ there is essentially no hope, since the computation rule for $\succ$ relates the section to \emph{itself}, which is impossible to obtain using a simple lifting property.
Thus, we have to be more clever.

Let $F_\N$ denote the endofunctor of $\sM$ defined by $F_\N(X) = X+1$.
Recall that an \textbf{endofunctor-algebra} for $F_\N$ is an object $X$ together with a map $F_\N(X) \to X$.
A natural numbers object in the ordinary sense is precisely an \emph{initial} such endofunctor-algebra.
And a natural numbers type over the terminal object is precisely a fibrant $F_\N$-algebra $\N$ such that any fibration $C\fib\N$ that is also a map of $F_\N$-algebras has an $F_\N$-algebra section.
This suggests that $\N$ is an ``initial fibrant $F_\N$-algebra''; we now make this precise using the technology of \emph{algebraic weak factorization systems}~\cite{gt:nwfs,garner:soa,riehl:nwfs-model}.

From now on, therefore, we will assume that $\sM$ is an \emph{excellent} model category, i.e.\ that in addition to being good it is a combinatorial model category.
This allows us to use the technology of algebraically-free monads and algebraic colimits of monads~\cite{kelly:transfinite,nlab:transfinite}.
Let $\dT_\N$ denote the algebraically-free monad on $F_\N$, which exists since $F_\N$ is an accessible endofunctor.
This means it is a monad $\dT_\N$ such that $\dT_\N$-monad-algebra structures on any object correspond bijectively to $F_\N$-endofunctor-algebra structures on that object, by precomposition with a given natural transformation $F_\N \to \dT_\N$.
By construction of $\dT_\N$, it is also accessible.

Now, by the construction of~\cite{garner:soa} and the combinatoriality of $\sM$, we have an accessible \textbf{fibrant factorization monad} $\dR$ on the arrow category $\sM^\to$, such that an arrow can be given an $\dR$-algebra structure if and only if it is a fibration, and $\dR$ preserves codomains.
In particular, when restricted to arrows with terminal codomain, we get an accessible \textbf{fibrant replacement monad} $\dR_1$ on $\sM$ itself, such that an object can be given an $\dR_1$-algebra structure if and only if it is fibrant.
Moreover, the composite of two $\dR$-algebras $X \xfib{f} Y \xfib{g} Z$ naturally acquires an $\dR$-algebra structure such that the commutative square
\[
\begin{tikzcd}
  X \ar[d,"gf",swap,fib] \ar[r,"f",fib] & Y \ar[d,"g",fib] \\ Z \ar[r,idmap]& Z
\end{tikzcd}
\]
is a morphism of $\dR$-algebras; this is noted for instance in~\cite[Remark 5.15]{riehl:mon-ams}.

\begin{thm}
  If $\sM$ is an excellent model category, then $\fibm$ has a natural numbers type over $1$.
\end{thm}
\begin{proof}
  Consider the algebraic coproduct of monads $\dT_\N+\dR_1$.
  By definition, this is an (accessible) monad such that a $(\dT_\N+\dR_1)$-algebra structure on an object $X$ is precisely a pair of a $\dT_\N$-algebra structure (i.e.\ an $F_\N$-endofunctor-algebra structure) and an (unrelated) $\dR_1$-algebra structure.
  Let $\N$ be the initial $(\dT_\N+\dR_1)$-algebra, i.e.\ the free $(\dT_\N+\dR_1)$-algebra on the initial object, $\N =(\dT_\N+\dR_1)(\emptyset)$.
  Since $\N$ is a $\dT_\N$-algebra, it comes with $\zero$ and $\succ$; and since it is an $\dR_1$-algebra, it is fibrant.

  Now suppose $p:\N.C\fib \N$ is a fibration of $F_\N$-endofunctor-algebras.
  Choose an $\dR$-algebra structure on $p$.
  Then the composite $\N.C\fib\N\fib 1$ acquires an $\dR$-algebra structure, which is to say that $\N.C$ acquires an $\dR_1$-algebra structure, such that $p$ is (not just an $\dR$-algebra but) an $\dR_1$-algebra map.
  Since $\N.C$ is also a $\dT_\N$-algebra and $p$ is also a $\dT_\N$-algebra map, it follows that $\N.C$ is a $(\dT_\N+\dR_1)$-algebra and $p$ is a $(\dT_\N+\dR_1)$-algebra map.
  But since $\N$ is the initial $(\dT_\N+\dR_1)$-algebra, any $(\dT_\N+\dR_1)$-algebra map with codomain $\N$ has a $(\dT_\N+\dR_1)$-algebra section, and hence in particular a $\dT_\N$-algebra section, as desired.
\end{proof}

It is not clear to us whether these natural numbers types are weakly stable on $\fibm$ in general; but we can obtain weak stability with a small modification.
Let $\Mf$ denote the subcategory of fibrant objects in $\sM$, and $\fibmf$ the restriction of the comprehension category $\fibm$ to $\Mf$.
Note that $\Mf$ is the smallest subcategory of $\sM$ to which $\cF$ can be restricted and still be a comprehension category, since every fibrant object is the comprehension of some fibration over $1$.
In particular, any semantics of type theory in $\fibm$ must land inside $\Mf$, so not much is lost by this restriction.

\begin{lem}
  For any good model category $\sM$, the comprehension category $\fibmf$ also satisfies~\eqref{eq:lf}.
\end{lem}
\begin{proof}
  The product of fibrant objects is of course fibrant, and the assumptions on a good model category ensure that the dependent exponential of a fibration along a fibration is again a fibration, hence its domain is fibrant if its codomain is.
\end{proof}

Moreover, any weakly stable structure possesed by $\fibm$ is also possesed by $\fibmf$.
Thus, all our preceding theorems about strictly stable structure in $\fibmbang$ also apply to $\fibmfbang$.

\begin{thm}\label{thm:nat-stable}
  If $\sM$ is an excellent model category, then $\fibmf$ has weakly stable natural numbers types.
\end{thm}
\begin{proof}
  It suffices to show that any natural numbers type $\N$ over $1$ is weakly stable, which is to say that for any (fibrant!)\ object $\Gamma$ the pullback $\Gamma^* \N$ is a natural numbers type over $\Gamma$.
  The proof is essentially the usual category-theoretic argument that pullback functors preserve natural numbers objects.
  Suppose $C\fib \Gamma^* \N$ is a fibration over $\Gamma$ equipped with a section $z:\Gamma\to C$ over $\Gamma^*(\zero)$ and a map $s:C\to C$ over $\Gamma^*(\succ)$.
  Then (since $\Gamma$ is fibrant) $\Gamma_* C \fib \Gamma_* \Gamma^* \N$ is a fibration, equipped with a point $\Gamma_*(z) : \Gamma_*(\Gamma^*(1)) = \Gamma_*(\Gamma) \to \Gamma_*(C)$ over $\Gamma_*(\Gamma^*(\zero))$ and a map $\Gamma_*(s):\Gamma_*(C)\to\Gamma_*(C)$ over $\Gamma_*(\Gamma^*(\succ))$.

  Now pull $\Gamma_* C$ back to $\N$ along the unit $\eta_\N$ of the adjunction $\Gamma^* \adj \Gamma_*$.
  Since $\eta$ is natural we have $\Gamma_*(\Gamma^*(\zero)) \circ \eta_1 = \eta_\N \circ \zero$ and $\Gamma_*(\Gamma^*(\succ)) \circ \eta_\N = \eta_\N \circ \succ$, so we have an induced map $z' : 1 \to \eta_\N^*\Gamma_*C$ over $\zero$ and $s':\eta_\N^*\Gamma_*C \to \eta_\N^*\Gamma_*C$ over $\succ$.
  Thus, using $\nrec$ for $\N$ we have a section $f:\N\to \eta_\N^*\Gamma_* C$ such that $f \circ \zero = z'$ and $f\circ \succ = s' \circ f$, or equivalently a map $g:\N \to \Gamma_*C$ over $\eta_\N$ such that $g\circ \zero = \Gamma_*(z) \circ \eta_1$ and $g\circ \succ = \Gamma_*(s) \circ g$.
  Finally, transposing $g$ across the adjunction $\Gamma^* \adj \Gamma_*$, we obtain $\Gamma^* \N \to C$ over $1_\N$ (the adjunct of $\eta_\N$) such that $h\circ \Gamma^*(\zero) = z$ and $h\circ \Gamma^*(\succ) = s\circ h$, as desired.
\end{proof}

\begin{cor}
  If $\sM$ is an excellent model category, then $\fibmfbang$ has strictly stable natural numbers types.\qed
\end{cor}

\section{\W-types}
\label{sec:w-types}

The construction of \W-types is similar to that of natural numbers, with one additional wrinkle.
Namely, the presence of parameters means that in weak stability we have to consider pullback along arbitrary morphisms $\sigma:\Delta\to\Gamma$ between fibrant objects, and in the case when $\Gamma$ is not terminal, fibrancy of $\Delta$ does not imply that $\sigma$ is a fibration.

Since a good model category is locally cartesian closed, the weakly stable $\Pi$-types in $\fibm$ asserted by \cref{thm:gmc-tt} are actually \emph{pseudo-stable}~\cite[Definition 3.4.2.8]{lw:localuniv}; thus we can use the definition of \textbf{weakly stable \W-types} relative to these from~\cite[Definition 3.4.4.9]{lw:localuniv}.

\begin{thm}
  If $\sM$ is an excellent model category, then $\fibm$ has \W-types.
  That is, for any $\Gamma\in\sM$, $A\in\F(\Gamma)$, and $B\in\F(\Gamma.A)$, there exists a \W-type as in~\cite[Definition 3.4.4.7]{lw:localuniv}.
\end{thm}
\begin{proof}
  Let $F_{\W A B}$ be the endofunctor of $\sM/\Gamma$ sending $X \to \Gamma$ to the local exponential $(A^*X)^B_A$, equipped with the composite projection $(A^*X)^B_A \to A\to \Gamma$.
  Since this is a composite of left or right adjoints between accessible categories, it is an accessible endofunctor, and $\sM/\Gamma$ is locally presentable, so $F_{\W A B}$ generates an algebraically-free monad $\dT_{\W A B}$.
  Let $\Gamma.\W$ be the initial algebra of $\dT_{\W A B} + \dR_{\Gamma}$, where $\dR_\Gamma$ is the restriction of $\dR$ to morphisms with target $\Gamma$.
  Then $\Gamma.\W\to \Gamma$ is a fibration, since it is an $\dR$-algebra, and its $F_{\W A B}$-endofunctor-algebra structure is exactly the map $\fold$ required of a \W-type.

  The input to the eliminator consists, as usual, of a fibration $p:\Gamma.\W.C\fib \Gamma.W$ that is a $F_{\W A B}$-algebra morphism.
  This is not completely obvious, with the eliminator defined following syntax as in~\cite[Definition 3.4.4.7]{lw:localuniv}; the point is that $\Gamma.A.\Pi[B,\W].\Pi[B,C[\mathsf{app}'_{B,W}]] \cong \Gamma.A.\Pi[B,\Sigma[\W,C]]$ by the mapping-in universal property of $\Sigma$-types; and by construction of $\Sigma$-types in $\sM$, the comprehension $\Gamma.\Sigma[\W,C]\fib\Gamma$ is just the composite $\Gamma.\W.C \fib \Gamma.W\fib \Gamma$.\footnote{Note a typo in~\cite[Definition 3.4.4.7]{lw:localuniv}: each $\Gamma.\W.\cdots$ in the left-hand columns of the diagrams should be $\Gamma.A.\cdots$.})
  As before, we choose an $\dR$-algebra structure on $p$, making this composite $\Gamma.\W.C \fib \Gamma.W\fib \Gamma$ an $\dR$-algebra and the square
  \[
  \begin{tikzcd}
    \Gamma.\W.C \ar[d,fib] \ar[r,fib] & \Gamma.\W \ar[d,fib] \\ \Gamma \ar[r,idmap]& \Gamma
  \end{tikzcd}
  \]
  an $\dR$-algebra map, hence also a $(\dT_{\W A B} + \dR_{\Gamma})$-algebra map.
  Thus, since $\W$ is the initial $(\dT_{\W A B} + \dR_{\Gamma})$-algebra, this morphism has a $(\dT_{\W A B} + \dR_{\Gamma})$-algebra section, and in particular a $F_{\W A B}$-endofunctor-algebra section, which is $\wrec$.
\end{proof}

\begin{rmk}
  In~\cite{vdbm:wtypes-hott} it is shown that 1-categorical \W-types in simplicial sets and certain other well-behaved model categories automatically preserve fibrancy.
  But in the general case we need to include a fibrant replacement, as before.
\end{rmk}

\begin{thm}\label{thm:wk-w}
  If $\sM$ is an excellent model category, then $\fibmf$ has weakly stable \W-types.
\end{thm}
\begin{proof}
  Let $\W\fib\Gamma$ be a $\W$-type constructed as above for $\Gamma\in\sM$, $A\in\F(\Gamma)$, and $B\in\F(\Gamma.A)$; we must show that its pullback along any $\sigma:\Delta\to\Gamma$ is also a \W-type.
  We consider separately the cases when $\sigma$ is an acyclic cofibration or a fibration; by factorization this suffices for the general case.

  If $\sigma$ is a fibration, the argument is similar to that of \cref{thm:nat-stable}.
  Suppose given a fibration $C\fib \sigma^*\W$ over $\Delta$ with structure map $d$ over $\sigma^*(\fold)$:
  \[
  \begin{tikzcd}
    (\sigma^* A)_! (\sigma^* f)_* (\sigma^*B)^* C \ar[rr,"d"] \ar[d] && C \ar[d,fib] \\
    (\sigma^* A)_! (\sigma^* f)_* (\sigma^*B)^* \sigma^* \W \ar[r,"\cong"] & \sigma^*(f_* B^* \W) \ar[r,"{\sigma^*(\fold)}",swap] & \sigma^*\W
  \end{tikzcd}
  \]
  Then using several Beck-Chevalley transformations, we have a composite map
  \begin{multline*}
    A_! f_* B^* \sigma_* C
    \toiso A_! f_* (B^*\sigma)_* (\sigma^*B)^* C
    \toiso A_! (A^*\sigma)_* (\sigma^* f)_* (\sigma^*B)^* C\\
    \to \sigma_* (\sigma^* A)_! (\sigma^* f)_* (\sigma^*B)^* C
    \xto{\sigma_* d} \sigma_* C
  \end{multline*}
  that lies over the corresponding map for $\sigma^*\W$ constructed from $\sigma^*(\fold)$.
  Now if we pull this back along the unit $\eta_\W : \W \to \sigma_* \sigma^* \W$ we get a map $d'$ satisfying
  \[
  \begin{tikzcd}
    A_! f_* B^* \eta_\W^* \sigma_* C \ar[r,"d'"] \ar[d] & \eta_\W^* \sigma_* C \ar[d,fib] \\
    A_! f_* B^* \W \ar[r,"\fold"] & \W
  \end{tikzcd}
  \]
  and thus a section $\wrec:\W\to \eta_\W^* \sigma_* C$ commuting with $d'$ and $\fold$.
  Transposing this back across the adjunction $\sigma^* \adj \sigma_*$, we obtain a section of $C$ commuting with $d$ and $\sigma^*(\fold)$, as desired.

  On the other hand, if $\sigma$ is an acyclic cofibration, then so is its pullback $\W^*\sigma : \sigma^*\W \to \W$ along the fibration $\W\fib\Gamma$.
  Moreover, since $\Delta$ is fibrant, so is the object $\sigma^*\W$.
  Therefore, the acyclic cofibration $\W^*\sigma$ admits a retraction $r$.
  Since $(\W^*\sigma)^* r^* C \cong C$, it suffices to extend the given structure $d$ on the fibration $C\fib\sigma^*\W$ to a corresponding structure on $r^* C\fib\W$ whose pullback to $\sigma^*\W$ is $d$.
  This is contained in the following lemma.
\end{proof}

\begin{lem}\label{thm:W-cofmnd}
  Suppose a pair of pullback squares in a good model category $\sM$:
  \[
  \begin{tikzcd}
    C \ar[r,cof,acyc',"i"] \ar[d,fib,"p",swap] \ar[dr,phantom,near start,"\lrcorner"] & D \ar[d,fib,"q"] \\
    V \ar[r,cof,acyc',"j"] \ar[d,fib] \ar[dr,phantom,near start,"\lrcorner"] & W \ar[d,fib] \\
    \Delta \ar[r,cof,acyc',"\sigma"] & \Gamma
  \end{tikzcd}
  \]
  in which all objects are fibrant, the downward-pointing arrows are fibrations, and the rightward-pointing arrows are acyclic cofibrations.
  Suppose moreover that for some fibration $f:B\fib A$ between fibrant objects of $\sM/\Gamma$, the object $W$ of $\sM/\Gamma$ is a $F_{\W A B}$-endofunctor-algebra, the object $C$ is a $F_{\W (\sigma^*A) (\sigma^*B)}$-endofunctor-algebra, and $p$ is a $F_{\W (\sigma^*A) (\sigma^*B)}$-morphism (when $V$ has its induced structure from $W$).
  Then there is a $F_{\W A B}$-endofunctor-algebra structure on $D$ inducing the structure on $C$ by pullback and making $q$ a $F_{\W A B}$-morphism.
\end{lem}
\begin{proof}
  By definition, $F_{\W A B}(X) = f_* B^* X$ (with its canonical map to $A$ forgotten), and similarly $F_{\W (\sigma^*A) (\sigma^*B)}(X) = (\sigma^*f)_* (\sigma^*B)^* X$.
  Thus we are given the following commutative diagram of solid arrows and we want to construct the dashed arrow making the other squares commute:
  \[
  \begin{tikzcd}
    (\sigma^*f)_* (\sigma^*B)^* C \ar[rr] \ar[dr] \ar[dd] && f_* B^* D \ar[dr,dashed] \ar[dd]\\
    & C \ar[rr,->,near end,"i",swap,crossing over] && D \ar[dd,fib,"q"] \\
    (\sigma^*f)_* (\sigma^*B)^* V \ar[rr] \ar[dr] && f_* B^* W \ar[dr] \\
    & V \ar[rr,->,"j",swap] \ar[from=uu,near end,"p",crossing over]  && W
  \end{tikzcd}
  \]
  This is equivalent to solving the following lifting problem:
  \[
  \begin{tikzcd}
    (\sigma^*f)_* (\sigma^*B)^* C \ar[r] \ar[d] & C \ar[r] & D \ar[d,fib] \\
    f_* B^* D \ar[urr,dashed] \ar[r] & f_* B^* W \ar[r] & W
  \end{tikzcd}
  \]
  so it suffices to show that $(\sigma^*f)_* (\sigma^*B)^* C \to f_* B^* D$ is an acyclic cofibration.
  Now since $C = \sigma^*D$, we have $(\sigma^*B)^*C = \sigma^* B^* D$, so the map $(\sigma^*B)^*C \to B^* D$ is a pullback of $\sigma$ along the composite fibration $B^*D \to D \to \Gamma$, hence an acyclic cofibration.
  But $(\sigma^*B)^*C$ is also the pullback of $B^* D$ along the map $B^*\sigma : \sigma^*B \to B$, so by the Beck-Chevalley condition for the pullback square
  \[
  \begin{tikzcd}
    \sigma^* B \ar[r,"{B^*\sigma}",acyc',cof] \ar[d,"{\sigma^* f}",swap,fib] & B \ar[d,"f",fib]\\
    \sigma^* A \ar[r,"{A^*\sigma}",acyc',cof] & A
  \end{tikzcd}
  \]
  we have $(\sigma^*f)_* (\sigma^*B)^* C \cong (A^*\sigma)^* f_* B^* D$.
  Thus, our desired map $(\sigma^*B)^*C \to B^* D$ is equivalently the pullback $(A^*\sigma)^* f_* B^* D \to f_* B^* D$ of the acyclic cofibration $A^*\sigma$ along the fibration $f_* B^* D \fib A$, and hence an acyclic cofibration.
\end{proof}

\begin{cor}
  If $\sM$ is an excellent model category, then $\fibmfbang$ has strictly stable \W-types.\qed
\end{cor}

\section{A higher inductive type not constructible from pushouts}
\label{sec:blass}

As mentioned in \cref{sec:introduction}, it is known that (at least in the presence of a universe) many recursive higher inductive types can in fact be constructed using only pushouts and the natural numbers, including all $n$-truncations and some localizations.
Thus, it seems worthwhile to give explicitly an example of a higher inductive type that \emph{cannot} be so constructed.

Briefly, the idea is that higher inductive types can be used to construct free algebras for infinitary algebraic theories.
However, Blass showed (modulo a large cardinal assumption) that these cannot be constructed in ZF \cite{blass:freealg}.
Therefore, no higher inductive types that are modelled by ZF can suffice to construct these.

To give this in detail, we work internally, and write $x=y$ instead of $\id x y$.
Let $\mathsf{even},\mathsf{odd}:\N\to\N$ be the inclusions of the even and odd numbers respectively, and let $\mathsf{par}:\N\to\N+\N$ be the inverse of the bijection $[\mathsf{even},\mathsf{odd}]:\N+\N\to\N$.
Given functions $f,g:\N\to A$, we write $f\cup g:\N\to A$ for the composite $\N \xto{\mathsf{par}} \N+\N \xto{[f,g]}A$.
And given $a:A$, we write $\{a\}:\N\to A$ for the constant function at $a$.

Now let $\iF$ be the higher inductive type generated by the following constructors (some numbered for later reference).
Intuitively, $\iF$ can be thought of as a set of notations for countable ordinals, so that its existence implies that of an ordinal of uncountable cofinality.
\ifmpcps
  \begin{enumerate} \renewcommand{\theenumi}{(\arabic{enumi})}
\else
  \begin{enumerate}[label=(\arabic*)]
\fi
\item[$\bullet$] $0:\iF$.
\item[$\bullet$] $\iS:\iF\to\iF$.
\item[$\bullet$] $\sup : (\N\to\iF)\to\iF$.
\item[$\bullet$] $\iF$ is a 0-type.
  As in~\cite[\sec7.3]{hottbook} we can ensure this with two constructors:
  \begin{itemize}
  \item $h: (S^1\to \iF) \to \iF$.
  \item For each $r:S^1\to\iF$ and $x:S^1$, an identification $r(x) = h(r)$.
  \end{itemize}
\item For any $f,g:\N\to\N$ such that $\forall n. ((\exists m. f(m)=n) \leftrightarrow (\exists m. g(m)=n))$, and any $h:\N\to \iF$, we have $\sup(h\circ f) = \sup(h\circ g)$.\label{item:f1}
\item For any $f,g:\N\to\iF$ we have $\sup(f \cup \{\sup(f\cup g)\}) = \sup(f\cup g)$.\label{item:f2}
\item For any $f,g:\N\to\iF$ we have $\sup(f \cup \{\iS(\sup(f\cup g))\}) = \iS(\sup(f\cup g))$.\label{item:f3}
\item This constructor requires some setup. Fix a bijection $\mathsf{pair}:\N\times\N\to \N$; then suppose given $b,c:\N\to\N$ jointly surjective, and $L:\N\to\N\to\N$ such that $\forall n. \exists m,\ell. (L(b(m),\ell)=c(n))$ and $\forall n. \exists m,\ell. (L(c(m),\ell)=b(n))$.  For $h:\N\to\iF$, set $h'(n) = \sup(h\circ L(n))$.  Now given such $h$, take by induction $h_0 = h$, $h_{k+1} = h_k'$; then set $f_k = h_k \circ b$ and $g_k = h_k \circ c$, and finally $\fbar(\mathsf{pair}(k,n)) = f_k(n)$ and $\gbar(\mathsf{pair}(k,n)) = g_k(n)$. Note that for each $i$, $\fbar(i)$, $\gbar(i)$ are algebraic expressions in the variables $h$ under the operation $\sup$.

  Now the constructor says: for each such $b$, $c$, $L$, and $h$, $\sup(\fbar) = \sup(\gbar)$.\label{item:f4}
\item $\sup(\{0\}) = 0$.\label{item:f5}
\end{enumerate}

\begin{thm}[{cf.\ \cite[\sec9]{blass:freealg}}]
  Assuming excluded middle and the higher inductive type $\iF$, there exists an uncountable regular cardinal.
\end{thm}
\begin{proof}
  The constructors of \iF correspond to the operations and the similarly-numbered axioms of the variety of algebras considered in~\cite[\sec9]{blass:freealg}.
  For~\ref{item:f2},~\ref{item:f3} and~\ref{item:f5} this is fairly obvious.

  Axiom~\ref{item:f1} of~\cite[\sec9]{blass:freealg} is $\sup(x_0,x_1,\dots) =\sup(y_0,y_1,\dots)$, where $\{x_0,x_1,\dots\}$ and $\{y_0,y_1,\dots\}$ are equal subsets of some fixed countably infinite set of variables.
  The enumerations of such variables $x_0,x_1,\dots$ and $y_0,y_1,\dots$ therefore correspond to $f,g:\N\to\N$, and the hypothesis of our~\ref{item:f1} says that these have the same image.

  It is not entirely clear to us how axiom~\ref{item:f4} of~\cite[\sec9]{blass:freealg} is to be interpreted as an equation of an algebraic theory; its statement appears to be conditional on certain equalities, hence essentially-algebraic rather than algebraic.
  Our~\ref{item:f4} is a modification (which may have been essentially what was intended in~\cite[\sec9]{blass:freealg}) that is algebraic and satisfies the following two necessary properties.

  First, we need that the canonical supremum operation on ordinal numbers satisfies~\ref{item:f4} (with \iF interpreted as the ordinals).
  To see this, note that the assumption on $L$ ensures that for any $h$, for each $n$ there is an $m$ with $h(b(n))\le h'(c(m))$, and dually.
  Thus, for any $k,n$ there is an $m$ with $f_k(n) \le g_{k+1}(m)$, and dually.
  Thus $\fbar$ and $\gbar$ are mutually cofinal, hence have the same supremum.

  Second, in the proof of Subcase 3.1 in~\cite[\sec9]{blass:freealg}, we need~\ref{item:f4} to imply that if $h:\N\to\iF$ satisfies $h(n) = \sup(\setof{ h(m) | m\le n})$, and $b,c:\N\to\N$ are jointly surjective and each cofinal, then $\sup(h\circ b)=\sup(h\circ c)$.
  To see this, let $L(n):\N\to\N$ be an enumeration of $\setof{m|m\le n}$; then cofinality ensures the hypothesis of~\ref{item:f4}, while the other assumption ensures that $h'=h$.
  Using~\ref{item:f1}, therefore, we have $\sup(h\circ b) = \sup(\fbar) = \sup(\gbar) = \sup(h\circ c)$.

  Now the proof of~\cite[\sec9]{blass:freealg} applies to derive a contradiction from the assumption that $\omega$ is the only infinite regular cardinal.
\end{proof}

\begin{cor}
  If it is consistent with ZFC that there are arbitrarily large strongly compact cardinals, then ZF does not prove that the higher inductive type \iF exists in $\mathbf{Set}$.
  In particular, since $\mathbf{Set}$ models type theory with pushouts and natural numbers, under this assumption \iF cannot be constructed from pushouts and the natural numbers.
\end{cor}
\begin{proof}
  As in~\cite[\sec9]{blass:freealg}, this follows from Gitik's~\cite{gitik:unc-sing} construction of a model of ZF with no uncountable regular cardinals, assuming the stated large cardinal principle.
\end{proof}

Thus, there is real extra power in recursive higher inductive types, above and beyond that obtainable by combining non-recursive higher inductive types and recursive ordinary inductive types.

\section{Propositional truncation}
\label{sec:prop-trunc}

Having dealt with both higher path-constructors in \crefrange{sec:pushouts}{sec:pushouts-type-theory} and recursive point-constructors in \crefrange{sec:natural-numbers}{sec:w-types}, we now combine them and consider a type with a recursive path-constructor.
One of the simplest such is the propositional truncation, whose type-theoretic rules are shown in \cref{fig:proptrunc}.
According to the usual pattern, there should be a second computation rule regarding $\treq$; but the typal form of this (which, as in \cref{sec:pushouts-type-theory}, is all we will obtain eventually) follows automatically since everything is a proposition, so we do not bother to posit it separately.
The semantic version of this definition is as follows.

\begin{figure}
  \centering
  \begin{mathpar}
    \inferrule{\Gamma\types A\type}{\Gamma\types \brck{A}\type}\and
    \inferrule{\Gamma\types A\type}{\Gamma,x:A\types \tr({x}) : \brck{A}}\and
    \inferrule{\Gamma\types A\type}{\Gamma,x:\brck A,y:\brck A\types \treq(x,y) : \id[\brck A]{x}{y}}\and
    \inferrule{\Gamma\types A\type \\ \Gamma,z:\brck A \types C\type \\ \Gamma,x:A \types c:C[\tr(x)/z] \\ \Gamma,x:\brck A,y:\brck A,u:C[x/z],v:C[y/z]\types d:\idover[z.C]{u}{v}{\treq(x,y)}}{\Gamma,w:\brck A \types \trrec(z.C,x.c,xyuv.d,a) : C[w/z] \\ \Gamma,a:A \types \trrec(z.C,x.c,xyuv.d,\tr(a)) \jdeq c[a/x]}
  \end{mathpar}
  \caption{Propositional truncation}
  \label{fig:proptrunc}
\end{figure}

\begin{defn}\label{defn:proptrunc}
  Let $\C$ be a comprehension category with stable classes of identity types and dependent identity types.
  A \textbf{stable class of propositional truncations} consists of, for any $A\in\T(\Gamma)$ a non-empty family of ``good propositional truncations'' $\brck A\in\T(\Gamma)$ with morphisms $\tr: \Gamma.A \to \Gamma.\brck A$ over $\Gamma$, and for any good identity type $\Id_{\brck A}$ a non-empty family of sections $\treq$ of the projection $\Gamma.\brck A.\brck A.\Id_{\brck A} \fib \Gamma.\brck A.\brck A$, plus for any $C\in\T(\Gamma.\brck A)$ equipped with a section $t:\Gamma.A\to\Gamma.\brck A.C$ over $\tr$, a good dependent identity type $\Id^{\brck A}_C$ over $\Id_{\brck A}$, and a section of the projection $\Gamma.\brck A.\brck A.\Id_{\brck A}.C.C.\Id^{\brck A}_C \to \Gamma.\brck A.\brck A.C.C$ over  $\treq$, a section $\trrec$ of $C$ such that $\trrec \circ \tr = t$; all equipped with appropriate reindexing operations.
  We say $\C$ has \textbf{strictly stable propositional truncations} if it has a stable class thereof in which each family of good structures is a singleton.
\end{defn}

The local universes coherence lemma is proven as usual; we combine it with the analogue of \cref{thm:transfer-depid}.

\begin{lem}
  If $\C$ satisfies~\eqref{eq:lf} and has a stable class of propositional truncations relative to some stable classes of identity types and dependent identity types, then $\C_!$ has strictly stable propositional truncations relative to the identity types obtained by strictifying the given ones and \emph{any} dependent identity types over these.
\end{lem}
\begin{proof}
  The local universe for the formation and introduction rules is just $V_A$, while that for the elimination and computation rules is
  \begin{align*}
    [& a:V_A,\\
    &c:\prod z:E_{\brck A}(a) . V_C \\
    &t:\prod x:E_{A}(a) . E_C(c(\tr(x))) \\
    &d:\prod x,y:E_{\brck A}(a), u:E_C(c(x)), v:E_C(c(y)) . E_{\Id^{\brck A}_C}(a,x,y,u,v) ]
  \end{align*}
  And if $C$ has eliminator data relative to some other dependent identity types, then by composing it with the map $k$ from \cref{thm:transfer-depid} we obtain eliminator data for the given strictified ones, allowing us to define the desired section.
\end{proof}

Passing to the model-categorical context, it is useful to isolate the following definition.

\begin{defn}\label{defn:prop}
  A morphism $X\to \Gamma$ in a good model category is a \textbf{proposition over $\Gamma$} if it is equipped with a section of the projection $X^\ivl_\Gamma \to X\times_\Gamma X$, i.e.\ a simplicial homotopy over $\Gamma$ between the two projections $X\times_\Gamma X \toto X$.
  A \textbf{map of propositions over $\Gamma$} is a map over $\Gamma$ that preserves these sections (or homotopies).
\end{defn}

Note that when interpreted with respect to the canonical identity types and dependent identity types in a good model category, the introduction rule for $\brck A$ says that it is a fibrant proposition with a map from $A$, and the elimination rule says that any fibration over $\brck A$ that is a map of propositions and has a section over $A$ has a section over $\brck A$.

\cref{defn:prop} does not require $X\to \Gamma$ to be a fibration, but when it is, the definition has the following reformulations.

\begin{lem}\label{thm:prop-diag}
  The following are equivalent for a fibration $p:X\fib\Gamma$ in a good model category.
  \begin{enumerate}
  \item $p$ admits some structure of a proposition over $\Gamma$.\label{item:pd1}
  \item The diagonal $X\to X\times_\Gamma X$ is a weak equivalence.\label{item:pd2}
  \item The diagonal $X\to X\times_\Gamma X$ is an acyclic cofibration.\label{item:pd3}
  \end{enumerate}
\end{lem}
\begin{proof}
  The diagonal is always a split mono, hence a cofibration, so \ref{item:pd2}$\Leftrightarrow$\ref{item:pd3} is easy.
  And \ref{item:pd3}$\Rightarrow$\ref{item:pd1} by lifting in the square
  \[
  \begin{tikzcd}
    X \ar[d,cof,acyc] \ar[r,"\r"] & X^\ivl_\Gamma \ar[d,fib] \\
    X\times_\Gamma X \ar[r,idmap] & X\times_\Gamma X
  \end{tikzcd}
  \]
  in which the right-hand map is a fibration since $p$ is a fibration, as in \cref{thm:stable-id}.
  Finally, if~\ref{item:pd1} then any two parallel maps over $\Gamma$ with target $X$ are simplicially homotopic, by composing with the specified section.
  Thus, consider either projection $X\times_\Gamma X\to X$; this is a retraction of the diagonal, while the composite $X\times_\Gamma X\to X \to X\times_\Gamma X$ is simplicially homotopic to the identity by the above.
  Thus the diagonal is a simplicial homotopy equivalence.
\end{proof}

\begin{lem}\label{thm:prop-over}
  Given a proposition $X\fib \Gamma$ over $\Gamma$ in a good model category, the following are equivalent for a further fibration $Y\fib X$.
  \begin{enumerate}
  \item $Y\fib X$ admits some structure of a proposition over $X$.
  \item the composite $Y\fib X\fib \Gamma$ admits some structure of a proposition over $\Gamma$.
  \item the composite $Y\fib X\fib \Gamma$ admits some structure of a proposition over $\Gamma$ such that $Y\to X$ is a map of propositions.
  \end{enumerate}
\end{lem}
\begin{proof}
  We have a pullback square
  \[
  \begin{tikzcd}
    Y\times_X Y \ar[r,cof,acyc] \ar[d] \ar[dr,phantom,near start,"\lrcorner"] & Y\times_\Gamma Y \ar[d,fib]\\
    X \ar[r,cof,acyc] & X\times_\Gamma X
  \end{tikzcd}
  \]
  in which the right-hand map is a fibration (the product of two fibrations in $\sM/\Gamma$) and the bottom map is an acyclic cofibration by \cref{thm:prop-diag}, hence the top map is also an acyclic cofibration.
  Thus, by the 2-out-of-3 property, the diagonal $Y\to Y\times_X Y$ is a weak equivalence (i.e.\ $Y$ can be structured as a proposition over $X$) if and only if the diagonal $Y\to Y\times_\Gamma Y$ is a weak equivalence (i.e.\ $Y$ can be structured as a proposition over $\Gamma$).
  Moreover, in this case we can structure $Y$ as a proposition over $\Gamma$ making $Y\to X$ a map of propositions by lifting in the following square:
  \[
  \begin{tikzcd}
    Y \ar[r] \ar[d,cof,acyc] & Y^\ivl_\Gamma \ar[d,fib]\\
    Y\times_\Gamma Y \ar[r] & (Y\times_\Gamma Y)\times_{(X\times_\Gamma X)} X^\ivl_\Gamma
  \end{tikzcd}
  \]
  Here the right-hand map is a fibration since it is the pullback corner product in $\sM/\Gamma$ of the fibration $Y\fib X$ and the cofibration $\mathbf{2} \to \ivl$, so such a lift exists.
\end{proof}

The following proof introduces one more new idea: algebraic pushouts of monads to ``glue in recursive paths''.

\begin{thm}
  If $\sM$ is an excellent model category, then $\fibmf$ has a stable class of propositional truncations relative to its canonical stable classes of identity types and dependent identity types, and hence $\fibmfbang$ has strictly stable propositional truncations.
\end{thm}
\begin{proof}
  Let $F_{\ivl}$ be the endofunctor of $\sM/\Gamma$ defined by $F_{\ivl}(X) = (X \times_\Gamma X) \otimes \ivl$.
  Then an $F_{\ivl}$-endofunctor-algebra structure on $X$ consists of a simplicial homotopy between two maps $X \times_\Gamma X \toto X$.
  These two maps can be arbitrary and are given as part of the algebra structure, but we can identify them with the projections by taking an algebraic colimit of monads.

  If we define $F_{\mathbf{2}}(X) = (X\times_\Gamma X) \otimes \mathbf{2}$, then an $F_{\mathbf{2}}$-endofunctor-algebra structure consists of two maps $X \times_\Gamma X \toto X$.
  Moreover, there is a natural transformation $F_{\mathbf{2}} \to F_\ivl$ such that the induced functor from $F_{\ivl}$-endofunctor-algebras to $F_{\mathbf{2}}$-endofunctor-algebras forgets the homotopy but remembers the maps.
  And if $\dT_{\mathbf{2}}$ and $\dT_\ivl$ denote the algebraically-free monads on these (accessible) endofunctors, we have an induced monad morphism $\dT_{\mathbf{2}} \to \dT_\ivl$ that has the same effect on algebras.

  Let $\dId$ denote the identity functor with its unique monad structure, so that every object has a unique $\dId$-algebra structure.
  Every object also has a natural (though not unique) $\dT_{\mathbf{2}}$-algebra structure consisting of the two projections, so there is a monad morphism $\dT_{\mathbf{2}} \to \dId$.
  Let $\dtprop$ be the algebraic monad pushout:
  \[
  \begin{tikzcd}
    \dT_{\mathbf{2}} \ar[r] \ar[d] & \dId \ar[d] \\ \dT_{\ivl} \ar[r] & \dtprop
  \end{tikzcd}
  \]
  By definition, this means that a $\dtprop$-algebra structure on an object consists precisely of a $\dT_{\ivl}$-algebra structure and (the unique) $\dId$-algebra structure whose underlying $\dT_{\mathbf{2}}$-algebra structures coincide.
  In other words, a $\dtprop$-algebra is exactly a proposition over $\Gamma$.

  Finally, consider the algebraic monad coproduct $\dtprop + \dR_\Gamma$, and let $\brck A = (\dtprop+\dR_\Gamma)(A)$.
  (More precisely, we define a ``good'' propositional truncation to be a fibration over $\Gamma$ equipped with an isomorphism to $(\dtprop+\dR_\Gamma)(A)$.)
  Then $\brck A \to \Gamma$ is a fibration, since it is an $\dR_\Gamma$-algebra, and $\brck A$ is a proposition since it is a $\dtprop$-algebra, while the unit of the monad $\dtprop + \dR_\Gamma$ supplies $\tr:A\to\brck A$.

  The eliminator data consists of a fibration $p:C\fib \brck A$ that is also a $\dtprop$-algebra morphism, together with a map $t:A\to C$ over $\tr$.
  As usual, we choose an $\dR$-algebra structure on $p$, inducing an $\dR$-algebra structure on the composite $C\fib \brck A \fib \Gamma$ such that $p$ becomes an $\dR$-algebra morphism.
  Thus, $t:A\to C$ is a map from $A$ to a $(\dtprop+\dR_\Gamma)$-algebra, so it factors uniquely through the free $(\dtprop+\dR_\Gamma)$-algebra on $A$, which by definition is $\brck A$.
  And when we compose this factorization $f:\brck A \to C$ with $p$, we get a factorization of $\tr$ through itself, which must therefore be the identity; thus $f$ is a section of $p$ as desired.

  It remains to prove weak stability.
  Since pullback preserves limits and simplicial homotopies, it preserves the introduction rule, i.e.\ it takes propositions to propositions.
  For elimination and computation, as in \cref{thm:wk-w} we deal separately with the cases when $\sigma:\Delta\to\Gamma$ is a fibration or an acyclic cofibration.

  If $\sigma$ is a fibration and $C\fib \sigma^*\brck A$ is a fibration of propositions over $\Delta$, then as usual we have a fibration $\sigma_* C \fib \sigma_* \sigma^* \brck A$ over $\Gamma$.
  Moreover, $\sigma_*$ also preserves propositions, since it is a simplicial right adjoint and hence preserves limits and simplicial homotopies.
  Thus, $\sigma_* C\to\sigma_*\sigma^*\brck A$ is also a fibration of propositions over $\Gamma$, while the unit $\eta : \brck A \to \sigma_* \sigma^* \brck A$ is a map of propositions.
  Thus, the pullback $\eta^* \sigma_* C \fib \brck A$ is also a proposition over $\Gamma$.
  Moreover, if $C$ has a section over $\sigma^*A$, then $\sigma_* C$ has a section over $A$, which therefore extends to a section of $\eta^* \sigma_* C$, i.e.\ a map $\brck A \to \sigma_* C$ over $\eta$.
  This transposes to a section of $C$, as desired.

  On the other hand, if $\sigma$ is an acyclic cofibration, then so is $\brck{A}^*\sigma : \sigma^*\brck A \to \brck A$, since $\brck A \fib \Gamma$ is a fibration.
  Since $\Delta$ is fibrant, so is $\sigma^*\brck A$, and so just as in \cref{thm:wk-w} $\brck{A}^*\sigma$ admits a retraction $r$.
  Now suppose $C\fib \sigma^*\brck A$ is a fibration of propositions over $\Delta$, with a section $c$ over $\sigma^*A$.
  Then $r^*C \fib \brck A$ is a fibration whose pullback along $\sigma$ is $C$, and by lifting in the following square
  \[
  \begin{tikzcd}
    \sigma^*A \ar[d,cof,acyc] \ar[r,"c"] & C \cong \sigma^*r^*C \ar[r] & r^*C \ar[d,fib]\\
    A \ar[rr] && \brck A
  \end{tikzcd}
  \]
  we obtain a section of $r^*C$ over $A$ whose pullback to $\Delta$ is the given $c$.
  Moreover, since $C\cong \sigma^* r^* C$, we have the following commutative square
  \[
  \begin{tikzcd}
    C \ar[d,acyc] \ar[r,cof,acyc] & r^* C \ar[d]\\
    C\times_\Delta C \ar[r,acyc] & r^*C \times_\Gamma r^*C
  \end{tikzcd}
  \]
  in which all maps except the right-hand one are known to be weak equivalences: the left-hand one since $C$ is a proposition over $\Gamma$, the top since it is a pullback of the weak equivalence $\sigma$ along a fibration, and the bottom since it is a homotopy pullback of weak equivalences.
  Thus, by 2-out-of-3 the right-hand map is also a weak equivalence, i.e.\ $r^*C$ is a proposition over $\Gamma$.
  Therefore, by \cref{thm:prop-over}, we can give it the structure of a proposition such that $r^*C \to \brck A$ is a map of propositions.
  So we can apply $\trrec$ to obtain a section of it, which pulls back to the desired section of $C$.
\end{proof}

\section{Cell monads}
\label{sec:cell-monads}

At this point we have enough examples to motivate and formulate a general notion of higher inductive type on the semantic side, although it remains an open problem to give a syntactic presentation of equal generality.
In this section we consider ``higher inductive types'' without parameters; in the next section we will generalize to allow parameters.

We start with a very general notion of ``typal initial algebra'' in the comprehension category context for which we can isolate the conditions necessary to prove the local universes coherence theorem.

\begin{defn}
  For any category $\C$, a \textbf{fibred category of structures} over $\C$ is a comprehension category structure $(\C,\S)$ such that the functor $\S\to\C^\to$ is a faithful amnestic isofibration.\footnote{Recall that a functor $U:\cA\to\cB$ is \emph{amnestic} if whenever $f:a\cong b$ is an isomorphism in $\cA$ such that $U a = U b$ and $U f = 1_{U a}$, then also $a = b$ and $f = 1_a$.
  It is an \emph{isofibration} if for any isomorphism $g:U a \cong b$, there exists an isomorphism $f:a\cong a'$ such that $U a' = b$ and $U f = g$.}
  We call an object of $\S$ an \textbf{\S-algebra}, and we call a lifting of $(X\to \Gamma)\in\C^\to$ to \S an \textbf{\S-structure} on it.
  We call a morphism in \S an \textbf{\S-morphism}.
\end{defn}

Note that by the assumptions on $\S$, the collection of \S-structures on any $X\to \Gamma$ is a partially ordered set (not just a preordered set).

\begin{defn}
  Let $(\C,\T)$ be a comprehension category, and let \S be a fibred category of structures over \C (note that there is no relation between \S and \T).
  A \textbf{typal initial $\S$-algebra} over $\Gamma$ is a type $H\in \T(\Gamma)$ together with an \S-structure on its comprehension $\Gamma.H\to\Gamma$, such that for any $C\in\T(\Gamma.H)$ together with an \S-algebra structure on the composite $\Gamma.H.C\to\Gamma.H\to\Gamma$ such that $\Gamma.H.C\to\Gamma.H$ is an \S-morphism, there exists a section $\Gamma.H \to \Gamma.H.C$ of $\Gamma.H.C\to\Gamma.H$ that is also an \S-morphism.\mswarning{This is naturally phrased using fibrations!}

  We say $\C$ has \textbf{weakly stable typal initial \S-algebras} if for any $\Gamma$ there exists a typal initial $\S$-algebra over $\Gamma$ whose reindexing along any $\sigma:\Delta\to\Gamma$ is again a typal initial $\S$-algebra.
  If $\C$ is split, we say it has \textbf{strictly stable typal initial \S-algebras} if we have an operation assigning to each $\Gamma$ a typal initial $\S$-algebra over $\Gamma$ with specified $\S$-algebra sections, in a way that is strictly preserved by the split reindexing functors.
\end{defn}

\begin{defn}\label{defn:replift}
  Let $(\C,\T)$ be a comprehension category and \S a fibred category of structures over \C.
  Given $A\in\T(\Gamma)$ and $B\in\T(\Gamma.A)$ and any \S-structure $\sA$ on $\Gamma.A\to\Gamma$, an \textbf{\S-lift} of $\sA$ to $B$ is an \S-structure on the composite $\Gamma.A.B\to\Gamma.A\to\Gamma$ such that $\Gamma.A.B\to\Gamma.A$ becomes an \S-morphism over $\Gamma$.
  We say that \S has \textbf{representable lifts} if for any $A,B,\sA$, the functor $(\C/\Gamma)\op \to \mathbf{Set}$ defined by
  \begin{equation}\label{eq:replift}
    (\sigma:\Delta\to\Gamma) \mapsto \{ \text{\S-lifts of $\sA[\sigma]$ to $B[\sigma]$}\}
  \end{equation}
  is representable.
  In other words, there exists a map $\varpi:V_{\sA,B}\to\Gamma$ and an $\S$-lift $\sB$ of $\sA[\varpi]$ to $B[\varpi]$, such that for any $\sigma:\Delta\to\Gamma$ and any $\S$-lift $\sB'$ of $\sA[\sigma]$ to $B[\sigma]$, there is a unique map $\tau:\Delta\to V_{\sA,B}$ such that $\varpi \circ \tau = \sigma$ and $\sB' = \sB[\tau]$.
\end{defn}

Note that the notion of representable lift \emph{does} involve both \S and \T, and also that the assumptions on \S ensure that the functor~\eqref{eq:replift} really does sensibly take values in (partially ordered) \emph{sets} (rather than groupoids or categories).

\begin{lem}
  Let $(\C,\T)$ be a comprehension category and \S a fibred category of structures over \C.
  If \C has weakly stable typal initial \S-algebras, \C has a terminal object, and \S has representable lifts, then $\C_!$ has strictly stable typal initial \S-algebras.
\end{lem}
\begin{proof}
  Choose a weakly stable typal initial \S-algebra $H\in\T(1)$ with \S-structure \sH, and let $V_H = 1$ and $E_H = H$ with $\name{H}= 1$.
  Then assign to each $\Gamma$ the algebra $(V_H,E_H,!)$; this is strictly stable by composition of names, as usual.
  That is, since the ``formation'' and ``introduction'' rules (which jointly correspond to a choice of \S-structure on the comprehension of a type) have no premises, their local universe is the terminal object.

  For the elimination rule, given a type $C\in \T_!(\Gamma.H[\Gamma])$ determined by $(V_C,E_C,\name{C})$, consider first the object $V = V_H\triangleleft V_C$ constructed as in~\cite{lw:localuniv}, which classifies maps $\Delta \to V_H = 1$ (which are unique) together with maps $\Delta.H[\Delta] \to V_C$.
  In particular, we have a universal map $c:V.H[V]\to V_C$ giving $E_C[c]\in \T(V.H[V])$.

  Now, by representable lifts, we have a map $V_{\sH[V],E_C[c]} \to V$ representing lifts of the \S-structure $\sH[V]$ of $H[V]$ to $E_C[c]$.
  Combining universal properties, we see that maps $\Gamma\to V_{\sH[V],E_C[c]}$ correspond to choices of a map $\name{C}:\Gamma.H[\Gamma]\to V_C$ together with a lift of $\sH[\Gamma]$ to $E_C[\name{C}]$.
  This is exactly the input to the ``elimination and computation rules'' of a typal initial \S-algebra, so we take $V_{\sH[V],E_C[c]}$ as the local universe for that.
  The weak stability of $H$ implies that the universal lift over $V_{\sH[V],E_C[c]}$ has a section that is an \S-morphism, which we can then reindex to $\Gamma$ to obtain a strictly stable eliminator.
\end{proof}

Note that we did not yet need to explicitly assume all of~\eqref{eq:lf}; the assumption of representable lifts encodes all of this that's relevant to \S.

Now we move on to good model categories.
In this case, our work in previous sections suggests that we should consider structures defined as the algebras for a monad.
To obtain a \emph{fibred} category of algebras, we need to consider fibred monads as well.

\begin{defn}
  A \textbf{fibred monad} on a category \C with pullbacks is a monad $\dT$ on $\C^\to$ in the category of fibrations over \C.
  That is, $\dT$ and its unit $\eta$ and multiplication $\mu$ live in the (strict) slice 2-category of \cCat over $\C$, meaning that $\dT$ preserves codomains and the codomain-part of $\eta$ and $\mu$ are the identity, and also $\dT$ preserves cartesian morphisms.
\end{defn}

In particular, such a \dT induces a monad $\dT_\Gamma$ on each slice category $\C/\Gamma$, and each pullback functor $\sigma^* : \C/\Gamma \to \C/\Delta$ is a \emph{strong monad morphism}, i.e.\ it satisfies $\sigma^* \dT_\Gamma \cong \dT_\Delta \sigma^*$ coherently.
That is, a fibred monad is equivalently an \emph{indexed} monad.

Let $\dT\alg$ denote the category of \dT-algebras, which is again fibred over \C; the fiberwiseness of $\eta$ means that the fiber of $\dT\alg$ over $\Gamma$ is $\dT_\Gamma\alg$.
We record:

\begin{lem}
  For any fibred monad \dT on a category \C with pullbacks, $\dT\alg$ is a fibred category of structures over \C.\qed
\end{lem}

\begin{lem}
  If \dT is a fibred monad on a model category \sM, then the full subcategory $\dT\algf$ of $\dT\alg$ consisting of \dT-algebra structures on fibrations over fibrant objects is a fibred category of structures over the full subcategory of fibrant objects \Mf.\qed
\end{lem}

The question, therefore, is what conditions we can impose on \dT to ensure that
\ifmpcps
  \begin{enumerate} \renewcommand{\theenumi}{(\arabic{enumi})}
\else
  \begin{enumerate}[label=(\arabic*)]
\fi
\item \fibmf has weakly stable initial $\dT\algf$-structures, and
\item $\dT\algf$ has representable lifts.
\end{enumerate}
It is familiar from homotopy theory that an algebraic theory (such as an operad or, in this case, a monad) must be sufficiently ``cofibrant'' to have good homotopical behavior.
Cofibrant objects, in turn, are generally constructed as ``cell complexes''.
Inspecting our proofs of weak stability in the preceding sections leads us to the following definitions.

\begin{lem}\label{thm:free-fibred}
  Let \sM be a locally presentable category, and $F$ a fibred endofunctor of \sM such that each fiber $F_\Gamma$ is accessible.
  Let $\dT_\Gamma$ be the algebraically-free monad on $F_\Gamma$.
  Then these monads $\dT_\Gamma$ are also indexed, i.e.\ we have coherent isomorphisms $\dT_\Delta \circ \sigma^* \cong \sigma^* \circ \dT_\Gamma$ commuting with the monad structures, and hence induce a fibred monad $\dT$.
\end{lem}
\begin{proof}
  As explained in~\cite{kelly:transfinite,nlab:transfinite}, $\dT_\Gamma$ is constructed out of colimits in \sM using a general method that depends on $F_\Gamma$ only insofar as we have to carry it out to some stage $\kappa$ such that $F_\Gamma$ is $\kappa$-accessible.
  It doesn't matter whether we go ``too far'', since the construction converges: once we reach a $\kappa$ such that $F_\Gamma$ is $\kappa$-accessible, continuing to larger values of $\kappa$ doesn't change the result.
  Therefore, given $\sigma:\Delta\to\Gamma$, we can choose $\kappa$ such that $F_\Gamma$ and $F_\Delta$ are \emph{both} $\kappa$-accessible, and then both $\dT_\Gamma$ and $\dT_\Delta$ can be constructed by \emph{exactly the same colimits}, only involving $F_\Gamma$ and $F_\Delta$ respectively.
  Since $\sigma^*$ preserves all colimits, the isomorphisms $F_\Delta \circ \sigma^* \cong \sigma^* \circ F_\Gamma$ therefore yield isomorphisms $\dT_\Delta \circ \sigma^* \cong \sigma^* \circ \dT_\Gamma$, and so on.
\end{proof}

Now let \sM be an excellent model category.
For any fibration $f:B\to A$ of fibrant objects in \sM and any $\Gamma\in\sM$, let $F^f_\Gamma$ be the polynomial endofunctor of $\sM/\Gamma$ determined by $\Gamma^*(f)$, i.e.\ the composite
\[ \sM/\Gamma \xto{\Gamma^*(B)^*} \sM/\Gamma^*(B) \xto{\Gamma^*(f)_*} \sM/\Gamma^*(A) \xto{\Gamma^*(A)_!} \sM/\Gamma \]
By the Beck-Chevalley condition for dependent exponentials, $F^f_\Gamma$ is an indexed endofunctor, hence we have an fibred endofunctor $F^f$, which is moreover (fiberwise) accessible.

Next, for any simplicial set $K$ we have a further fibred endofunctor $F^f \otimes K$, which is again accessible.
We denote the free fibred monad it generates by $\dT^{f,K}$.
Finally, for any morphism $i:K\to L$ of simplicial sets, we have an induced map $F^f \otimes K \to F^f \otimes L$, which generates a map of fibred monads $\dT^{f,K}\to \dT^{f,L}$.

\begin{defn}
  Let \sM be an excellent model category.
  \begin{itemize}
  \item A \textbf{monad cell} is a fibred monad morphism $\dT^{f,K} \to \dT^{f,L}$ obtained as above from a fibration $f:B\to A$ of fibrant objects in \sM and a \emph{cofibration} $i:K\to L$ of simplicial sets.
  \item A \textbf{(finite) relative cell monad} is a fibred monad morphism obtained as a (finite) composite of pushouts of monad cells.
  \item A \textbf{(finite) cell monad} is a fibred monad \dT such that the unique map from the initial fibred monad (which is the identity $\dId$) is a (finite) relative cell monad.
  \end{itemize}
\end{defn}

As in previous sections, the proof of weak stability for typal initial \dT-algebras will proceed by considering separately the cases when $\sigma$ is a fibration or an acyclic cofibration.
For the fibration case, the relevant lemma will be the following.

\begin{lem}\label{thm:fibmnd-radj}
  For any fibred monad \dT on a good model category \sM, and any morphism $\sigma:\Delta\to\Gamma$, the adjunction $\sigma^*: \sM/\Gamma \toot \sM/\Delta : \sigma_*$ lifts to an adjunction $\sigma^*:\dT_\Gamma\alg \to \dT_\Delta\alg: \sigma_*$.
\end{lem}
\begin{proof}
  Since $\sigma^*$ is a strong monad morphism, by doctrinal adjunction~\cite{kelly:doc-adjn} the entire adjunction $\sigma^*\dashv\sigma_*$ lifts to the 2-category of categories-with-monads and lax monad morphisms.
  But the construction of categories of algebras is functorial on this 2-category, hence takes this adjunction to the desired adjunction between categories of algebras.
\end{proof}

For the acyclic cofibration case, the relevant lemma will be the following analogue of \cref{thm:W-cofmnd}.

\begin{lem}\label{thm:cell-cofmnd}
  Let $\dS\to\dT$ be a relative cell monad on an excellent model category, and suppose given a pair of pullback squares:
  \[
  \begin{tikzcd}
    C \ar[r,cof,acyc',"i"] \ar[d,fib,"p",swap] \ar[dr,phantom,near start,"\lrcorner"] & D \ar[d,fib,"q"] \\
    V \ar[r,cof,acyc',"j"] \ar[d,fib] \ar[dr,phantom,near start,"\lrcorner"] & W \ar[d,fib] \\
    \Delta \ar[r,cof,acyc',"\sigma"] & \Gamma
  \end{tikzcd}
  \]
  in which all objects are fibrant, the downward-pointing arrows are fibrations, and the rightward-pointing arrows are acyclic cofibrations.
  Suppose moreover that:
  \begin{enumerate}
  \item $C$ and $V$ are $\dT_\Delta$-algebras and $p$ is a $\dT_\Delta$-morphism,
  \item $W$ is a $\dT_\Gamma$-algebra and $j$ is a cartesian morphism in $\dT\alg$,
  \item $D$ is an $\dS_\Gamma$-algebra and $q$ is an $\dS_\Gamma$-morphism (where $W$ has its induced $\dS_\Gamma$-structure), and
  \item $i$ is a cartesian morphism in $\dS\alg$ (where $C$ has its induced $\dS_\Delta$-structure).
  \end{enumerate}
  Then there is a $\dT_\Gamma$-structure on $D$ such that $q$ is a $\dT_\Gamma$-morphism and $i$ is a cartesian morphism in $\dT\alg$.
\end{lem}
\begin{proof}
  It is easy to check that this property is preserved by algebraic pushouts and (possibly transfinite) composites of monad morphisms.
  Thus, it suffices to show that monad cells $\dT^{f,K} \to \dT^{f,L}$ have this property.
  In this case, generalizing the argument of \cref{thm:W-cofmnd}, we have the following commutative diagram of solid arrows and we want to construct a dashed arrow making the other squares commute:
  \[ \mathclap{
  \begin{tikzcd}[ampersand replacement=\&]
    (\sigma^*f)_* (\sigma^*B)^* C \otimes K \ar[rrr] \ar[dr] \ar[ddd] \&\&\& f_* B^* D \otimes K \ar[dr] \ar[ddd] \ar[ddrr,bend left]\\
    \&(\sigma^*f)_* (\sigma^*B)^* C \otimes L \ar[rrr,crossing over] \ar[dr] \&\&\& f_* B^* D \otimes L \ar[dr,dashed] \ar[ddd]\\
    \&\& C \ar[rrr,->,near end,crossing over] \&\&\& D \ar[ddd,fib] \\
    (\sigma^*f)_* (\sigma^*B)^* V \otimes K \ar[rrr] \ar[dr] \&\&\& f_* B^* W \otimes K \ar[dr] \\
    \&(\sigma^*f)_* (\sigma^*B)^* V \otimes L \ar[rrr] \ar[dr] \ar[from=uuu,crossing over] \&\&\& f_* B^* W \otimes L \ar[dr] \\
    \&\& V \ar[rrr,->] \ar[from=uuu,crossing over] \&\&\& W
  \end{tikzcd}
  } \]
  This is equivalent to solving the following lifting problem:
  \[
  \begin{tikzcd}
    ((\sigma^*f)_* (\sigma^*B)^* C \otimes L) \sqcup_{((\sigma^*f)_* (\sigma^*B)^* C \otimes K)} (f_* B^* D \otimes K) \ar[r] \ar[d] & D \ar[d,fib] \\
    f_* B^* D \otimes L \ar[ur,dashed] \ar[r] & W
  \end{tikzcd}
  \]
  But the left-hand map is now the pushout product of $(\sigma^*f)_* (\sigma^*B)^* C \otimes L \to f_* B^* D$, which we showed in \cref{thm:W-cofmnd} to be an acyclic cofibration, and the cofibration $K\to L$ of simplicial sets; thus it is also an acyclic cofibration.
\end{proof}

\begin{thm}\label{thm:cell-wkstab}
  If \dT is a cell monad on an excellent model category \sM, then \fibmf has weakly stable typal initial $\dT\algf$-algebras.
\end{thm}
\begin{proof}
  As usual, for any $\Gamma$, let $H\fib\Gamma$ be the initial $(\dT_\Gamma+\dR_\Gamma)$-algebra.
  Then it is a fibration and a $\dT_\Gamma$-algebra.
  And given any fibration $C\fib H$ that is a $\dT_\Gamma$-morphism, choose an \dR-structure on it to make $C\to\Gamma$ a $(\dT_\Gamma+\dR_\Gamma)$-algebra and $C\to H$ a $(\dT_\Gamma+\dR_\Gamma)$-morphism, so that it has a $(\dT_\Gamma+\dR_\Gamma)$-section.
  Thus, $H$ is a typal initial $\dT\algf$-algebra.

  To show that it is weakly stable, by factorization we consider separately the cases when $\sigma:\Delta\to\Gamma$ is a fibration or an acyclic cofibration.
  If $\sigma$ is a fibration, then for any $\dT_\Delta$-algebra fibration $C\fib \sigma^*H$, by \cref{thm:fibmnd-radj} we have a $\dT_\Gamma$-algebra map $\sigma_*C\fib \sigma_*\sigma^*H$, which is a fibration since $\sigma$ is a fibration.
  Thus, pulling it back along the unit $H\to \sigma_*\sigma^* H$, we obtain a $\dT_\Gamma$-algebra fibration over $H$, which therefore has a $\dT_\Gamma$-algebra section.
  This gives a map $H\to \sigma_* C$, whose transpose under the adjunction of \cref{thm:fibmnd-radj} is the desired $\dT_\Delta$-algebra section $\sigma^*H \to C$.

  If $\sigma$ is an acyclic cofibration, then so is its pullback $H^*(\sigma) : \sigma^*H \to H$ along the fibration $H\fib \Gamma$.
  Since $\Delta$ is fibrant, so is $\sigma^*H$, and thus $H^*(\sigma)$ has a retraction $r$.
  Now for any $\dT_\Delta$-algebra fibration $C\fib \sigma^*H$, let $D = r^*C$, so that $C \cong (H^*(\sigma))^*(D)$.
  By \cref{thm:cell-cofmnd} applied to the relative cell monad $\dId\to\dT$, we can find a $\dT_\Gamma$-algebra structure on $D$ making $D\to H$ a $\dT_\Gamma$-algebra fibration whose pullback is $C\to\sigma^* H$.
  Thus, applying the eliminator for $H$ we get a $\dT_\Gamma$-algebra section of $D$, which pulls back to a $\dT_\Delta$-algebra section of $C$.
\end{proof}

Finally, we also have:

\begin{thm}\label{thm:cell-replift}
  If \dT is a cell monad on an excellent model category \sM, then $\dT\algf$ has representable lifts.
\end{thm}
\begin{proof}
  To make the induction go through, we will prove the following stronger result.
  Let $\dS\to\dT$ be a \emph{relative} cell monad and let $\Gamma$ be a fibrant object.
  For any fibrations $B\fib A\fib \Gamma$ such that $A\fib \Gamma$ is a $\dT_\Gamma$-algebra and $B\fib\Gamma$ is an $\dS_\Gamma$-algebra and $B\to A$ is an $\dS_\Gamma$-algebra map, we will show that there is a representing object $V$ for extensions of the \dS-algebra structure on $B$ to a \dT-algebra structure making $B\to A$ into a \dT-algebra map, and that moreover the map $V\to\Gamma$ is a fibration.

  In particular, this property implies that $V$ is fibrant, as it must be since we are considering $\dT\algf$ rather than $\dT\alg$.
  But the stronger assumption that $V\to\Gamma$ is a fibration ensures that the property is preserved by both algebraic pushouts and (possibly transfinite) composites of monad morphisms.
  In the latter case we take a (possible transfinite) inverse composite of the representing fibrations.

  Thus, it suffices to show that the property holds for monad cells $\dT^{f,K} \to \dT^{f,L}$.
  In other words, we must show that if $A\fib \Gamma$ is an $(F^f\otimes L)$-algebra and $B\fib\Gamma$ is an $(F^f\otimes K)$-algebra and $B\to A$ is a $(F^f\otimes K)$-algebra map, there is a representing object for extensions of the $(F^f\otimes K)$-algebra structure on $B$ to a $(F^f\otimes L)$-algebra structure making $B\to A$ into a $(F^f\otimes L)$-algebra map.

  Such an extension is a dotted arrow filling in the following diagram:
  \[
  \begin{tikzcd}
    F^f_\Gamma(B) \otimes K \ar[rr] \ar[dr] \ar[dd] && B \ar[dd,fib]\\
    & F^f_\Gamma(B) \otimes L \ar[ur,dashed] \\
    F^f_\Gamma(A) \otimes K \ar[rr] \ar[dr] && A\\
    & F^f_\Gamma(A) \otimes L \ar[ur] \ar[from=uu,crossing over]
  \end{tikzcd}
  \]
  We define the representing object to be the following pullback:
  \[
  \begin{tikzcd}[column sep=huge]
    V \ar[r] \ar[d,fib] \ar[dr,phantom,near start,"\lrcorner"] & \left(B^{F^f_\Gamma(B)}_\Gamma\right)^L_\Gamma \ar[d,fib]\\
    \Gamma \ar[r] &
    \left(B^{F^f_\Gamma(B)}_\Gamma\right)^K_\Gamma \times_{\left(A^{F^f_\Gamma(B)}_\Gamma\right)^K_\Gamma} \left(A^{F^f_\Gamma(B)}_\Gamma\right)^L_\Gamma
  \end{tikzcd}
  \]
  Here $B^{F^f_\Gamma(B)}_\Gamma$ and $A^{F^f_\Gamma(B)}_\Gamma$ denote exponentials in $\sM/\Gamma$, while $(-)^K_\Gamma$ and $(-)^L_\Gamma$ denote simplicial powers in $\sM/\Gamma$.
  Of course, $(X^Y_\Gamma)^K_\Gamma \cong X^{Y\otimes K}_\Gamma$, but we have written it using the former notation to make it clear that the right-hand map is the pullback corner map for the cofibration $K\to L$ and the map $B^{F^f_\Gamma(B)}_\Gamma \to A^{F^f_\Gamma(B)}_\Gamma$, which is a fibration since $B\fib A$ and $F^f_\Gamma(B) \fib \Gamma$ are.
  Hence the right-hand map is a fibration, and thus so is its pullback on the left.

  The bottom map is determined by the solid arrows in the previous diagram, and a choice of dotted arrow in the previous diagram would be equivalently a lift of this map to $\left(B^{F^f_\Gamma(B)}_\Gamma\right)^L_\Gamma$, and hence also equivalently a section of $V\fib\Gamma$.
  Moreover, the entire diagram lives in $\sM/\Gamma$ and is preserved by pullback; thus a choice of such a lift after pullback along $\sigma$ is the same as a factorization of $\sigma$ through $V$.
  This says exactly that $V$ has the desired universal property.
\end{proof}

\begin{cor}
  If \sM is an excellent model category and \dT is a cell monad on it, then $\fibmfbang$ has strictly stable typal initial $\dT\algf$-algebras.\qed
\end{cor}

We refrain from discussing examples until after the generalization to monads with parameters in the next section.

\begin{rmk}
  There is a natural generalization of a cell monad that allows cells generated by endofunctor morphisms $F^f\times K \to F^f \times L$ for any fibration $f$ and any cofibration $K\to L$ in \sM, rather than requiring $K\to L$ to be a cofibration of simplicial sets.
  (This is more general since $1\otimes K \to 1\otimes L$ is a cofibration in \sM for any cofibration $K\to L$ of simplicial sets, and $X\times (1\otimes K) \cong X\otimes K$.)
  The above proofs go through for this generalization as long as the cofibration $K\to L$ satisfies the pushout-product property in \sM, such as if \sM is itself a cartesian monoidal model category.
  Similarly, we could replace simplicial sets by another suitable monoidal model category over which \sM is enriched.
\end{rmk}

\section{Cell monads with parameters}
\label{sec:cell-monads-param}

Now we generalize the construction of \cref{sec:cell-monads} to allow parameters, i.e.\ premises for the formation and introduction rules.
Because our type theory has no universes, we have to treat type parameters separately from term parameters.
Intuitively, a ``parameter scheme'' is a list of hypothesized judgments in type theory; but in this section we take a semantic point of view instead.

Recall that in a comprehension category $(\C,\T)$, a \textbf{dependent projection} is a map that is isomorphic to the comprehension of some (unspecified) type $\Gamma.A\to\Gamma$, and a \textbf{display map} is a finite composite of dependent projections.
All pullbacks of dependent projections and display maps exist and are again of the same sort.
We view a display map as a semantic version of a \emph{context extension}.
Let $\D_1$ and $\D$ be the categories whose objects are dependent projections and display maps, respectively, and whose morphisms are pullback squares; then the codomain projections $\D_1\to\C$ and $\D\to\C$ are fibrations with groupoid fibers.

Let $\T_{\cong}$ be the subcategory of $\T$ containing only the cartesian arrows; then $\T_{\cong}\to\C$ is also a fibration with groupoid fibers.
Finally, let $\D_{1,*}$ be the category whose objects are dependent projections equipped with a section, and whose morphisms are cartesian ones commuting with the sections.
The forgetful functor $\D_{1,*} \to \D_1$ is a discrete fibration.

\begin{defn}\label{defn:param-scheme}
  Let $(\C,\T)$ be a comprehension category.
  We define the set of \textbf{parameter schemes over $(\C,\T)$} together with, for every such parameter scheme $\P$, a fibration $\cInst(\P)\to\C$ of \textbf{instantiations of \P}, mutually inductively as follows.
  \begin{itemize}
  \item There is an \textbf{empty parameter scheme} $\emptyps$, and $\cInst(\emptyps) = \C$ (i.e.\ there is a unique instantiation of $\emptyps$ over every $\Gamma\in\C$).
  \item If \P is a parameter scheme, an \textbf{extension of \P by a type parameter} is a new parameter scheme $\typeps{\P}{\al}$ determined by a cartesian functor $\al:\cInst(\P)\to\D$ over \C.
    Its instantiations are determined by the following pullback:
    \[
    \begin{tikzcd}
      \cInst(\typeps{\P}{\al}) \ar[rr] \ar[d] \pullback{dr} && \T_{\cong} \ar[d,"\mathrm{cod}"] \\
      \cInst(\P) \ar[r,"\al"] \ar[d] & \D \ar[r,"\mathrm{dom}"] & \C\\
      \C
    \end{tikzcd}
    \]
  \item If \P is a parameter scheme, an \textbf{extension of \P by a term parameter} is a new parameter scheme $\termps\P\al\be$ determined by a cartesian functor $\al:\cInst(\P)\to\D$ over \C and a cartesian functor $\be:\cInst(\P) \to \D_1$ over $\cInst(\P) \xto{\al} \D \xto{\mathrm{dom}} \C$.
    Its instantiations are determined by the following pullback:
    \[
    \begin{tikzcd}
      \cInst(\termps{\P}{\al}{\be}) \ar[rr] \ar[d] \pullback{dr} & & \D_{1,*} \ar[d] \\
      \cInst(\P) \ar[dr,"\al"] \ar[d] \ar[rr,"\be"] & \phantom{.} & \D_1 \ar[d,"\mathrm{cod}"]\\
      \C & \D \ar[r,"\mathrm{dom}"] & \C
    \end{tikzcd}
    \]
  \end{itemize}
\end{defn}

\begin{rmk}{}
  The above definition could be made more general; as stated it essentially assumes that the parameters are defined using only ``pseudo-stable structure'', but we could in principle allow weakly stable structure as well.
  However, it is hard to think of examples requiring the more general version, so we refrain from introducing the complication.
\end{rmk}

Intuitively, $\al$ assigns the context extension in which a type or term parameter lives, while $\be$ assigns the type to which a term parameter belongs (which can depend on the context extension).
For instance, if a type parameter is just $\Gamma\types B\type$, then \al will assign to every instantiation over $\Gamma$ the terminal display map $1_\Gamma$.
On the other hand, if the parameter scheme \P already contains a type parameter $A$, then the additional type parameter $\Gamma,x:A \types B\type$ will have a functor $\al$ that assigns to each instantiation over $\Gamma$ the comprehension of the corresponding type $A$.

\begin{defn}
  A \textbf{fibred category of structures with parameters} over a comprehension category $(\C,\T)$ consists of a parameter scheme $\P$, a fibration $\S\to \C$, and a faithful amnestic isofibration $\S \to \cInst(\P)\times_C C^\to$ over \C.
\end{defn}

If $\Theta$ is an instantiation of \P over $\Gamma$, the objects of the fiber over $(\Theta,X\to\Gamma)$ will be called \textbf{$\S(\Th)$-structures on $X$}; as before they form a partially ordered set.

\begin{defn}
  Let $(\C,\T)$ be a comprehension category and \S a fibred category of structures with parameters over \C.
  For any instantiation $\Th$ of the parameter scheme of \S over $\Gamma\in \C$, a \textbf{typal initial $\S(\Th)$-algebra} consists of a type $H\in\T(\Gamma)$ and an $\S(\Th)$-structure on $\Gamma.H\to \Gamma$, such that for any $C\in\T(\Gamma.H)$ together with an $\S(\Th)$-structure on the composite $\Gamma.H.C\to\Gamma.H\to\Gamma$ such that $\Gamma.H.C\to\Gamma.H$ is an $\S(\Th)$-morphism, there exists a section $\Gamma.H \to \Gamma.H.C$ of $\Gamma.H.C\to\Gamma.H$ that is also an \S-morphism.

  We say $\C$ has \textbf{weakly stable typal initial \S-algebras} if for any $\Gamma$ and $\Th$ there exists a typal initial $\S(\Th)$-algebra whose reindexing along any $\sigma:\Delta\to\Gamma$ is a typal initial $\S(\Theta[\sigma],\theta[\sigma])$-algebra.
  If $\C$ is split, we say it has \textbf{strictly stable typal initial \S-algebras} if we have an operation assigning to each $\Gamma$ and $\Th$ a typal initial $\S(\Th)$-algebra over $\Gamma$ with specified $\S$-algebra sections, in a way that is strictly preserved by the split reindexing functors.
\end{defn}

\begin{defn}\label{defn:replift-param}
  Let $(\C,\T)$ be a comprehension category and \S a fibred category of structures with parameters over \C.
  Given $A\in\T(\Gamma)$ and $B\in\T(\Gamma.A)$ and an instantiation $\Th$ of the parameter scheme over $\Gamma$, and any $\S(\Th)$-structure $\sA$ on $\Gamma.A\to\Gamma$, an \textbf{\S-lift} of $\sA$ to $B$ is an $\S(\Th)$-structure on the composite $\Gamma.A.B\to\Gamma.A\to\Gamma$ such that $\Gamma.A.B\to\Gamma.A$ becomes an $\S(\Th)$-morphism over $\Gamma$.
  We say that \S has \textbf{representable lifts} if for any $A,B,\Th,\sA$, the functor $(\C/\Gamma)\op \to \mathbf{Set}$ defined by
  \begin{equation*}
    (\sigma:\Delta\to\Gamma) \mapsto \{ \S\text{-lifts of $\sA[\sigma]$ to $B[\sigma]$}\}
  \end{equation*}
  is representable.
  In other words, there exists a map $\varpi:V_{\sA,B}\to\Gamma$ and an $\S$-lift $\sB$ of $\sA[\varpi]$ to $B[\varpi]$, such that for any $\sigma:\Delta\to\Gamma$ and any $\S$-lift $\sB'$ of $\sA[\sigma]$ to $B[\sigma]$, there is a unique map $\tau:\Delta\to V_{\sA,B}$ such that $\varpi \circ \tau = \sigma$ and $\sB' = \sB[\tau]$.
\end{defn}

We are almost ready to prove the local universes coherence theorem, but first we need to address the difference between parameters over $\C$ and parameters over $\C_!$.
Note that the dependent projections and display maps of $\C_!$ are exactly the same as those of \C.

\begin{prob}
  Every parameter scheme \P over $(\C,\T)$ gives rise to a parameter scheme $\P_!$ over $(\C,\T_!)$ together with a cartesian functor $\rho:\cInst(\P_!)\to\cInst(\P)$ over \C.
\end{prob}
\begin{constr}
  By induction on \P.
  Of course, $\emptyps_!=\emptyps$.
  We define $\typeps{\P}{\al}_! = \typeps{\P_!}{\al\circ \rho}$ and $\termps{\P}{\al}{\be}_! = \termps{\P_!}{\al\circ \rho}{\be\circ \rho}$.
  The well-typedness of the latter for $\T_!$ is why we defined $\be$ to take values in $\D_1$ rather than $\T_{\cong}$.
  The maps $\rho$ in the two inductive steps are defined by the following diagrams:
    \[
    \begin{tikzcd}
      \cInst(\typeps{\P}{\al}_!) \ar[dr,"\rho"] \ar[rr] \ar[dd] && (\T_!)_{\cong} \ar[dr] \ar[ddr] \\
      & \cInst(\typeps{\P}{\al}) \ar[rr,crossing over] \ar[d] \pullback{dr} && \T_{\cong} \ar[d,"\mathrm{cod}"] \\
      \cInst(\P_!) \ar[r,"\rho"] \ar[dr] & \cInst(\P) \ar[r,"\al"] \ar[d] & \D \ar[r,"\mathrm{dom}"] & \C\\
      & \C
    \end{tikzcd}
    \]
    \[
    \begin{tikzcd}
      \cInst(\termps{\P}{\al}{\be}_!) \ar[r,"\rho"] \ar[d] & \cInst(\termps{\P}{\al}{\be}) \ar[rr] \ar[d] \pullback{dr} & & \D_{1,*} \ar[d] \\
      \cInst(\P_!) \ar[r,"\rho"] \ar[dr] & \cInst(\P) \ar[dr,"\al"] \ar[d] \ar[rr,"\be"] & \phantom{.} & \D_1 \ar[d,"\mathrm{cod}"]\\
      & \C & \D \ar[r,"\mathrm{dom}"] & \C
    \end{tikzcd}
    \]
\end{constr}

\begin{defn}
  If \S is a fibred category of structures with parameters over $\C$, with parameter scheme \P, we define $\S_!$ to be the fibred category of structures with parameter scheme $\P_!$ over $C_!$ obtained by pullback along $\rho:$
  \[
  \begin{tikzcd}
    \S_! \ar[r] \ar[d] \pullback{dr} & \S \ar[d] \\
    \cInst(\P_!)\times_C \C^\to \ar[r,"\rho\times 1"] & \cInst(\P)\times_C \C^\to
  \end{tikzcd}
  \]
\end{defn}

\begin{thm}\label{thm:lu-str-param}
  Let $(\C,\T)$ be a comprehension category and \S a fibred category of structures with parameters over \C.
  If \C has weakly stable typal initial \S-algebras, \C satisfies~\eqref{eq:lf}, and \S has representable lifts, then $\C_!$ has strictly stable typal initial $\S_!$-algebras.
\end{thm}
\begin{proof}
  Suppose given $\Gamma\in\C$ and an instantiation $\Th$ over $\Gamma$ of the parameter scheme $\P_!$ of $\S_!$.
  By induction on $\P$, we construct a local universe $V_\Th$ for the formation and introduction rules (i.e.\ the type $H\in\T(\Gamma)$ and its $\S(\Th)$-structure), an instantiation $\Xi\in\cInst(\P_!)_{V_\Th}$, and a map $\name{\Th} : \Gamma\to V_\Th$ such that $\Th = \Xi[\name{\Th}]$.

  In the case of $\emptyps$, we take $V_\Th = 1$, and $\Xi$ and $\name{\Th}$ are unique.

  In the case of $\typeps{\P}{\al}$, we have inductively $\Th\in\cInst(\P_!)$ along with $V_\Th$, $\Xi$, and $\name{\Th}$, and also $\al:\cInst(\P)\to\D$ over \C and (as the additional data extending $\Th$ to a $\T_!$-instantiation of $\typeps{\P}{\al}_!$) a type $A\in\T_!(\mathrm{dom}(\al(\rho(\Th))))$, determined as usual by $(V_A,E_A,\name{A})$.
  Let $\pi_A : V_{(\Th,A)}\to V_\Th$ be the local exponential 
  \[ (V_\Th \times V_A \to V_\Th)^{\al(\rho(\Xi))} \]
  in $\C/V_\Th$, which exists by~\eqref{eq:lf}.
  Its comes with a universal evaluation map $V_{(\Th,A)} \times_{V_\Th} \al(\rho(\Xi)) \to V_A$.
  But $V_{(\Th,A)} \times_{V_\Th} \al(\rho(\Xi))$ is (isomorphic to) the domain of $\al(\rho(\Xi[\pi_A]))$, so together with $V_A$ and $E_A$ this evaluation map determines a type in $T_!$ over this domain, which is exactly what we need to extend $\Xi[\pi_A]$ to an instantiation of $\typeps{\P}{\al}_!$ over $V_{(\Th,A)}$.
  Finally, the universal property of $V_{(\Th,A)}$ says that $\name{(\Th,A)} : \Gamma \to V_{(\Th,A)}$ can be determined by a map $\Gamma \times_{V_\Th} \al(\rho(\Xi)) \to V_A$, but
  \[\Gamma \times_{V_\Th} \al(\rho(\Xi)) \cong \mathrm{dom}(\al(\rho(\Xi)))[\name{\Th}] \cong \mathrm{dom}(\al(\rho(\Xi[\name{\Th}]))) \cong \mathrm{dom}(\al(\rho(\Th)))  \]
  so $\name{A}$ is exactly such a map.

  In the case of $\termps{\P}{\al}{\be}$, we have again inductively $\Th$, $V_\Th$, $\Xi$, and $\name{\Th}$, and also $\al:\cInst(\P)\to\D$ over \C and $\be:\cInst(\P)\to\D_1$ over $\mathrm{dom}\circ \al$, and (as the additional data extending $\Th$ to a $\T_!$-instantiation of $\termps{\P}{\al}{\be}_!$) a section $a$ of $\be(\rho(\Th))$.
  Let $\pi_a : V_{(\Th,a)} \to V_\Th$ to be the dependent exponential
  \[ {\al(\rho(\Xi))}_*(\be(\rho(\Xi))) \]
  which exists by~\eqref{eq:lf}.
  It comes with a universal evaluation map $V_{(\Th,a)} \times_{V_\Th} \al(\rho(\Xi)) \to \be(\rho(\Xi))$ over $\mathrm{dom}(\al(\rho(\Xi))) = \mathbf{cod}(\be(\rho(\Xi)))$, i.e.\ a section of $\be(\rho(\Xi))[\pi_a]$, which is exactly what we need to extend $\Xi[\pi_a]$ to an instantiation of $\termps{\P}{\al}{\be}_!$ over $V_{(\Th,a)}$.
  Finally, the universal property of $V_{(\Th,a)}$ says that $\name{(\Th,a)} : \Gamma \to V_{(\Th,a)}$ can be determined by a map $\al(\rho(\Xi))[\name{\Th}] \to \be(\rho(\Xi))$, but $\al(\rho(\Xi))[\name{\Th}] \cong \al(\rho(\Xi[\name{\Th}])) \cong \al(\rho(\Th))$ and similarly for $\be(\rho(\Xi))$, so $a$ is itself such a map.
  
  This completes the definition of $V_\Th$, $\Xi\in\cInst(\P_!)_{V_\Th}$, and $\name{\Th} : \Gamma\to V_\Th$ such that $\Th = \Xi[\name{\Th}]$.
  Now let $V_H = V_\Th$ and $\name{H} = \name{\Th}$, and let $E_H\in\T(V_H)$ be a weakly stable typal initial $\S(\Xi)$-algebra, with $\S(\Xi)$-structure $\sH$.
  Then $(V_H,E_H,\name{H})\in\T_!(\Gamma)$ is the specified typal initial $\S(\Th)$-algebra in our intended strictly stable structure for $\T_!$.
  As usual, this is strictly stable by composition of names.

  For the elimination and computation rules, given $C\in \T_!(\Gamma.H)$ determined by $(V_C,E_C,\name{C})$, consider first the object $\Vtil = V_H\triangleleft V_C$ constructed as in~\cite{lw:localuniv}, which classifies maps $\si : \Delta \to V_H$ together with maps $E_H[\si] \to V_C$.
  In particular, we have universal maps $v:\Vtil\to V_H$ and $c:\Vtil.E_H[v]\to V_C$ giving $E_C[c]\in \T(\Vtil.E_H[v])$.

  Now since \S has representable lifts, we have a map $V_{\sH[v],E_C[c]} \to \Vtil$ representing lifts of the \S-structure $\sH[v]$ of $E_H[v]$ to $E_C[c]$.
  Combining universal properties, we see that maps $\Gamma\to V_{\sH[v],E_C[c]}$ correspond to choices of maps $\name{H}:\Gamma\to V_H$ and $\name{C}:\Gamma.E_H[\name{H}]\to V_C$ together with a lift of $\sH[\name{H}]$ to $E_C[\name{C}]$.
  Thus, we can we take $V_{\sH[V],E_C[c]}$ as the local universe for the elimination and computation rules.
  The weak stability of $H$ implies that the universal lift over $V_{\sH[V],E_C[c]}$ has a section that is an \S-morphism, which we can then reindex to $\Gamma$ to obtain a strictly stable eliminator.
\end{proof}

Now we move on to monads with parameters.
For any category \C with pullbacks and $\Gamma\in\C$, let $\cMnd(\C)_\Gamma$ denote the category of fibred monads on $\C/\Gamma$.
A fibred monad can be restricted along the projection $\C/\Delta\to\C/\Gamma$ for any $\sigma:\Delta\to\Gamma$, so these categories assemble into a fibration $\cMnd(\C) \to \C$.

\begin{defn}
  If $(\C,\T)$ is a comprehension category with pullbacks, a \textbf{fibred monad with parameters} on \C is a parameter scheme $\P$ together with a cartesian morphism $\dT:\cInst(\P) \to \cMnd(\C)$ over \C.
\end{defn}

Thus, for any instantiation $\Th$ over $\Gm$, we have an indexed monad $\dT(\Th)$ on $\C/\Gm$, whose fiber over $1_\Gm$ is an ordinary monad $\dT(\Th)_\Gm$ on $\C/\Gm$.
We write $\dT\alg_{\Th}$ for the category of algebras of $\dT(\Th)_\Gm$, each of which has an underlying object of $\C/\Gamma$, i.e.\ of $\C^\to$ over $\Gm$.
These categories vary functorially with respect to cartesian morphisms of $\Th$ (in particular, isomorphisms and reindexing along morphisms in \C), so we have a fibration $\dT\alg \to \C$ with a projection $\dT\alg \to \cInst(\P)\times_\C \C^\to$ over \C.
The following are easy to verify.

\begin{lem}
  For any fibred monad with parameters \dT on a comprehension category \C with pullbacks, $\dT\alg$ is a fibred category of structures with parameters over \C (with the same parameter scheme).\qed
\end{lem}

\begin{lem}
  If \dT is a fibred monad with parameters on a model category \sM, then the full subcategory $\dT\algf$ of $\dT\alg$ consisting of \dT-algebra structures on fibrations over fibrant objects is a fibred category of structures with parameters over \Mf.\qed
\end{lem}

\begin{defn}\label{defn:cell-param}
  If \sM is an excellent model category, then a \textbf{(finite) cell monad with parameters} on \sM is a fibred monad with parameters such that each fibred monad $\dT(\Th)$ is a (finite) cell monad on $\sM/\Gamma$.
\end{defn}

\begin{thm}\label{thm:cell-wkstab-param}
  If \dT is a cell monad with parameters on an excellent model category \sM, then \fibmf has weakly stable typal initial $\dT\algf$-algebras.
\end{thm}
\begin{proof}
  Note that for any $\sigma:\Delta\to\Gamma$ and any instantiation $\Th$ over $\Gm$, the monad $\dT(\Th[\si])_\De$ (the value of the fibred monad $\dT(\Th[\si])$ on $\C/\Delta$ at $1_\Delta$) coincides with the monad $\dT(\Th)_{\si}$ (the value of the fibred monad $\dT(\Th)$ on $\C/\Gamma$ at $\si$).
  This is because $\dT: \cInst(\P) \to \cMnd(\C)$ preserves cartesian arrows, and definition of the reindexing of a fibred monad and the fact that $(\C/\Gm)/(\si:\De\to\Gm) \simeq \C/\De$.
  Therefore, we can essentially repeat the proof of \cref{thm:cell-wkstab} once for each instantiation.
\end{proof}

\begin{thm}\label{thm:cell-replift-param}
  If \dT is a cell monad with parameters on an excellent model category \sM, then $\dT\algf$ has representable lifts.
\end{thm}
\begin{proof}
  Inspecting \cref{defn:replift,defn:replift-param}, we see that \dT has representable lifts just when each fibred monad $\dT(\Th)$ on $\C/\Gamma$ does.
  Thus, this follows directly from \cref{thm:cell-replift}.
\end{proof}

\begin{cor}\label{thm:cellmndparam}
  If \sM is an excellent model category and \dT is a cell monad with parameters on it, then $\fibmfbang$ has strictly stable typal initial $\dT\algf$-algebras.\qed
\end{cor}

The proofs of \cref{thm:cell-wkstab-param,thm:cell-replift-param} may seem almost too slick, as if something is being swept under the rug.
In particular, note that in \cref{defn:cell-param} we did not assume any compatibility between the \emph{cell structures} of the monads $\dT(\Th)$!
However, we did assume that the resulting monads fit together into a cartesian morphism from instantiations to \emph{fibred} monads, which implies that \emph{however} the parameters are being used in the construction of the monads $\dT(\Th)$ \emph{must} be stable under reindexing, and this is all that is required for \cref{thm:cell-wkstab-param,thm:cell-replift-param}.
In fact, the internals of $\cInst(\P)$ are only needed to prove the local universes coherence theorem, \cref{thm:lu-str-param}; for \cref{thm:cell-wkstab-param,thm:cell-replift-param}, $\cInst(\P)$ could just as well be \emph{any} fibration over $\C$.

Although \cref{thm:cellmndparam} is very general, unlike our results in \crefrange{sec:pushouts-type-theory}{sec:prop-trunc} it requires more work to extract from \cref{thm:cellmndparam} a statement about models of type theory.
Specifically, in order to apply it to any particular higher inductive type $\iH$, we have to do the following.
\begin{enumerate}
\item Manually translate the parameters of $\iH$ to a parameter scheme (\cref{defn:param-scheme}).
  This is practically algorithmic along the lines discussed above: for each parameter, $\al$ corresponds to the extended context of that parameter, and in the case of a term parameter, $\be$ corresponds to the type of that parameter.
  It's just that we don't have a general theorem that does the work for us.
\item Manually construct a cell monad \dT (with parameters) from the constructors of $\iH$.
  Although it may not be obvious from the examples we have seen so far, this is also pretty close to algorithmic, although there are certain limits on the constructors we can handle.
  We will explain in more detail below.\label{item:hit-mnd}
\item Manually generalize the notions of weakly and strictly stable typal initial $\dT\algf$-algebras to refer to stable classes of identity types, dependent identity types, and identity applications rather than the simplicial structure of $\sM$, and prove an analogue of \cref{thm:transfer-depid}.
  This is also practically algorithmic: transpose the left homotopies to right homotopies, replace the simplicial path objects by the type-theoretic structures, and transfer by composing with the equivalences $h$ and $k$ from \cref{thm:transfer-depid}.
\end{enumerate}
Ideally, of course, we would like a general correspondence between syntax and semantics that does all this work for us.
This would entail giving a precise \emph{syntactic} definition of what a ``higher inductive type'' is, which is already a significant task.\footnote{The approach of~\cite{acdf:qiits} is promising, but inapplicable as it stands to homotopy type theory, due to the unsolved problem of defining $(\infty,1)$-categories and functors internally in type theory.}
We leave it for future work.

The construction of a cell monad from a list of higher inductive constructors (item~\ref{item:hit-mnd}) goes roughly as follows.
Each constructor of $\iH$ corresponds to a \emph{cell} of the monad \dT.
The domain of that constructor, which we require to be ``strictly positive'' in $\iH$, determines the polynomial functor $F_{\W A B}$ in the same way as for ordinary inductive types.
(It may be possible to generalize further, for instance by allowing identity types of $\iH$ to appear in such domains, but we ignore that possibility for now.)
The ``dimension'' of the constructor determines the simplicial cofibration $K\to L$: ordinary point-constructors use the cofibration $\emptyset\to 1$, 1-dimensional path-constructors use $\mathbf{2}\to\ivl$, 2-path constructors use the ``inclusion of the boundary of a 2-globe'', and so on.
Finally, in the case of path-constructors, the source and target of the path determine the ``attaching map'' of the cell.
This is all made clearer by a discussion of examples.

\begin{eg} \label{eg:ordinary-inductives}
  Ordinary inductive types are the special case of higher inductive types that have no path-constructors.
  In this case, each cell is of the form $\emptyset \to \dT_{\W A B}$, so instead of a ``cell complex'' we are really just building a \emph{coproduct} of the monads $\dT_{\W A B}$ corresponding to each constructor.
  In the case of only one constructor, we just get $\dT_{\W A B}$ itself, as we used in \cref{sec:w-types}.

  An example with two constructors is the natural numbers, for which the constructor domains yield the maps $\emptyset\to 1$ and $1\to 1$ respectively.
  The former yields the monad for pointed objects, the latter the monad for objects with an endomorphism; thus their algebraic coproduct is the monad for pointed objects with an endomorphism, the same one we used in \cref{sec:natural-numbers}.
\end{eg}

\begin{eg}\label{eg:glob-cx}
  Any cell complex construction of a simplicial set $K$ using ``globular'' cells can be translated into a cell monad for a higher inductive version of $K$.
  For instance, the circle has a 0-dimensional cell and a 1-dimensional cell:
  \[
  \begin{tikzcd}
    \emptyset \ar[rrr] &&& 1 \ar[rrr] &&& S^1 \\
    & \emptyset \ar[ul] \ar[r] \ar[urr,phantom,near end,"\llcorner"] & 1 \ar[ur] &
    & \mathbf{2} \ar[ul] \ar[r] \ar[urr,phantom,near end,"\llcorner"] & \ivl \ar[ur]
  \end{tikzcd}
  \]
  If we take simplicial copowers of this diagram with the constant endofunctor at $1\in \sM$, which is $F_{\W 1 \emptyset}$, we get a ``cell endofunctor''; and then taking free monads we obtain a cell monad.
  The corresponding higher inductive type is the usual $S^1$.
  Similarly, we can construct the sphere $S^2$, and so on.
  In such cases, of course, the technology of cell monads is overkill; we can simply take the simplicial copower $1\otimes K$ and fibrantly replace it, arguing as in \crefrange{sec:pushouts}{sec:pushouts-type-theory}.
\end{eg}

\begin{eg} \label{eg:non-rec}
  More general non-recursive higher inductive types are constructed in a similar way with nontrivial endofunctors.
  For instance, given $f_i : A \to B_i$ as in \crefrange{sec:pushouts}{sec:pushouts-type-theory}, we have a cell endofunctor
  \[ \mathclap{
  \begin{tikzcd}[ampersand replacement=\&]
    \emptyset \ar[rrr] \&\&\& F^{\emptyset\to(B_1+B_2)} \ar[rrr] \&\&\& F \\
    \& \emptyset \ar[ul] \ar[r] \ar[urr,phantom,near end,"\llcorner"] \& F^{\emptyset \to (B_1+B_2)} \ar[ur] \&
    \& F^{\emptyset\to A} \otimes \mathbf{2} \ar[ul,"{[\inl\circ f_1,\inr\circ f_2]}"] \ar[r] \ar[urr,phantom,near end,"\llcorner"] \& F^{\emptyset\to A} \otimes \ivl \ar[ur]
  \end{tikzcd}
  } \]
  Note that the endofunctor $F^{\emptyset\to B_1+B_2}$ is constant at $B_1+B_2$, while $F^{\emptyset\to A}\otimes \mathbf{2}$ is constant at $A\otimes \mathbf{2} \cong A+A$.
  Thus the map marked $[\inl\circ f_1,\inr\circ f_2]$ really is just such a copairing.
  Taking free monads then yields a cell monad whose higher inductive type is the pushout of $f_1$ and $f_2$.

  This construction applies to a \emph{particular} $A,B_i,f_i$ in some slice $\sM/\Gamma$; in the general case we construct a parameter scheme as follows.
  First there are three type parameters for which $\al:\cInst(\P)\to\D$ is fiberwise constant at $1_\Gamma$ (i.e.\ $A,B_1,B_2$ are all in the same arbitrary context $\Gamma$).
  This yields a parameter scheme $\P_3$ whose instantiations consist of three types $A,B_1,B_2$ in the same context.
  Then we have two term parameters for which $\al(A,B_1,B_2) = (\Gamma.A\to \Gamma)$ and for which $\be(A,B_1,B_2)$ is $\Gamma.A.B_1[A]\to \Gamma.A$ and $\Gamma.A.B_2[A]\to \Gamma.A$ respectively.
  An instantiation of the resulting parameter scheme thus consists of three types $A,B_1,B_2$ and morphisms $A\to B_1$ and $A\to B_2$ over $\Gamma$.

  We then need to verify that this construction is stable, up to isomorphism, under reindexing of these data when performed in arbitrary slices.
  This is always obvious, so subsequently we omit to mention it.
\end{eg}

\begin{eg} \label{eg:prop-trunc}
  The propositional truncation (\cref{sec:prop-trunc}) expressed as a higher inductive type has two constructors:
  \begin{itemize}
  \item $\tr:A \to\brck{A}$
  \item $\treq : \prod_{x,y:\brck A} \id[\brck A]{x}{y}$
  \end{itemize}
  The only parameter is $A$, so the parameter scheme has one type parameter with $\al$ fiberwise constant at $1_\Gamma$.
  Again we have a cell monad with two cells, but now it no longer arises from a cell endofunctor.
  \[
  \begin{tikzcd}
    \emptyset \ar[rrr] &&& \dT_1 \ar[rrr] &&& \dT_2 \\
    & \emptyset \ar[ul] \ar[r] \ar[urr,phantom,near end,"\llcorner"] & F^{\emptyset\to A} \ar[ur] &
    & F^{2\to 1} \otimes \mathbf{2} \ar[ul] \ar[r] \ar[urr,phantom,near end,"\llcorner"] & F^{2\to 1} \otimes \ivl \ar[ur]
  \end{tikzcd}
  \]
  By definition we have $F^{2\to 1}(X) = X\times X$, while it is easy to verify that $\dT_1(X) = X+A$ (the free monad on the constant endofunctor at $A$, whose algebras are objects under $A$).
  The attaching map
  \[F^{2\to 1}(X)\otimes \mathbf{2} \cong (X\times X) + (X\times X) \to \dT_0(X) = X+A\]
  consists of the first and second product projections, injected into $X$.
  This is the semantic version of the syntax $\prod_{x,y:\brck A} \id[\brck A]{x}{y}$, which says that we have two copies of the type being defined in the domain ($X\times X$) and the source and target of the path ($\id[\brck A]{x}{y}$) come from the first and the second copies respectively.
  Note that this map does not land in $F^{\emptyset\to A}$, only in the free monad it generates.
\end{eg}

In general, when describing the attaching maps of a cell monad it is useful to know that monad morphisms $\dS\to\dT$ are equivalent to functors $\dT\alg\to\dS\alg$ over \sM, i.e.\ natural ways to assign an \dS-algebra structure to any \dT-algebra.
In particular, if \dS is algebraically-free on an endofunctor $F$, then such a morphism is determined by a natural way to assign an $F$-endofunctor-algebra structure to any \dT-algebra.

\begin{eg} \label{eg:james}
  The \textbf{James construction} $JA$ of a type $A$ equipped with a point $\star : A$, as defined in~\cite{brunerie:thesis}, has three constructors:
  \begin{itemize}
  \item $\varepsilon : JA$
  \item $\alpha : A \to JA \to JA$
  \item $\delta : \prod_{x:JA} \id[JA]{x}{\alpha(\star,x)}$
  \end{itemize}
  The parameters are $A$ and $\star$, so our parameter scheme has one type parameter with $\al$ fiberwise constant at $1_\Gamma$, and one term parameter with $\al$ again constant at $1_\Gamma$ while $\be(A) = (\Gamma.A\to \Gamma)$.

  The first cell is $\emptyset \to F^{0\to 1}$, yielding a monad whose algebras are pointed objects.
  The second cell is $\emptyset \to F^{A\to A}$, yielding a monad whose algebras are objects with an ``$A$-indexed family of endomorphisms'' $A\times X\to X$.

  The third cell is $F^{1\to 1}\otimes \mathbf{2}\to F^{1\to 1}\otimes \ivl$.
  Since $F^{1\to 1}(X) = X$, an endofunctor-algebra for $F^{1\to 1}\otimes \mathbf{2}$ is an object equipped with two endomorphisms.
  The nontrivial attaching map $F^{1\to 1}\otimes \mathbf{2} \to \dT_2$ therefore consists of a natural way to assign two endomorphisms to a pointed object equipped with a map $\alpha:A\times X\to X$.
  The syntax tells us that first endomorphism is the identity $1_X$, while the second is $\alpha \circ (\star,1_X)$.
  We leave it to the reader to translate from typal initial algebras to dependent identity types as in \cref{thm:transfer-depid}.
\end{eg}

\begin{eg} \label{eg:localization}
  Given a type $I$, two types $i:I\types S_i\type$ and $i:I\types T_i\type$, and a family of functions $i:I,s:S_i \types f_i(s) : T_i$, corresponding to a map $f:S\to T$ in $\sM/I$, an object $X\in \sM$ is \textbf{internally $f$-local} if the induced precomposition map $(I^*X)^T \to (I^*X)^S$ in $\sM/I$ is an equivalence.
  The \textbf{$f$-localization} of a type $A$ can be defined as a higher inductive type with the following five constructors:
  \begin{itemize}
  \item $\mathsf{loc}:A\to L_f A$
  \item $\mathsf{ext}:\prod_{i:I} \prod_{g:S_i \to L_f A} T_i \to L_f A$
  \item $\mathsf{ext}':\prod_{i:I} \prod_{g:S_i \to L_f A} T_i \to L_f A$
  \item $\mathsf{rtr}:\prod_{i:I} \prod_{g:S_i \to L_f A} \prod_{s:S_i} \id[L_f A]{g(s)}{\mathsf{ext}(i,g,f_i(s))}$
  \item $\mathsf{rtr}':\prod_{i:I} \prod_{h:T_i \to L_f A} \prod_{t:T_i} \id[L_f A]{h(s)}{\mathsf{ext}'(i,h\circ f_i,t)}$
  \end{itemize}
  The first constructor $\mathsf{loc}$ equips $L_f A$ with a map from $A$, while the other four ensure that $L_f A$ is internally $f$-local, being essentially a reorganization of the ``bi-invertibility'' definition of equivalence from~\cite[Chapter 4]{hottbook}: $\mathsf{ext}$ and $\mathsf{ext}'$ give two maps in the opposite direction, while $\mathsf{rtr}$ and $\mathsf{rtr}'$ ensure that they are a section and a retraction of the precomposition map.
  (In~\cite{rss:modalities} localizations are constructed instead using pushouts and an auxilary higher inductive type $\mathcal{J}_f A$, which is just $L_f A$ with $\mathsf{ext}'$ and $\mathsf{rtr}'$ omitted.)

  The parameter scheme for $L_f A$ has one type parameter with $\al$ constant at $1_\Gamma$, two more type parameters with $\al(I) = (\Gamma.I\to\Gamma)$, and one term parameter with $\al(I,S,T) = (\Gamma.I.S\to \Gamma.I\to\Gamma)$ and $\be(I,S,T) = (\Gamma.I.S.T[S] \to \Gamma.I.S)$.
  The first cell is $\emptyset \to F^{\emptyset\to A}$.
  The second and third cells are both $\emptyset \to F^{T^*S \to T}$, where $T^*S$ denotes the pullback of $S\to I$ along $T\to I$.
  Finally, the fourth and fifth cells are $F^{S^*S\to S} \otimes (\mathbf{2}\to \ivl)$ and $F^{T^*T\to T} \otimes (\mathbf{2}\to \ivl)$, while the attaching maps can be obtained by interpreting the terms $g(s)$, ${\mathsf{ext}(i,g,f_i(s))}$, $h(s)$, and ${\mathsf{ext}'(i,h\circ f_i,t)}$ in the internal \emph{extensional} type theory of the locally cartesian closed category \sM.
\end{eg}

Hopefully these examples are convincing of the generality of our construction and the relative straightforwardness of its application to individual examples.
It does, however, have the following notable limitation: because the ``attaching map'' $F^{f,K}\to \dT_n$ must be a map of fibred monads, it must exist \emph{before} we take the monad coproduct with $\dR$, and therefore in defining it we cannot assume that $\dT_n(X)$ is fibrant.
Thus, the source and target terms of a higher constructor cannot use operations that only exist in fibrant types, such as path-concatenation and eliminators of other inductive types.

For instance, in the torus constructor $\mathsf{sq} : \id[{\id[T^2]{\mathsf{base}}{\mathsf{base}}}]{\mathsf{left}\ct\mathsf{right}}{\mathsf{right}\ct\mathsf{left}}$, the source and target terms involve path-concatenation, so we cannot deal with this constructor directly.
Reducing it to a 1-constructor using the ``hub-and-spoke'' method~\cite[\sec6.7]{hottbook} does not help, since the ``rim'' of the wheel, and hence the outer boundaries of the spokes, are defined using the eliminator of $S^1$, which is also only defined when the target is a (fibrant) type.

However, in this case (and many others) there is a different workaround.
As the cofibration $K\to L$ for the monad cell corresponding to $\mathsf{sq}$, we choose the inclusion $\partial(\ivl\times\ivl) \to \ivl\times\ivl$ of the boundary of the simplicial square.
The non-fibrant monad $\dT_2$ generated by the first three constructors does have simplicial paths $\mathsf{left}$ and $\mathsf{right}$, which we can glue together in the appropriate ways to get a map from $\partial(\ivl\times\ivl)$.
Then our new monad cell glues in a simplicial square with the appropriate boundary, and when we then fibrantly replace it we can deduce from this a ``globular'' 2-cell relating the two boundary composites.
We leave the details to the interested reader.

In particular, any simplicial set $K$ with finitely many nondegenerate simplices can be constructed as a finite complex of ``simplicial cells'' involving the boundary cofibrations $\partial\Delta^n \to \Delta^n$.
Thus, this method yields a higher inductive type representing $K$; although as in \cref{eg:glob-cx} such a simple example can be obtained more easily as a fibrant replacement of $1\otimes K$.

The general workaround applies also to many other examples, but is not universally applicable.
In particular, it does not apply to constructor~\ref{item:f4} from \cref{sec:blass}, which uses \N-induction to define the source and target!
However, this problem goes away when the 1-categorical natural numbers object happens to be fibrant, as in simplicial sets or sets, since then the natural numbers type is just the natural numbers object and so can eliminate into arbitrary objects.
So at least in such cases our construction does apply to yield the \iF that is unachievable within ZF.

\section{Summary}
\label{sec:summary}

Through the course of the paper, we have accumulated a large amount of machinery.
For reference, therefore, we give here a self-contained summary of what we have obtained, and what it yields specifically for the semantics of higher inductive types.

\begin{thm}
  Let $\sM$ be an excellent model category: that is, simplicial, combinatorial, right proper, simplicially locally cartesian closed, and with all monos cofibrations, and cofibrations stable under limits.
  Let $\dT$ be a cell monad with parameters on $\sM$.
  Then:
  \begin{enumerate}
  \item The comprehension category \fibmf of fibrant objects of \sM and fibrations of \sM has weakly stable typal initial $\dT\algf$-algebras (\cref{thm:cell-wkstab-param}).
  \item $\dT\algf$ has representable lifts (\cref{thm:cell-replift-param}), and \Mf satisfies~\eqref{eq:lf}.
  \item Hence by \cref{thm:lu-str-param}, its local universe splitting $\fibmfbang$ has strictly stable typal initial $\dT\algf$-algebras (\cref{thm:cellmndparam}).
  \end{enumerate}
\end{thm}

From this general result along with \cref{eg:non-rec,eg:prop-trunc,eg:james,eg:localization}, and/or by the warm-up cases treated individually in \cref{sec:coproducts,sec:pushouts,sec:coherence-pushouts,sec:pushouts-type-theory,sec:natural-numbers,sec:w-types,sec:prop-trunc}, we obtain:
\begin{thm}
  Suppose $\sM$ an excellent model category.  
  Then the comprehension category \fibmf has weakly stable structure, and its local universes splitting $\fibmfbang$ has strictly stable structure, for the following type formers:
  \begin{enumerate}
  \item coproduct types;
  \item pushout types;
  \item a natural numbers type;
  \item $\W$-types;
  \item propositional truncations;
  \item James constructions;
  \item localizations;
  \item a torus type.
  \end{enumerate}
\end{thm}


Moreover, we expect cell monads with parameters to suffice for obtaining many other higher inductive types, under the limitation that the constructors do not involve ``fibrant structure'', i.e.\ eliminations into the higher inductive type itself.

Finally, we recall that the ``excellent'' model categories to which these results apply include sets and simplicial sets (with their standard model structures); and more generally all locally presentable locally cartesian closed 1-categories with a trivial model structure, and all right proper, simplicially locally cartesian closed, simplicial Cisinski model categories, which suffice to present all locally presentable locally cartesian closed $(\infty,1)$-categories, including all Grothendieck $(\infty,1)$-toposes.
 
\bibliographystyle{alphamod}
\bibliography{all}

\newcommand{\etalchar}[1]{$^{#1}$}
\begin{thebibliography}{{HoT}15}

\bibitem[ACDF16]{acdf:qiits}
Thorsten Altenkirch, Paolo Capriotti, Gabe Dijkstra, and Fredrik~Nordvall
  Forsberg.
\newblock Quotient inductive-inductive types.
\newblock {\em CoRR}, 2016.
\newblock \href {http://arxiv.org/abs/1612.02346} {\path{arXiv:1612.02346}}.
\newblock URL: \url{http://arxiv.org/abs/1612.02346}.

\bibitem[AGS15]{ags:it-hott}
Steve Awodey, Nicola Gambino, and Kristina Sojakova.
\newblock Homotopy-initial algebras in type theory.
\newblock Extended abstract in LICS'12, 2015.
\newblock \href {http://arxiv.org/abs/1504.05531} {\path{arXiv:1504.05531}}.

\bibitem[AK16]{ak:tt-qit}
Thorsten Altenkirch and Ambrus Kaposi.
\newblock Type theory in type theory using quotient inductive types.
\newblock {\em ACM SIGPLAN Notices}, 51(1):18--29, 2016.

\bibitem[AW09]{aw:htpy-idtype}
Steve Awodey and Michael~A. Warren.
\newblock Homotopy theoretic models of identity types.
\newblock {\em Math. Proc. Camb. Phil. Soc.}, 146(45):45--55, 2009.
\newblock \href {http://arxiv.org/abs/0709.0248} {\path{arXiv:0709.0248}}.

\bibitem[BGL{\etalchar{+}}17]{bglsss:hottcoq}
Andrej Bauer, Jason Gross, Peter~LeFanu Lumsdaine, Michael Shulman, Matthieu
  Sozeau, and Bas Spitters.
\newblock The {HoTT} {Library}: A formalization of homotopy type theory in
  {Coq}.
\newblock {\em Certified Programs and Proofs}, 2017.
\newblock \href {http://arxiv.org/abs/1610.04591} {\path{arXiv:1610.04591}}.

\bibitem[Bla83]{blass:freealg}
Andreas Blass.
\newblock Words, free algebras, and coequalizers.
\newblock {\em Fundamenta Mathematicae}, 117(2):117--160, 1983.
\newblock URL: \url{http://eudml.org/doc/211359}.

\bibitem[Bru16]{brunerie:thesis}
Guillaume Brunerie.
\newblock {\em On the homotopy groups of spheres in homotopy type theory}.
\newblock PhD thesis, Universit\'e de Nice, 2016.
\newblock arXiv:1606.05916.

\bibitem[Car86]{cartmell:gatcc}
John Cartmell.
\newblock Generalised algebraic theories and contextual categories.
\newblock {\em Annals of Pure and Applied Logic}, 32(0):209 -- 243, 1986.

\bibitem[CHM18]{chm:cubical-hits}
Thierry Coquand, Simon Huber, and Anders M\"{o}rtberg.
\newblock On higher inductive types in cubical type theory.
\newblock In {\em Proceedings of the 33rd Annual ACM/IEEE Symposium on Logic in
  Computer Science}, LICS '18, pages 255--264, New York, NY, USA, 2018. ACM.
\newblock \href {http://dx.doi.org/10.1145/3209108.3209197}
  {\path{doi:10.1145/3209108.3209197}}.

\bibitem[Cis02]{cisinski:topos}
Denis-Charles Cisinski.
\newblock Th\'eories homotopiques dans les topos.
\newblock {\em J. Pure Appl. Algebra}, 174:43--82, 2002.

\bibitem[Cis06]{cisinski:presheaves}
Denis-Charles Cisinski.
\newblock {\em Les pr\'efaisceaux comme mod\`eles type d'homotopie}, volume 308
  of {\em Ast\'erisque}.
\newblock Soc. Math. France, 2006.

\bibitem[Cis12]{cisinski:lccc-rpcmc}
Denis-Charles Cisinski.
\newblock Blog comment on post ``{T}he mysterious nature of right properness''.
\newblock
  \url{http://golem.ph.utexas.edu/category/2012/05/the_mysterious_nature_of_right.html\#c041306},
  May 2012.

\bibitem[Cis14]{cisinski:elegant}
Denis-Charles Cisinski.
\newblock Univalent universes for elegant models of homotopy types.
\newblock 2014.
\newblock \href {http://arxiv.org/abs/1406.0058} {\path{arXiv:1406.0058}}.

\bibitem[Gar09]{garner:soa}
Richard Garner.
\newblock Understanding the small object argument.
\newblock {\em Appl. Categ. Structures}, 17(3):247--285, 2009.
\newblock \href {http://arxiv.org/abs/0712.0724} {\path{arXiv:0712.0724}}.

\bibitem[Git80]{gitik:unc-sing}
M.~Gitik.
\newblock All uncountable cardinals can be singular.
\newblock {\em Israel J. Math.}, 35(1-2):61--88, 1980.
\newblock \href {http://dx.doi.org/10.1007/BF02760939}
  {\path{doi:10.1007/BF02760939}}.

\bibitem[GK17]{gk:univlcc}
David Gepner and Joachim Kock.
\newblock Univalence in locally cartesian closed $\infty$-categories.
\newblock {\em Forum Math.}, 29(3):617--652, 2017.
\newblock \href {http://arxiv.org/abs/1208.1749} {\path{arXiv:1208.1749}}.

\bibitem[GT06]{gt:nwfs}
Marco Grandis and Walter Tholen.
\newblock Natural weak factorization systems.
\newblock {\em Arch. Math. (Brno)}, 42(4):397--408, 2006.

\bibitem[{HoT}13]{hottagda}
{HoTT Project}.
\newblock The homotopy type theory {Agda} library.
\newblock \url{http://github.com/HoTT/HoTT-Agda/}, 2013.

\bibitem[{HoT}15]{hottcoq}
{HoTT Project}.
\newblock The homotopy type theory {Coq} library.
\newblock \url{http://github.com/HoTT/HoTT/}, 2015.

\bibitem[HS98]{hs:gpd-typethy}
Martin Hofmann and Thomas Streicher.
\newblock The groupoid interpretation of type theory.
\newblock In {\em Twenty-five years of constructive type theory ({V}enice,
  1995)}, volume~36 of {\em Oxford Logic Guides}, pages 83--111. Oxford Univ.
  Press, New York, 1998.

\bibitem[Kel74]{kelly:doc-adjn}
G.~M. Kelly.
\newblock Doctrinal adjunction.
\newblock In {\em Category Seminar (Proc. Sem., Sydney, 1972/1973)}, pages
  257--280. Lecture Notes in Math., Vol. 420. Springer, Berlin, 1974.

\bibitem[Kel80]{kelly:transfinite}
G.~M. Kelly.
\newblock A unified treatment of transfinite constructions for free algebras,
  free monoids, colimits, associated sheaves, and so on.
\newblock {\em Bull. Austral. Math. Soc.}, 22(1):1--83, 1980.

\bibitem[KL12]{klv:ssetmodel}
Chris Kapulkin and Peter~LeFanu Lumsdaine.
\newblock The simplicial model of univalent foundations (after {Voevodsky}).
\newblock 2012.
\newblock \href {http://arxiv.org/abs/1211.2851} {\path{arXiv:1211.2851}}.

\bibitem[Kra16]{kraus:nonrec-hit}
Nicolai Kraus.
\newblock Constructions with non-recursive higher inductive types.
\newblock {\em LICS'16}, 2016.

\bibitem[Lur09]{lurie:higher-topoi}
Jacob Lurie.
\newblock {\em Higher topos theory}.
\newblock Number 170 in Annals of Mathematics Studies. Princeton University
  Press, 2009.
\newblock \href {http://arxiv.org/abs/math.CT/0608040}
  {\path{arXiv:math.CT/0608040}}.

\bibitem[LW15]{lw:localuniv}
Peter~LeFanu Lumsdaine and Michael~A. Warren.
\newblock The local universes model: An overlooked coherence construction for
  dependent type theories.
\newblock {\em ACM Trans. Comput. Logic}, 16(3):23:1--23:31, July 2015.
\newblock \href {http://arxiv.org/abs/1411.1736} {\path{arXiv:1411.1736}}.

\bibitem[{nLa}]{nlab:transfinite}
{nLab authors}.
\newblock Transfinite construction of free algebras.
\newblock
  \url{https://ncatlab.org/nlab/show/transfinite+construction+of+free+algebras}.

\bibitem[Rie11]{riehl:nwfs-model}
Emily Riehl.
\newblock Algebraic model structures.
\newblock {\em New York Journal of Mathematics}, 17:173--231, 2011.
\newblock \href {http://arxiv.org/abs/0910.2733} {\path{arXiv:0910.2733}}.

\bibitem[Rie13]{riehl:mon-ams}
Emily Riehl.
\newblock Monoidal algebraic model structures.
\newblock {\em Journal of Pure and Applied Algebra}, 217(6):1069--1104, 2013.
\newblock \href {http://dx.doi.org/https://doi.org/10.1016/j.jpaa.2012.09.029}
  {\path{doi:https://doi.org/10.1016/j.jpaa.2012.09.029}}.
\newblock URL:
  \url{http://www.sciencedirect.com/science/article/pii/S0022404912003052}.

\bibitem[Rij16]{rijke:omegaloc-talk}
Egbert Rijke.
\newblock Localization at maps between omega-compact types via a small object
  argument in {HoTT}.
\newblock
  \url{http://www.fields.utoronto.ca/talks/localization-maps-between-omega-compact-types-small-object-argument-hott},
  May 2016.

\bibitem[Rij17]{rijke:join}
Egbert Rijke.
\newblock The join construction.
\newblock 2017.
\newblock \href {http://arxiv.org/abs/1701.07538} {\path{arXiv:1701.07538}}.

\bibitem[RSS17]{rss:modalities}
Egbert Rijke, Michael Shulman, and Bas Spitters.
\newblock Modalities in homotopy type theory.
\newblock arXiv:1706.07526, 2017.

\bibitem[Shu14]{shulman:hiru-tdd}
Michael Shulman.
\newblock Higher inductive-recursive univalence and type-directed definitions.
\newblock \url{http://homotopytypetheory.org/2014/06/08/hiru-tdd/}, June 2014.

\bibitem[Shu15a]{shulman:elreedy}
Michael Shulman.
\newblock The univalence axiom for elegant {Reedy} presheaves.
\newblock {\em Homology, Homotopy, and Applications}, 17(2):81--106, 2015.
\newblock \href {http://arxiv.org/abs/1307.6248} {\path{arXiv:1307.6248}}.

\bibitem[Shu15b]{shulman:invdia}
Michael Shulman.
\newblock Univalence for inverse diagrams and homotopy canonicity.
\newblock {\em Mathematical Structures in Computer Science}, 25:1203--1277, 6
  2015.
\newblock \href {http://arxiv.org/abs/1203.3253} {\path{arXiv:1203.3253}},
  \href {http://dx.doi.org/10.1017/S0960129514000565}
  {\path{doi:10.1017/S0960129514000565}}.

\bibitem[Shu17]{shulman:eiuniv}
Michael Shulman.
\newblock Univalence for inverse {EI} diagrams.
\newblock {\em Homology, Homotopy and Applications}, 19(2):219--249, 2017.
\newblock arXiv: 1507.03634.

\bibitem[Shu19]{MO:enriched-cccs}
Mike Shulman\phantom{x}(mathoverflow.net/users/49).
\newblock Enriched cartesian closed categories.
\newblock MathOverflow, 2019.
\newblock URL: \url{http://mathoverflow.net/questions/322085}.

\bibitem[Str91]{streicher:semtt}
Thomas Streicher.
\newblock {\em Semantics of type theory: correctness, completeness, and
  independence results}.
\newblock Progress in Theoretical Computer Science. Birkh\"auser, 1991.

\bibitem[{Uni}13]{hottbook}
{Univalent Foundations Program}.
\newblock {\em Homotopy Type Theory: Univalent Foundations of Mathematics}.
\newblock \url{http://homotopytypetheory.org/book/}, first edition, 2013.

\bibitem[{van}16]{vandoorn:proptrunc-nonrec}
Floris {van Doorn}.
\newblock Constructing the propositional truncation using non-recursive {HITs}.
\newblock {\em Certified Programs and Proofs '16}, 2016.
\newblock \href {http://arxiv.org/abs/1512.02274} {\path{arXiv:1512.02274}}.

\bibitem[vdBM15]{vdbm:wtypes-hott}
Benno van~den Berg and Ieke Moerdijk.
\newblock W-types in homotopy type theory.
\newblock {\em Mathematical Structures in Computer Science}, 25:1100--1115, 6
  2015.
\newblock \href {http://dx.doi.org/10.1017/S0960129514000516}
  {\path{doi:10.1017/S0960129514000516}}.

\end{thebibliography}

\end{document}